\documentclass[fleqn,10pt,oneside]{article}

\usepackage{soul}
\usepackage{hyperref}
\usepackage[final]{graphicx}
\usepackage{bbm}  %
\usepackage{bm}  %
\usepackage{amsmath,amsthm,amssymb,amscd}
\usepackage{caption,subcaption}
\usepackage{booktabs,multirow}
\usepackage{setspace} 
\usepackage[top=1.30in, bottom=1.30in, left=1.0in, right=2.35in]{geometry}

\usepackage[final]{showlabels}
\usepackage{dashbox}
\usepackage{microtype}
\usepackage[medium]{titlesec}
\usepackage{srcltx}
\DeclareMathOperator*{\argmax}{arg\,max}

\newtheorem{theorem}{Theorem}[section]
\newtheorem{lemma}[theorem]{Lemma}

\newtheorem{definition1}{Definition}[section]
\newtheorem{observe}{Observation}[section]
\newtheorem{remark1}[observe]{Remark}
\newtheorem{example1}{Example}[section]
\newtheorem{aside1}[observe]{Aside}

\newenvironment{definition}[1][]{\begin{definition1}[#1] \rm}{\end{definition1}}
\newenvironment{observation}{\begin{observe} \rm}{\end{observe}}
\newenvironment{remark}{\begin{remark1} \rm}{\end{remark1}}

\def\qed{\hfill$\blacksquare$\\} \renewenvironment{proof}{\noindent {\bf 
Proof.}}{\qed}

\newif\ifshowboxes \showboxestrue

\providecommand{\e}[1]{\ensuremath{\times 10^{#1}}}

\newcommand{\overbar}[1]{\mkern 1.5mu\overline{\mkern-1.5mu#1\mkern-1.5mu}\mkern%
  1.5mu}

\newcommand{\mach}{u}

\newcommand{\li}{L^\infty}

\renewcommand{\d}{\,\mathrm{d}}

\newcommand{\norm}[1]{\ensuremath{ {\lVert #1 \rVert} }}
\newcommand{\bnorm}[1]{\ensuremath{ {\bigl\lVert #1 \bigr\rVert} }}

\newcommand{\snorm}[1]{\ensuremath{ {\left\| #1 \right\|} }}

\newcommand{\abs}[1]{\ensuremath{ {\lvert #1 \rvert} }}

\def\C{\mathbbm{C}}
\def\R{\mathbbm{R}}

\def\1{\mathbbm{1}}

\def\O{\mathcal{O}}

\usepackage{color} %
\usepackage{xcolor}
\usepackage{xspace}
\renewcommand{\hat}{\widehat}

\newcommand{\td}[1][]{\textcolor{red}{\fcolorbox{red}{pink}{\textbf{\textcolor{red}{\scalebox{0.7}[1.0]{\small TODO}}}}\ifthenelse{\equal{#1}{}}{}{~\emph{#1}}}\xspace}

\def\P{\mathcal{P}}

\begin{document}

\begin{center}
    \begin{minipage}[t]{5.7in}

    In this paper, we show that the monomial basis is generally as good
    as a well-conditioned polynomial basis for interpolation, provided that
    the condition number of the Vandermonde matrix is smaller than the
    reciprocal of machine epsilon.  This leads to a practical
    algorithm for piecewise polynomial interpolation over general regions in
    the complex plane using the monomial basis.  Our analysis also yields a
    new upper bound for the condition number of an arbitrary Vandermonde
    matrix, which generalizes several previous results.

 \vspace{ 0.15in}

 \thispagestyle{empty}

   \end{minipage}
 \end{center}
 
 \vspace{ 1.10in}
 \vspace{ 0.30in}
 
 \begin{center}
   \begin{minipage}[t]{4.4in}
     \begin{center}

\textbf{On polynomial interpolation in the monomial basis}
\\
   \vspace{ 0.30in}
 
 Zewen Shen$\mbox{}^{\dagger\, \diamond\, \star}$ and
 Kirill Serkh$\mbox{}^{\ddagger\, \diamond}$  \\
 v6, Dec 13, 2024
 
     \end{center}
   \vspace{ -50.0in}
   \end{minipage}
 \end{center}
 
 \vspace{ 1.05in}

 \vfill
 
 \noindent 
 $\mbox{}^{\diamond}$  This author's work was supported in part by the NSERC
 Discovery Grants RGPIN-2020-06022 and DGECR-2020-00356.
 \\

 \vspace{2mm}
 
 \noindent
 $\mbox{}^{\dagger}$ Dept.~of Computer Science, University of Toronto,
 Toronto, ON M5S 2E4\\
 \noindent
 $\mbox{}^{\ddagger}$ Dept.~of Math. and Computer Science, University of Toronto,
 Toronto, ON M5S 2E4 \\
 
 \vspace{2mm}
 \noindent 
 $\mbox{}^{\star}$  Corresponding author
 \\

 \vfill
 \eject
\tableofcontents

\section{Introduction} 
\label{sec:intro}

Function approximation has been a central topic in numerical analysis since
its inception.  One of the most effective methods for approximating a
function $F:[-1,1]\to\R$ is the use of an interpolating polynomial $P_N$ of
degree $N$ which satisfies $P_N(x_j)=F(x_j)$ for a set of $(N+1)$
interpolation points $\{x_j\}_{j=0,1,\dots,N}$.  In practice, the
interpolation points are typically chosen to be the Chebyshev points, and
the resulting interpolating polynomial, known as the Chebyshev interpolant,
is a nearly optimal approximation to $F$ in the space of polynomials of
degree at most $N$ \cite{nick}.  A common basis for representing the
interpolating polynomial $P_N$ is the Lagrange polynomial basis, and the
evaluation of $P_N$ in this basis can be done stably using the Barycentric
interpolation formula \cite{bary,higham3}.  Some other commonly used bases
are Newton polynomials, Chebyshev polynomials, and Legendre polynomials.
Alternatively, the monomial basis can be used to represent $P_N$, such that
$P_N(x)=\sum_{k=0}^N a_k x^k$ for some coefficients
$\{a_k\}_{k=0,1,\dots,N}$.  The computation of the monomial coefficient
vector $a^{(N)}:=(a_0,a_1,\dots,a_N)^T\in\R^{N+1}$ of the interpolating
polynomial~$P_N$ requires the solution of the linear system
$V^{(N)}a^{(N)}=f^{(N)}$, where 
  \begin{align}
V^{(N)}:=\begin{pmatrix}
1 & x_0 & x_0^2 & \cdots & x_0^N \\ 
1 & x_1 & x_1^2 & \cdots & x_1^N \\
\vdots & \vdots & \vdots & \ddots &\vdots \\
1 & x_N & x_N^2 & \cdots & x_N^N \\ 
\end{pmatrix}\in\R^{(N+1)\times(N+1)}
  \end{align}
is a Vandermonde matrix, and
$f^{(N)}:=\bigl(F(x_0),F(x_1),\dots,F(x_N)\bigr)^T\in\R^{N+1}$ is a vector
of the function values of $F$ at the $(N+1)$ interpolation points on the
interval $[-1,1]$.  It is well-known that, given any set of real
interpolation points within the unit interval, the condition number of
$V^{(N)}$ grows at least as fast as $\sqrt{2}(1+\sqrt{2})^{N-1}/\sqrt{N+1}$
\cite{beck}.  It follows that the accuracy of the numerical solution to this
linear system deteriorates rapidly as $N$ increases, and, as a result, this
algorithm for constructing $P_N$ is often considered unstable. But, is this
really the case?  Let $\{x_j\}_{j=0,1,\dots,N}$ be the set of $(N+1)$
Chebyshev points on the interval $[-1,1]$, and suppose that
$F(x)=\cos(2x+1)$. We solve the resulting Vandermonde system using LU
factorization with partial pivoting.  In Figure \ref{fig:111a}, we show the
polynomial interpolation error $\norm{F-P_N}_{L^\infty([-1,1])}$ and the
monomial approximation error $\norm{F-\hat P_N}_{L^\infty([-1,1])}$, where
$P_N$ is approximated using the Barycentric interpolation formula and 
$\hat P_N$ is evaluated using the computed monomial expansion.  The
Chebyshev interpolant constructed from the computed monomial expansion is,
surprisingly, as accurate as the Chebyshev interpolant evaluated using the
Barycentric interpolation formula (which is accurate up to machine
precision), despite the huge condition number of the Vandermonde matrix
reported in Figure \ref{fig:111b}. 

\begin{figure}[h]
    \centering
    \begin{subfigure}{0.49\textwidth}
      \centering
      \includegraphics[width=\textwidth]{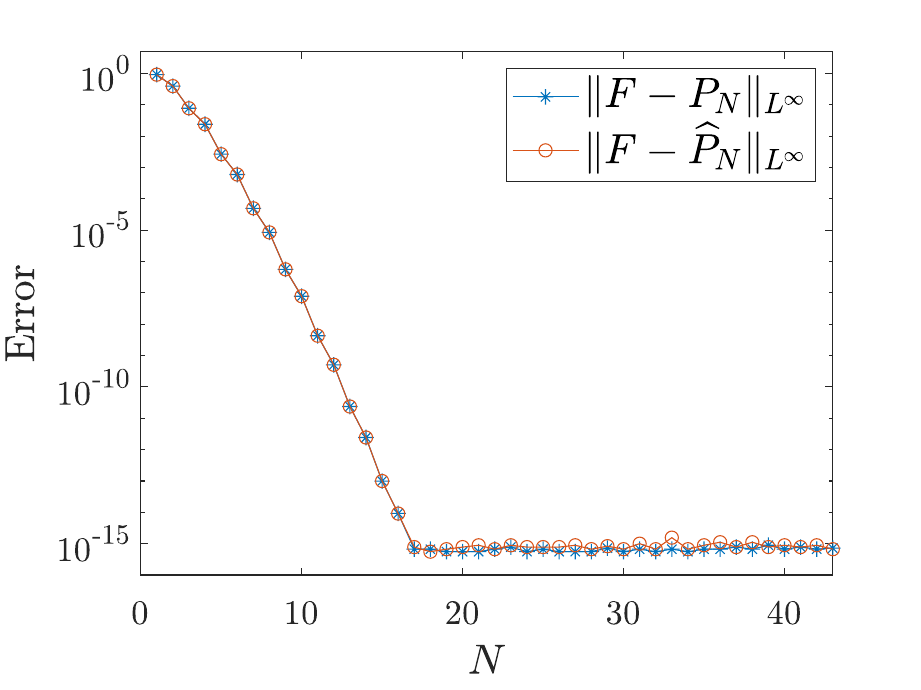}
      \caption{$F(x)=\cos(2x+1)$}
  \label{fig:111a}
    \end{subfigure}
\begin{subfigure}{0.49\textwidth}
      \centering
      \includegraphics[width=\textwidth]{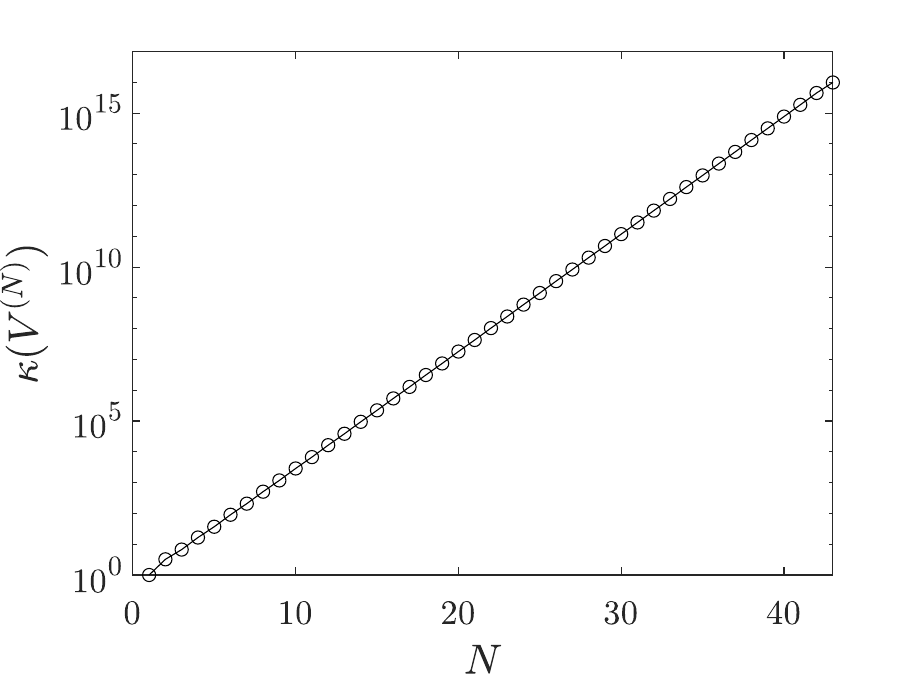}
      \caption{Condition number of $V^{(N)}$}
  \label{fig:111b}
    \end{subfigure}
      \caption{{\bf Polynomial interpolation of $\cos(2x+1)$ in the monomial
      basis over $[-1,1]$}. The $x$-axis label $N$ denotes the order of
      approximation. The polynomial $P_N$ denotes the interpolating polynomial
      approximated using the Barycentric interpolation formula. The polynomial
      $\hat P_N$ denotes the computed monomial expansion. The $L^\infty$
      error is estimated by comparing the approximated function values at
      $10000$ equidistant points over $[-1,1]$ with the true function
      values.}
      \label{fig:111}
\end{figure}

What happens when the function $F$ requires a higher-order polynomial for
accurate approximation?  In Figure \ref{fig:222}, we compare the accuracy of
the two approximations when $F(x)=\cos(8x+1)$ and when $F(x)=\cos(12x+1)$.
Initially, the computed monomial expansion is as accurate as the Chebyshev
interpolant evaluated using the Barycentric interpolation formula.  However,
the convergence of polynomial interpolation in the monomial basis stagnates
after reaching a certain error threshold.  It appears that, the higher the
polynomial order necessary to approximate the functions, the larger that
error threshold becomes. But consider the accuracy of the two approximations
when $F(x)=\frac{1}{x-\sqrt{2}}$ and when $F(x)=\frac{1}{x-0.5i}$, shown in
Figure \ref{fig:333}. These two functions each have a singularity in a
neighborhood of the interval $[-1,1]$, and Chebyshev interpolants of
degree~$\geq 40$ are required to approximate them to machine precision. Yet,
no stagnation of convergence is observed. In Figure \ref{fig:444}, we
consider the case where $F$ is a non-smooth function, and we find that the
accuracy of the two approximations is, once again, the same.   The stability
of polynomial interpolation in the monomial basis clearly depends in a
subtle way on the function being approximated.

\begin{figure}[h]
    \centering
  \begin{subfigure}{0.49\textwidth}
      \centering
      \includegraphics[width=\textwidth]{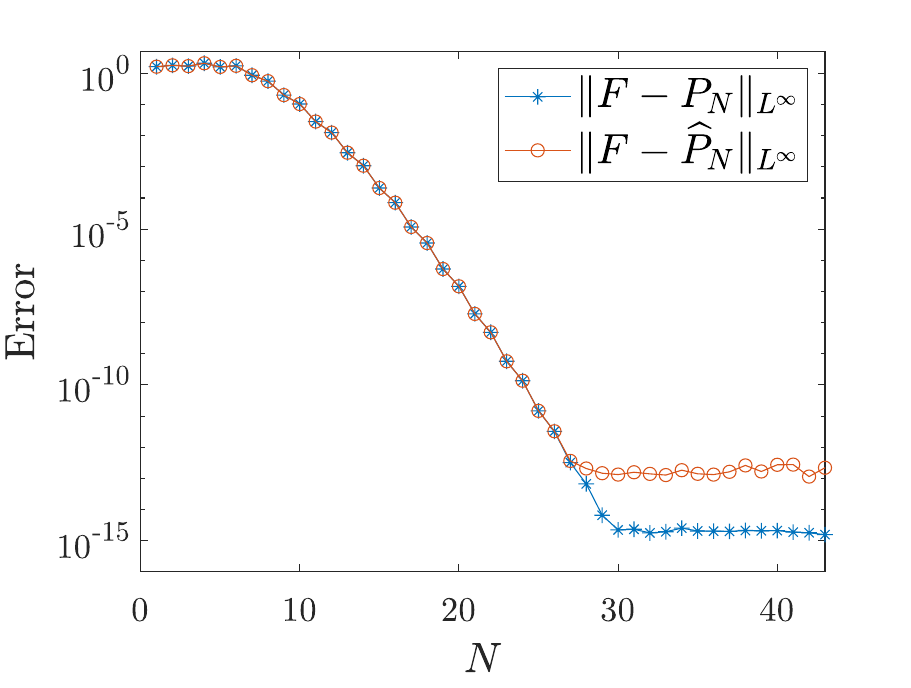}
      \caption{$F(x)=\cos(8x+1)$}
      \label{fig:222a}
    \end{subfigure}
  \begin{subfigure}{0.49\textwidth}
      \centering
      \includegraphics[width=\textwidth]{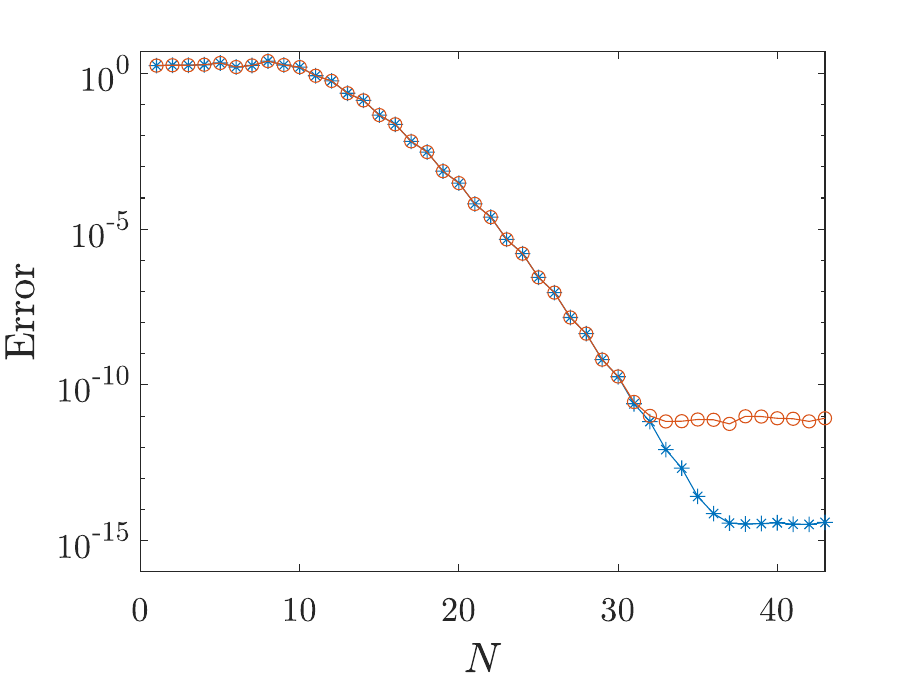}
      \caption{$F(x)=\cos(12x+1)$}
      \label{fig:222b}
    \end{subfigure}
      \caption{{\bf Polynomial interpolation of more complicated functions
      in the monomial basis, over $[-1,1]$}. }
      \label{fig:222}
\end{figure}

\begin{figure}[h]
    \centering
  \begin{subfigure}{0.49\textwidth}
      \centering
      \includegraphics[width=\textwidth]{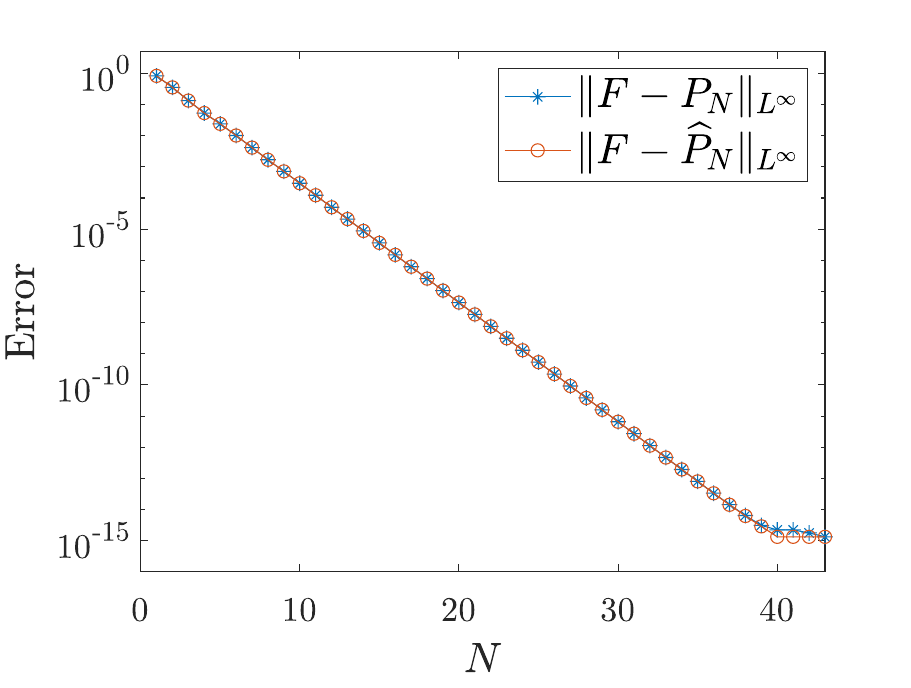}
      \caption{$F(x)=\frac{1}{x-\sqrt{2}}$}
\label{fig:333a}
    \end{subfigure}
  \begin{subfigure}{0.49\textwidth}
      \centering
      \includegraphics[width=\textwidth]{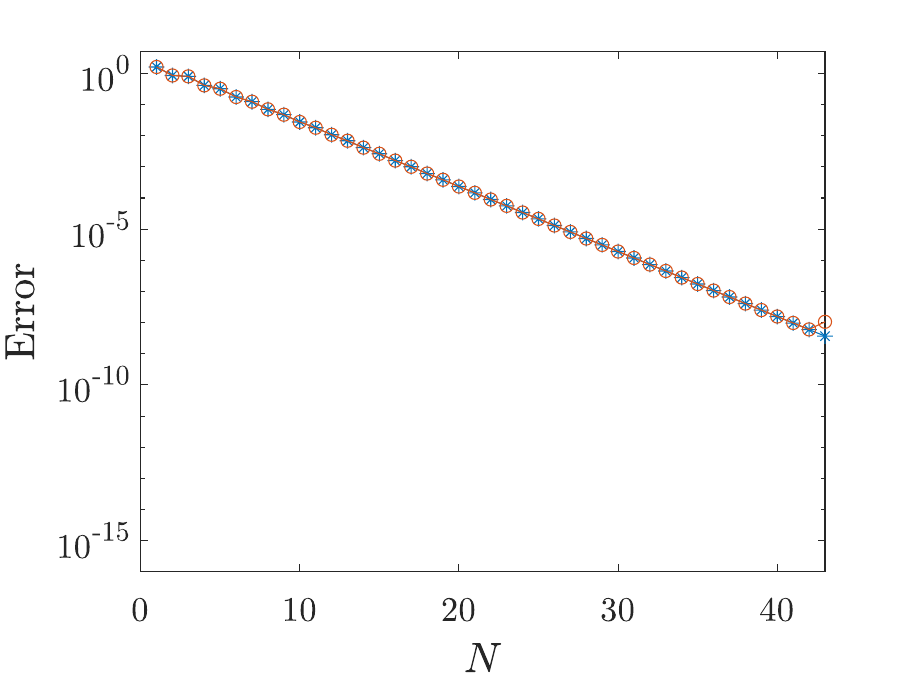}
      \caption{$F(x)=\frac{1}{x-0.5i}$}
\label{fig:333b}
    \end{subfigure}
      \caption{{\bf Polynomial interpolation of functions with a singularity
      near the approximation domain, in the monomial basis, over $[-1,1]$}.}
\label{fig:333}
\end{figure}

\begin{figure}[h]
    \centering
  \begin{subfigure}{0.49\textwidth}
      \centering
      \includegraphics[width=\textwidth]{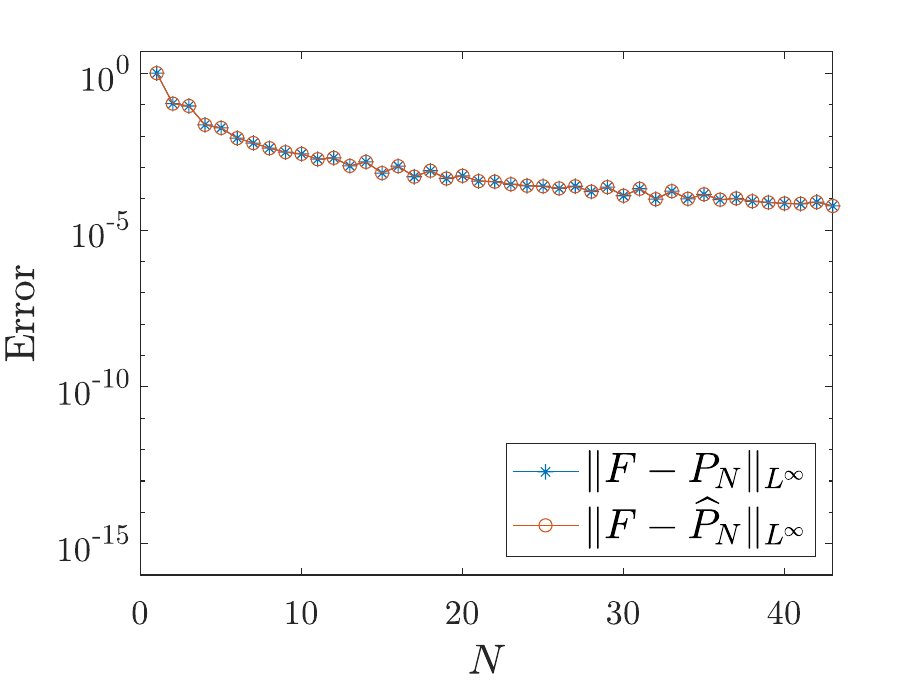}
      \caption{$F(x)=\abs{x+0.1}^{2.5}$}
\label{fig:444a}
    \end{subfigure}
  \begin{subfigure}{0.49\textwidth}
      \centering
      \includegraphics[width=\textwidth]{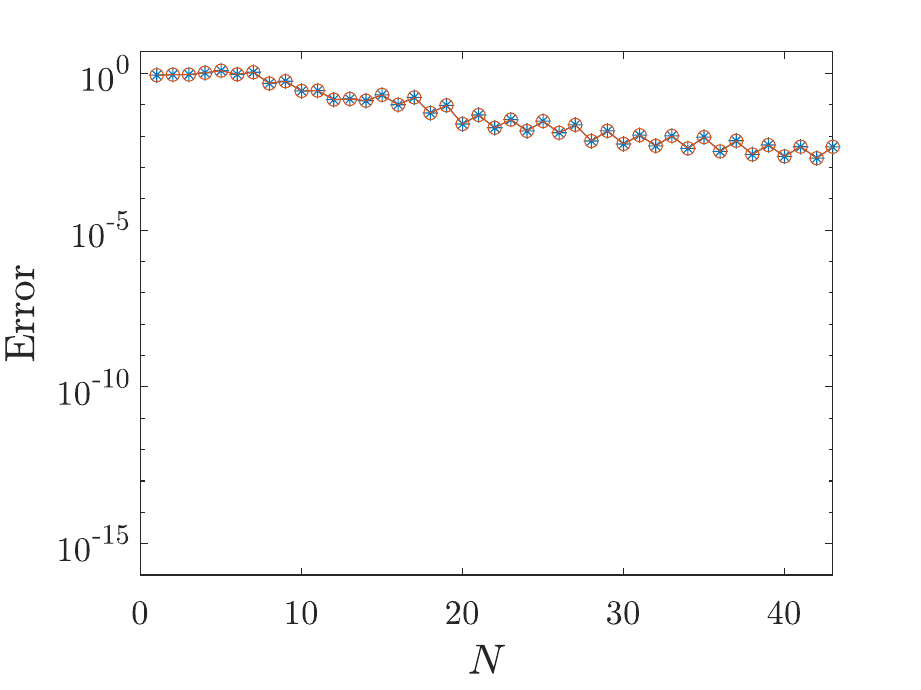}
      \caption{$F(x)=\abs{\sin(5x)}^3$}
\label{fig:444b}
    \end{subfigure}
      \caption{{\bf Polynomial interpolation of non-smooth functions in the
      monomial basis, over $[-1,1]$}. }
\label{fig:444}
\end{figure}

These seemingly mysterious experiments can be explained partially from the
point of view of backward error analysis. The forward error
$\norm{a^{(N)}-\hat a^{(N)}}_2$ of the numerical solution $\hat a^{(N)}$ to
the Vandermonde system $V^{(N)}a^{(N)}=f^{(N)}$ can indeed be huge, but it
is the backward error, i.e., $\norm{V^{(N)}\hat a^{(N)}-f^{(N)}}_2$, that
matters for the accuracy of the approximation.  This is because a backward
error of size around machine precision implies that the difference between the
computed monomial expansion, which we denote by $\hat P_N$, and the exact
interpolating polynomial, $P_N$, is a polynomial which approximately vanishes
at all of the interpolation points. More specifically, the residual error
$\norm{P_N-\hat P_N}_{\li([-1,1])}$ is bounded by
$\Lambda_N\cdot\norm{V^{(N)}\hat a^{(N)}-f^{(N)}}_2$, where $\Lambda_N$
denotes the Lebesgue constant associated with the interpolation points.
This constant is defined by the formula $\Lambda_N=\sup_{F\in
C([-1,1]), F\neq 0}\frac{\norm{P_N}_\infty}{\norm{F}_\infty}$, with
$\norm{\cdot}_\infty$ denoting the $\infty$-norm in $C([-1,1])$, and $P_N$
denoting the interpolating polynomial of $F$ for the given 
interpolation points.  This constant grows logarithmically when well-conditioned
interpolation points, such as Chebyshev points, are used.  As a result, we
can bound the monomial approximation error $\norm{F-\hat P_N}_{\li([-1,1])}$
using the following inequality:
  \begin{align}
  \hspace*{-2.5em}
\norm{F-\hat P_N}_{\li([-1,1])}\leq
&\,\norm{F-P_N}_{\li([-1,1])}+\norm{P_N-\hat P_N}_{\li([-1,1])}\notag \\
\leq &\, \norm{F-P_N}_{\li([-1,1])}+\Lambda_N\cdot \norm{V^{(N)}\hat
a^{(N)}-f^{(N)}}_2\notag \\
\lesssim &\, \norm{F-P_N}_{\li([-1,1])}+ \norm{V^{(N)}\hat
a^{(N)}-f^{(N)}}_2.
  \end{align}
Here, we introduce the notation $x_N \lesssim y_N$ to indicate the existence
of a constant $c_N > 0$, independent of the specific values of $x_N$ and
$y_N$ but dependent on $N$, such that $x_N \leq c_N y_N$. This constant
$c_N$ is assumed to be a rational function of $N$, where $N$ represents
either the dimensionality or the approximation order.  Additionally, we
write $x_N \approx y_N$ to mean that both $x_N \lesssim y_N$ and $y_N
\lesssim x_N$.  

When the
backward error is smaller than the polynomial interpolation error, the
monomial approximation error is dominated by the polynomial interpolation
error, and the use of a monomial basis does not incur any additional loss of
accuracy.  Once the polynomial interpolation error becomes smaller than the
backward error, the convergence of the approximation stagnates.  For
example, for the function approximated in Figure \ref{fig:333a}, we find
that the backward error is around the size of machine epsilon for all $N\leq
43$, which explains why no stagnation is observed. On the other hand, in
Figure \ref{fig:222a}, the backward error is around the size of $10^{-13}$
for $N\geq 20$, which leads to stagnation once the polynomial interpolation
error is less than $10^{-13}$.

The explanation above brings up a new question: when will the backward error
be small? Given a matrix $A\in\C^{n \times n}$ and a vector $b\in\C^n$, the
numerical solution $\hat x$ to the linear system $Ax=b$ computed by a
backward stable solver is the exact solution to the linear system
  \begin{align}
(A+\delta A)\hat x = b 
  \end{align}
for a perturbation matrix $\delta A\in\C^{n\times n}$ that satisfies
$\norm{\delta A}_2\leq  u\cdot \zeta_n \norm{A}_2$, where $u$ denotes
machine epsilon and $\zeta_n>0$ is a small growth factor that is independent of
$A$ and $b$ (see Section 7.6 in \cite{higham_text}).  A classical example of
a backward stable linear system solver is the Householder QR factorization,
whose growth factor $\zeta_n$ is bounded by a low-degree polynomial in
$n$ (see Theorem 19.5 and Equation (19.14) in \cite{higham_text}). Another
commonly used solver, LU factorization with partial pivoting, has a growth
factor $\zeta_n$ that can grow exponentially as $n$ increases (see Section
9.3 in \cite{higham_text}).  This upper bound for
$\norm{\delta A}_2$ is, however, rarely approached, and the
stability of LU factorization with partial pivoting is often comparable to
that of QR factorization in practice (see Lecture 22 in \cite{nick_linalg}
for a detailed discussion). When we state in this paper that a
linear system is solved using a backward stable solver, we assume that the
growth factor $\zeta_n$ associated with this solver is bounded by a
low-degree polynomial in $n$ with small coefficients.

When the Vandermonde system $V^{(N)}a^{(N)}=f^{(N)}$ is solved using a
backward stable solver, the backward error of the numerical solution to the
Vandermonde system satisfies $\norm{V^{(N)}\hat
a^{(N)}-f^{(N)}}_2=\norm{\delta V^{(N)} \hat a^{(N)}}_2\leq u\cdot \zeta_n
\norm{V^{(N)}}_2 \norm{\hat a^{(N)}}_2$.  Since $\norm{V^{(N)}}_2\in
[\sqrt{N+1}, N+1]$ when the interpolation points are inside the unit
interval $[-1,1]$, the backward error is essentially dominated by $u\cdot
\norm{\hat a^{(N)}}_2$. Furthermore, one can show that $\norm{\hat
a^{(N)}}_2\approx \norm{a^{(N)}}_2$ when 
the condition number of $V^{(N)}$ satisfies $\kappa(V^{(N)})\lesssim
\frac{1}{u}$ (see Theorem \ref{thm:mono_err} for a formal statement).
It follows that the monomial approximation error can be
quantified a priori using information about the interpolating polynomial,
which means that a theory of polynomial interpolation in the monomial basis
can be developed.  Although the examples provided in this section are
limited to the interval $[-1,1]$, it is important to note that these
observations are equally applicable to a simply connected compact
domain in the complex plane. In the rest of the paper, we consider
polynomial interpolation in the monomial basis over a simply connected
compact domain in the complex plane, with the interval as a special case.

The rest of the paper is organized as follows.  In Section \ref{sec:mono1d},
we analyze polynomial interpolation in the
monomial basis over a simply connected compact domain in the complex plane.
Our analysis reveals the conditions under which the monomial basis is as
good as a well-conditioned polynomial basis for interpolation, resulting in
a practical algorithm for using the monomial basis, with no extra error and
with almost no extra cost. As a by-product of our analysis, we derive a
tight upper bound for the condition number of any Vandermonde matrix (see
Theorem \ref{thm:cond2}). In Section~\ref{sec:exp}, we provide extensive
numerical experiments to support our analysis.  In Section~\ref{sec:app}, we
present applications where the use of a monomial basis for interpolation
offers a substantial advantage over other bases.   In Section \ref{sec:dis},
we summarize our key points, review related work, and discuss generalizations
of our results.

\section{Polynomial interpolation in the monomial basis}
  \label{sec:mono1d}

In this section, we consider polynomial interpolation of a function over a
simply connected compact set in the complex plane. We structure the section as follows.
First, we show in Theorem \ref{thm:mono_err} that, when
$\kappa(V^{(N)})\lesssim \frac{1}{u}$, the monomial approximation error is
bounded by the sum of the polynomial interpolation error and a extra error
term that involves the 2-norm of the monomial coefficient vector of the
interpolant.  Then, we provide tight upper bounds for this extra error term
and for the growth of the condition number of the Vandermonde matrix in
Theorems \ref{thm:coeffnorm} and \ref{thm:cond2}, respectively.  
In Section \ref{sec:good}, we elucidate the frequently observed and
seemingly paradoxical success of polynomial interpolation in the monomial
basis when $\kappa(V^{(N)})\lesssim \frac{1}{u}$. Finally, we summarize
the necessary considerations for the proper usage of the monomial basis in
Section \ref{sec:limit}.

Let $\Omega\subset \C$ be simply connected and compact, and let
$F:\Omega\to\C$ be an arbitrary function.  
The $N$th degree interpolating polynomial, denoted by $P_N$, 
of the function $F$ for a given set of $(N+1)$ distinct
interpolation points $Z:=\{z_j\}_{j=0,1,\dots,N}\subset \Omega$
can be expressed as $P_N(z)=\sum_{k=0}^N a_k z^k$, where
the monomial coefficient vector $(a_0,a_1,\dots,a_N)^T$ is the solution to
the Vandermonde system 
  \begin{align}
\begin{pmatrix}
1 & z_0 & z_0^2 & \cdots & z_0^N \\ 
1 & z_1 & z_1^2 & \cdots & z_1^N \\
\vdots & \vdots & \vdots & \ddots &\vdots \\
1 & z_N & z_N^2 & \cdots & z_N^N \\ 
\end{pmatrix}
\begin{pmatrix}
a_0 \\ 
a_1  \\
\vdots  \\
a_N \\
\end{pmatrix}=\begin{pmatrix}
F(z_0)\\ 
F(z_1)  \\
\vdots  \\
F(z_N) \\
\end{pmatrix}.
\label{for:vand}
  \end{align}
In this paper, we choose to fix the set of basis functions as
$\{z^k\}_{k=0,1,\dots,N}$.  
However, depending on the geometry and location of $\Omega$, it can be
advantageous to use a scaled and translated monomial basis 
$\{(\frac{z-\mu}{\nu})^k\}_{k=0,1,\dots,N}$.
Our analysis also encompasses this case, as the Vandermonde system
corresponding to the scaled and translated basis with the points
$\{z_j\}_{j=0,1,\dots,N} \subset \Omega$ is equivalent to that of the
original basis $\{z^k\}_{k=0,1,\dots,N}$ using the transformed points
$\{\frac{z_j-\mu}{\nu}\}_{j=0,1,\dots,N} \subset \frac{\Omega
- \mu}{\nu}$. We assume without loss of generality that
$\Omega$ lies within the unit disk, ensuring that $\norm{V^{(N)}}_2 \in
[\sqrt{N+1}, N+1]$.

In the following, we denote the Vandermonde matrix, the monomial coefficient
vector, and the right-hand side vector  by $V^{(N)}$, $a^{(N)}$,
and $f^{(N)}$, respectively.

In order to study the size of the residual of the numerical solution to the
Vandermonde system, we require the following lemma, which provides a bound
for the 2-norm of the solution to a perturbed linear system. 
\begin{lemma}
  \label{lem:small_residual}
Let $N$ be a positive integer. Suppose that $A\in \C^{N\times N}$ is
invertible, $b\in \C^N$, and that $x\in\C^N$ satisfies $Ax=b$. Suppose
further that $\hat x\in\C^N$ satisfies $(A+\delta A)\hat x=b$ for some
$\delta A\in \C^{N\times N}$.  If there exists an $\alpha >1$ such that
  \begin{align}
 \norm{\delta A}_2\leq \frac{1}{\alpha\cdot\norm{A^{-1}}_2},
\label{for:sr1}
  \end{align}
then the matrix $A+\delta A$ is invertible, and $\hat x$ satisfies
  \begin{align}
\frac{\alpha}{\alpha+1}\norm{x}_2\leq 
\norm{\hat x}_2 \leq \frac{\alpha}{\alpha-1}\norm{x}_2.
  \end{align}
\end{lemma}
\begin{proof}
By multiplying both sides of $(A+\delta A)\hat x=b$ by $A^{-1}$, we have
that
  \begin{align}
(I+A^{-1}\delta A)\hat x = x,
  \end{align}
where $I$ denotes the identity matrix. By (\ref{for:sr1}), the term
$A^{-1}\delta A$ satisfies
  \begin{align}
\norm{A^{-1}\delta A}_2\leq \norm{A^{-1}}_2\norm{\delta
A}_2\leq \frac{1}{\alpha}<1.
  \label{for:sr4}
  \end{align}
Thus, it follows that the matrix $A+\delta
A$ is invertible, and $\norm{\hat x}_2$ satisfies
  \begin{align}
\hspace*{-1.5em}
\norm{\hat x}_2 \leq \norm{(I+A^{-1}\delta A)^{-1}}_2 \norm{x}_2\leq
\frac{1}{1-\norm{A^{-1}\delta A}_2}\norm{x}_2\leq \frac{\alpha}{\alpha-1}\norm{x}_2.
  \label{for:sr5}
  \end{align}
In addition, by (\ref{for:sr4}), $\norm{x}_2$ satisfies
  \begin{align}
\norm{x}_2\leq \norm{I+A^{-1}\delta A}_2\norm{\hat x}_2 \leq
\Bigl(1+\frac{1}{\alpha}\Bigr)\norm{\hat x}_2.
  \label{for:sr6}
  \end{align}
The proof is completed by combining (\ref{for:sr5}) and (\ref{for:sr6}).
\end{proof}

The following theorem provides upper bounds for the monomial
approximation error. It is related to frame approximation \cite{frame,
frame2} and  the method of fundamental solutions
\cite{barnett2,stein} in the sense that it also uses a similar idea, that it
is possible to achieve a small approximation error with an ill-conditioned
basis, provided that there exists a small coefficient vector with a good
fit.
\begin{theorem}
  \label{thm:mono_err}
Suppose that there exists some constant $\gamma_N\geq 0$ such that
the computed monomial coefficient vector $\hat a^{(N)}=(\hat a_0,\hat
a_1,\dots,\hat a_N)^T$ satisfies
  \begin{align}
\bigl(V^{(N)}+\delta V^{(N)}\bigr)\hat a^{(N)} = f^{(N)} \label{for:mm0}
  \end{align}
for some $\delta V^{(N)}\in \C^{(N+1)\times (N+1)}$ with
  \begin{align}
\norm{\delta V^{(N)}}_2\leq \mach\cdot
\gamma_N,\label{for:mm1}
  \end{align}
where $u$ denotes machine epsilon.
Let $\hat P_N(z):=\sum_{k=0}^N \hat a_k z^k$ be the computed monomial
expansion. The monomial approximation error is bounded by
  \begin{align}
\norm{F-\hat P_N}_{L^\infty(\Omega)} \leq&\, \norm{F-P_N}_{L^\infty(\Omega)}
+ \mach\cdot \gamma_N\Lambda_N \norm{\hat a^{(N)}}_2,\label{for:mm7}
  \end{align}
where $\Lambda_N:=\sup_{F\in C(\Omega), F\neq
0}\frac{\norm{P_N}_{L^\infty(\Omega)}}{\norm{F}_{L^\infty(\Omega)}}$ is the
Lebesgue constant for $Z$.  If, in addition, 
  \begin{align}
\norm{(V^{(N)})^{-1}}_2\leq \frac{1}{2u\cdot \gamma_N},\label{for:mm8}
  \end{align}
then the 2-norm of the numerical solution $\hat a^{(N)}$ is bounded by
  \begin{align}
\frac{2}{3} \norm{a^{(N)}}_2\leq \norm{\hat a^{(N)}}_2\leq
2\norm{a^{(N)}}_2,\label{for:mm999}
  \end{align}
and the monomial approximation error can be quantified a priori by 
  \begin{align}
\hspace*{-0.0em}
\norm{F-\hat P_N}_{L^\infty(\Omega)}
\leq&\, \norm{F-P_N}_{L^\infty(\Omega)} + 2\mach\cdot \gamma_N\Lambda_N
\norm{a^{(N)}}_2.\label{for:priori1}
  \end{align}
\end{theorem}

\begin{proof}
By the triangle inequality, the definition of the Lebesgue
constant~$\Lambda_N$, equation~\eqref{for:mm0} and  inequality
\eqref{for:mm1}, the monomial approximation error satisfies
  \begin{align}
\norm{F-\hat P_N}_{L^\infty(\Omega)}\leq&\, 
\norm{F-P_N}_{L^\infty(\Omega)}
+\norm{\hat P_N-P_N}_{L^\infty(\Omega)}\notag\\
\leq &\,\norm{F-P_N}_{L^\infty(\Omega)}+\Lambda_N\norm{V^{(N)}\hat
a^{(N)}-f^{(N)}}_2 \notag\\
\leq&\, \norm{F-P_N}_{L^\infty(\Omega)}+\mach\cdot
\gamma_N\Lambda_N \norm{\hat a^{(N)}}_2.\label{for:mm11}
  \end{align}
If $\norm{(V^{(N)})^{-1}}_2\leq \frac{1}{2u\cdot \gamma_N}$, then by Lemma
\ref{lem:small_residual}, the 2-norm of the computed monomial
coefficient vector~$\hat a^{(N)}$ is bounded by
  \begin{align}
\frac{2}{3} \norm{a^{(N)}}_2\leq \norm{\hat a^{(N)}}_2\leq
2\norm{a^{(N)}}_2,
  \end{align}
and \eqref{for:mm11} becomes
  \begin{align}
\hspace*{-0.0em}
\norm{F-\hat P_N}_{L^\infty(\Omega)}
\leq&\, \norm{F-P_N}_{L^\infty(\Omega)} + 2\mach\cdot \gamma_N\Lambda_N
\norm{a^{(N)}}_2.
  \end{align}
\end{proof}

As described in the introduction, when the Vandermonde system is
solved using a backward stable linear system solver, the set of assumptions
\eqref{for:mm0} and \eqref{for:mm1} is satisfied with a constant
$\gamma_N\leq \zeta_N \norm{V^{(N)}}_2$, where $\zeta_N>0$ is a growth
factor bounded by a low-degree polynomial in $N$, and is independent of the
matrix and the right-hand-side vector. It follows then that $\gamma_N\lesssim 1$.
Therefore, given interpolation points with a Lebesgue constant
$\Lambda_N\lesssim 1$, we have that $\norm{F-\hat P_N}_{L^\infty(\Omega)}\lesssim
\norm{F-P_N}_{L^\infty(\Omega)} + \mach \norm{a^{(N)}}_2$ when
$\kappa(V^{(N)})\lesssim 1/u$, where we write the condition \eqref{for:mm8}
as $\kappa(V^{(N)}) = \norm{V^{(N)}}_2\cdot \norm{(V^{(N)})^{-1}}_2\leq
\frac{\norm{V^{(N)}}_2}{2u\cdot \gamma_N}\lesssim 1/u$.  In Figure
\ref{fig:1111}, we plot the values of $\norm{F-\hat
P_N}_{L^\infty(\Omega)}$, $\norm{F-P_N}_{L^\infty(\Omega)}$, and $\mach
\norm{a^{(N)}}_2$, for the functions appearing in Section~\ref{sec:intro},
in order to validate the theorem above.

\begin{figure}
    \centering

    \begin{subfigure}{0.49\textwidth}
      \centering
      \includegraphics[width=\textwidth]{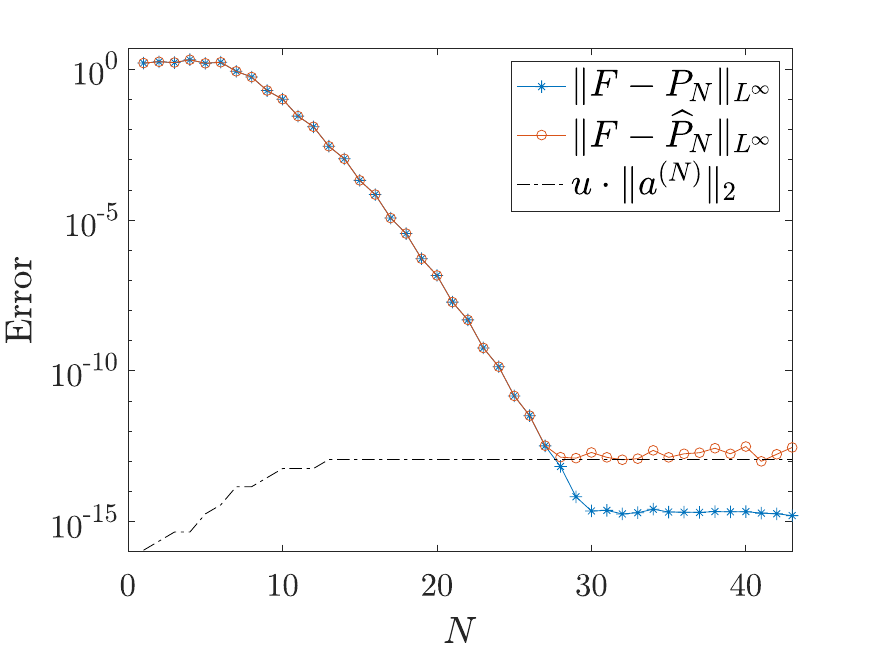}
      \caption{$F(x)=\cos(8x+1)$}
    \end{subfigure}
    \begin{subfigure}{0.49\textwidth}
      \centering
      \includegraphics[width=\textwidth]{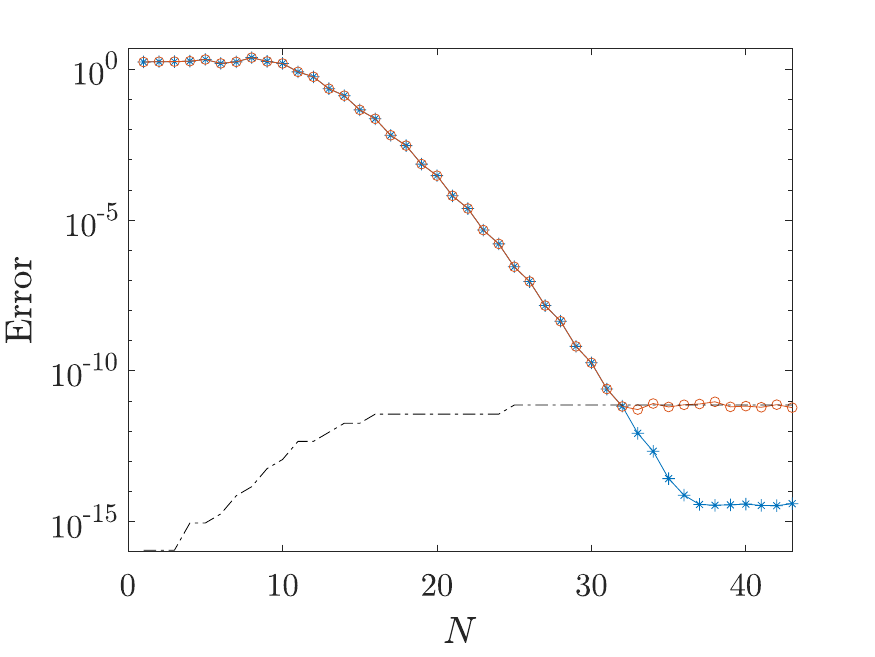}
      \caption{$F(x)=\cos(12x+1)$}
    \end{subfigure}
    \begin{subfigure}{0.49\textwidth}
      \centering
      \includegraphics[width=\textwidth]{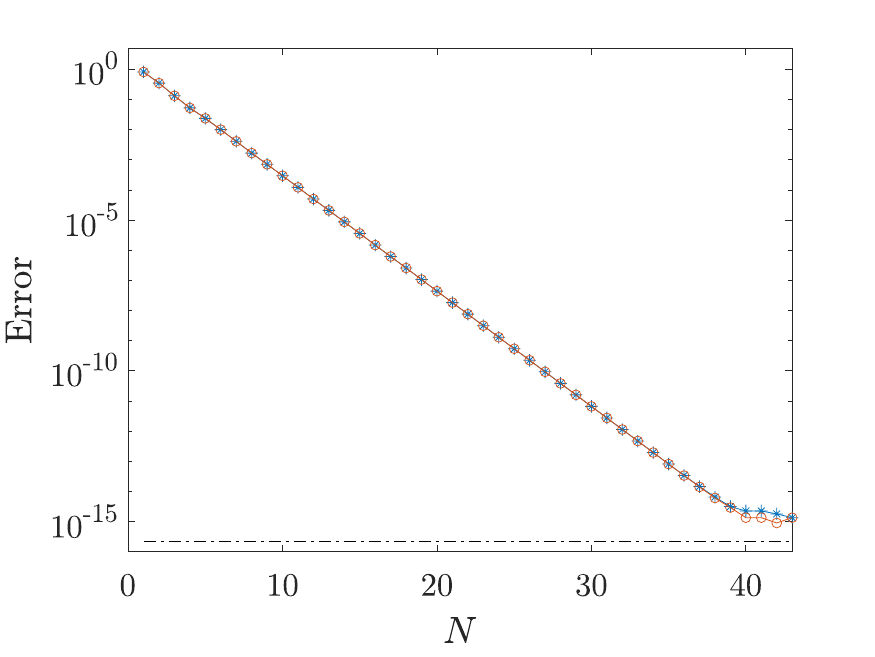}
      \caption{$F(x)=\frac{1}{x-\sqrt{2}}$}
    \end{subfigure}
    \begin{subfigure}{0.49\textwidth}
      \centering
      \includegraphics[width=\textwidth]{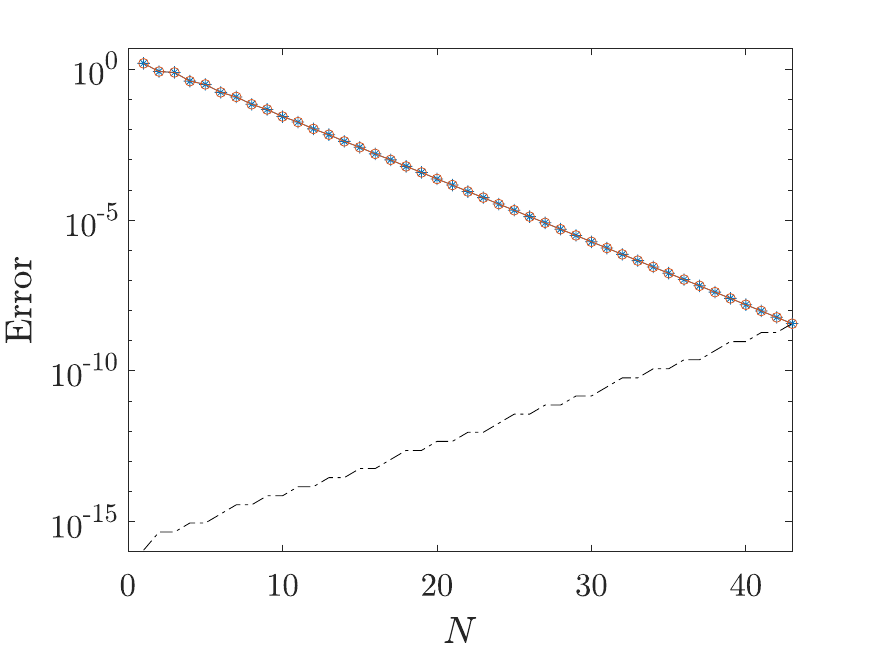}
      \caption{$F(x)=\frac{1}{x-0.5i}$}
    \end{subfigure}
    \begin{subfigure}{0.49\textwidth}
      \centering
      \includegraphics[width=\textwidth]{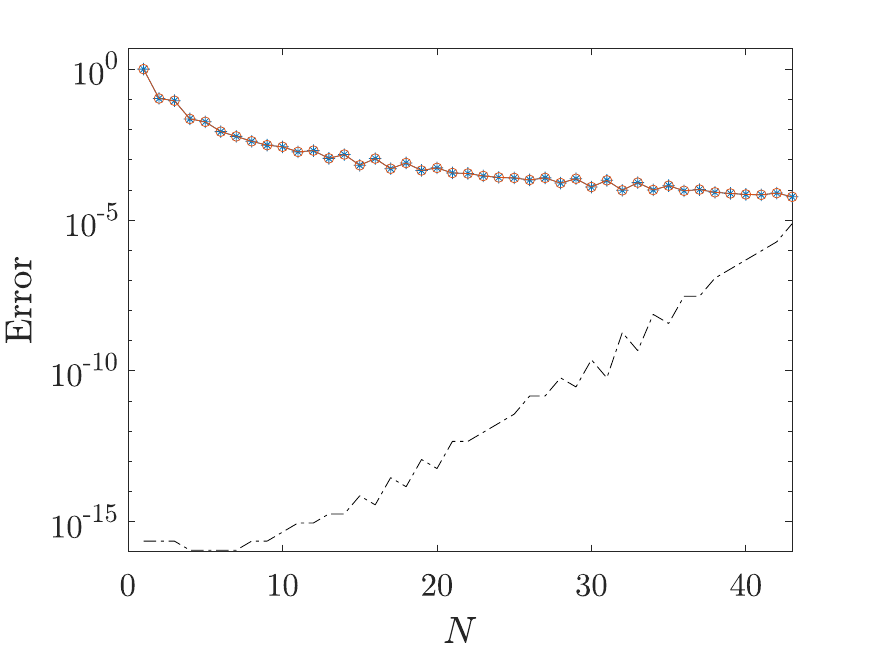}
      \caption{$F(x)=\abs{x+0.1}^{2.5}$}
    \end{subfigure}
    \begin{subfigure}{0.49\textwidth}
      \centering
      \includegraphics[width=\textwidth]{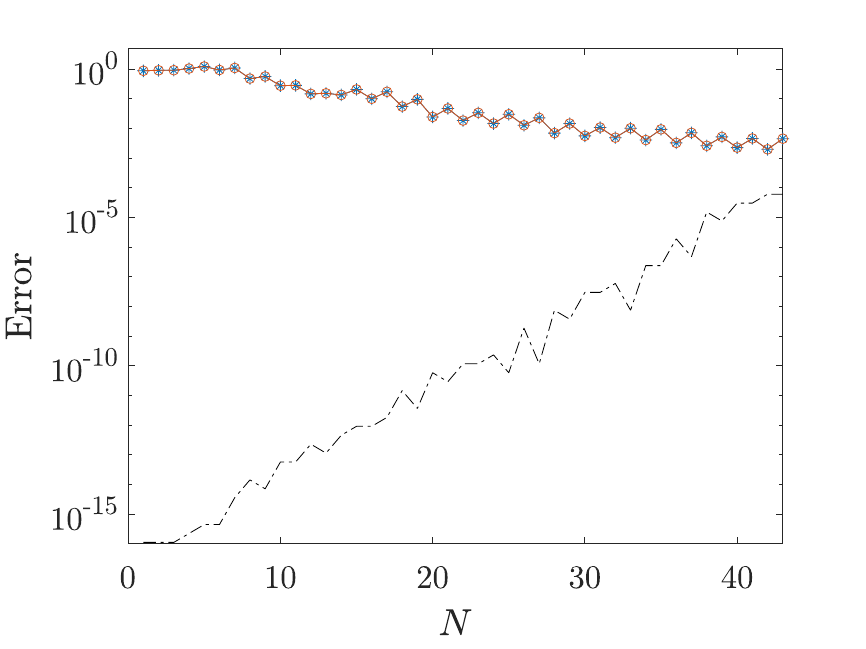}
      \caption{$F(x)=\abs{\sin(5x)}^3$}
    \end{subfigure}
      \caption{{\bf Polynomial interpolation error, monomial approximation
      error and~$u\cdot \norm{a^{(N)}}_2$, for $\Omega=[-1,1]$}. These
      functions are the ones that appear in Section~\ref{sec:intro}.     }
  \label{fig:1111}
\end{figure}

\begin{remark}
The second term on the right-hand side of \eqref{for:priori1} is
an upper bound of the backward error $\norm{P_N-\hat P_N}_{\li(\Omega)}$,
i.e., the extra loss of accuracy caused by the use of a monomial basis.
Note that the absolute condition number of the evaluation of $P_N(z)$ in
the monomial basis is $\lesssim\norm{a^{(N)}}_2$ when $\abs{z}\approx 1$,
so that the resulting error is~$\lesssim u\cdot \norm{a^{(N)}}_2$,
which is around the same size as $2\mach\cdot
\gamma_N\Lambda_N \norm{a^{(N)}}_2$.
\end{remark}

\begin{remark}
  \label{rem:leastsq_mono}
It is well-known that selecting a set of interpolation points with a small
Lebesgue constant is crucial for the interpolating polynomial $P_N$ to
approximate the function $F$ nearly as accurately as the optimal polynomial
approximant \cite{nick}.  This requirement on the Lebesgue constant
$\Lambda_N$ of the interpolation points, that $\Lambda_N\lesssim 1$, can be
circumvented by instead constructing an approximating polynomial using
least-squares fitting. To construct such an approximating polynomial, one
typically selects sample points on the boundary of the domain and uses a
polynomial basis whose dimension is smaller than the number of sample
points, as in, for example, \cite{nick2}.  Although our theoretical analysis
only addresses approximations constructed using interpolation, the analysis
for the least-squares case bears considerable similarity.  The principal
difference lies in the requirement for a more general notion of the Lebesgue
constant. One such idea, using the 2-norm of the pseudoinverse mapping
values at the sample points to the coefficients of an expansion in a basis,
appears in \cite{mhz}.
\end{remark}

In order to determine the relationship between the size of the monomial
coefficient vector and the properties of the function being approximated, we
will need the following definition, which generalizes the Bernstein ellipse
to the case of a simply connected compact set in the complex plane. 
\begin{definition}
  \label{def:genbern}
Given a simply connected compact set $\Omega$ in the complex plane and
$\rho>1$, we define $E_\rho$ to be the level set $\{x+iy\in\C:G(x,y)=\log
\rho\}$, where $G:\R^2\to\R$ is the unique solution to the exterior Laplace
equation
  \begin{align}
\hspace*{-0em}\nabla^2 G = &\,0 \text{ in } \R^2\setminus\Omega, \notag\\
\hspace*{-0em}G=&\,0 \text{ on } \partial\Omega,\notag\\
G(x,y)\sim&\, \log r \text{ as }  r:=\sqrt{x^2+y^2} \to\infty.\label{for:lap}
  \end{align}
We let $E_\rho^o$ denote the bounded open region with boundary
$E_\rho$. 
\end{definition}

\begin{remark}
  \label{rem:conform}
One can show that the function $G$ defined above is positive in
$\R^2\setminus\Omega$. Given a harmonic conjugate
$H:\R^2\setminus\Omega\to\R$ of $G$, the function $\Phi(x+iy):=e^{G(x,y)+i
H(x,y)}$ is a conformal mapping from $\C\setminus \Omega$ to $\C\setminus
\overbar D_1$, where $D_1$ is the open unit disk centered at the origin (see
Section 4.1 in \cite{walsh}). When $\partial\Omega$ is continuous,
$\Phi^{-1}$ can be extended continuously from $\C\setminus D_1$ to
$\overbar{\C \setminus \Omega}$ (see Section 2.1 in \cite{bbcm}).
Furthermore, when $\partial\Omega$ is a Jordan curve, $\Phi$ becomes a
continuous bijective map from $\C\setminus D_1$ to $\overbar{\C \setminus
\Omega}$.
\end{remark}

When $\Omega=[a,b]\subset \R$, the level set $E_\rho$ is a
Bernstein ellipse with parameter $\rho$, with foci at $a$ and $b$.
In Figure \ref{fig:erho}, we plot examples of level sets $E_\rho$ 
for various sets $\Omega$, for different values of $\rho$.

\begin{figure}
    \centering
    \begin{subfigure}{0.49\textwidth}
      \centering
      \includegraphics[width=\textwidth]{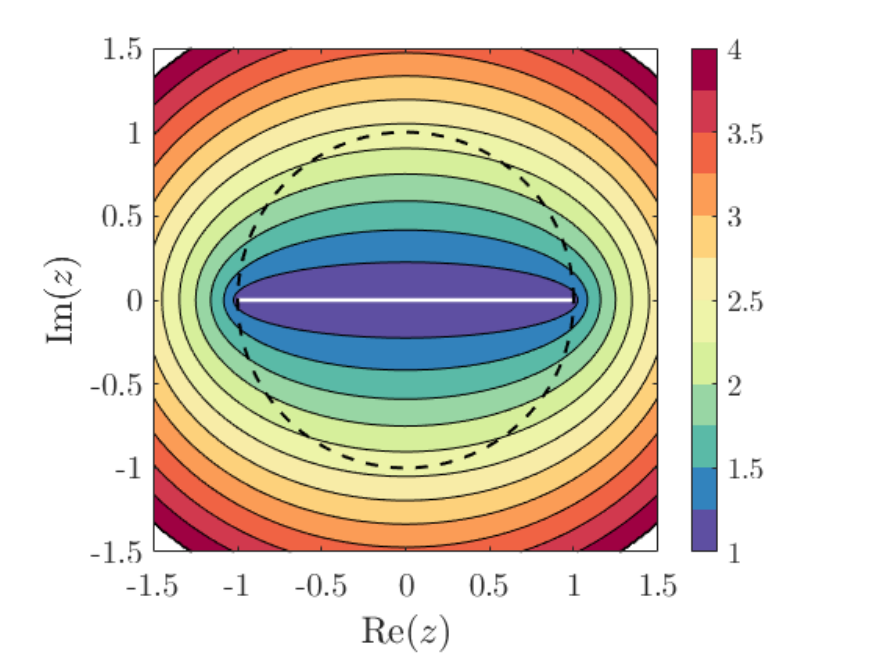}
      \caption{$\Omega=\left\{t:t\in[-1,1]\right\}$}
  \label{fig:erhoa}
    \end{subfigure}
  \begin{subfigure}{0.49\textwidth}
      \centering
      \includegraphics[width=\textwidth]{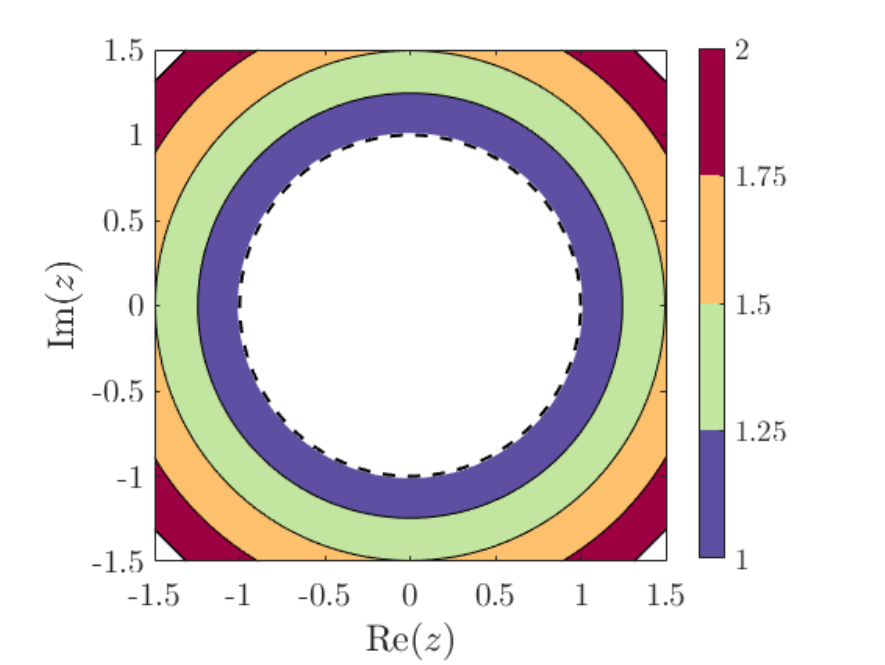}
      \caption{$\Omega=\left\{r e^{\pi it}:r\in[0,1], t\in[-1,1]\right\}$}
    \end{subfigure}
  \begin{subfigure}{0.49\textwidth}
      \centering
      \includegraphics[width=\textwidth]{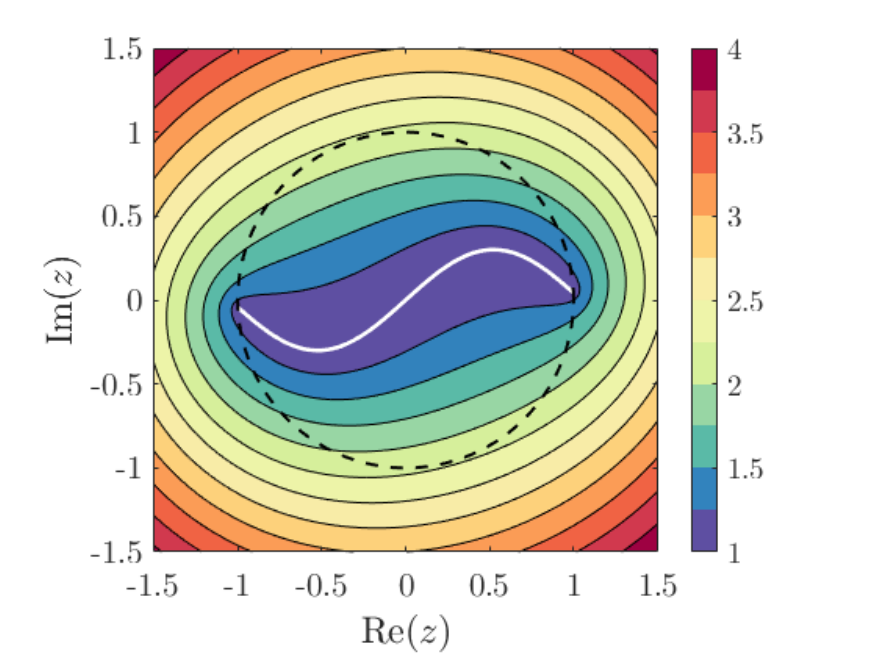}
      \caption{$\Omega=\left\{t+0.3i\sin(3t):t\in[-1,1]\right\}$}
    \end{subfigure}
  \begin{subfigure}{0.49\textwidth}
      \centering
      \includegraphics[width=\textwidth]{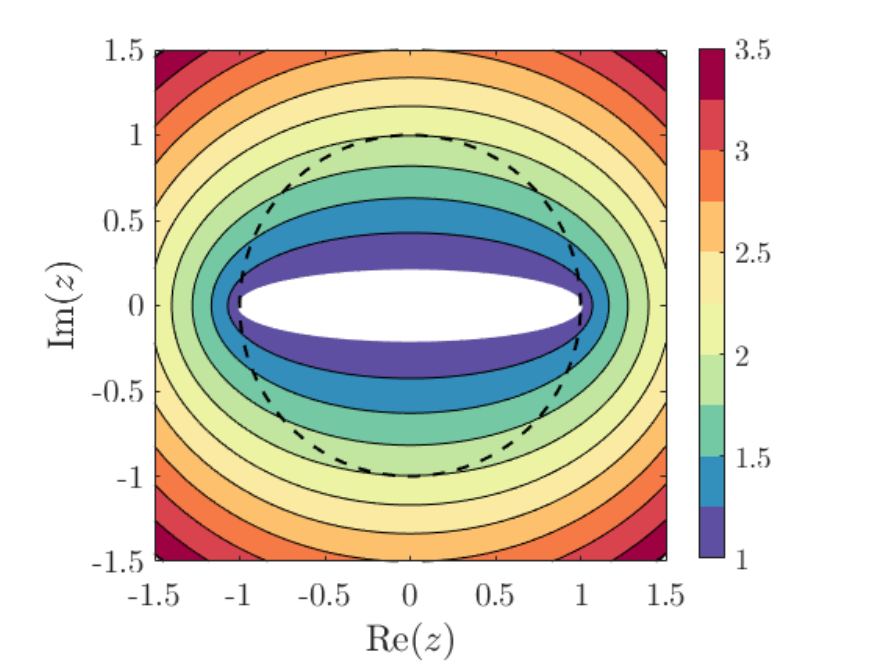}
      \caption{$\Omega=\left\{r\bigl(\cos(\pi t) + 0.2i \sin(\pi
      t)\bigr): r\in[0,1], t\in[-1,1]\right\}$}
    \label{fig:erho_elli}
    \end{subfigure}
  \begin{subfigure}{0.49\textwidth}
      \centering
      \includegraphics[width=\textwidth]{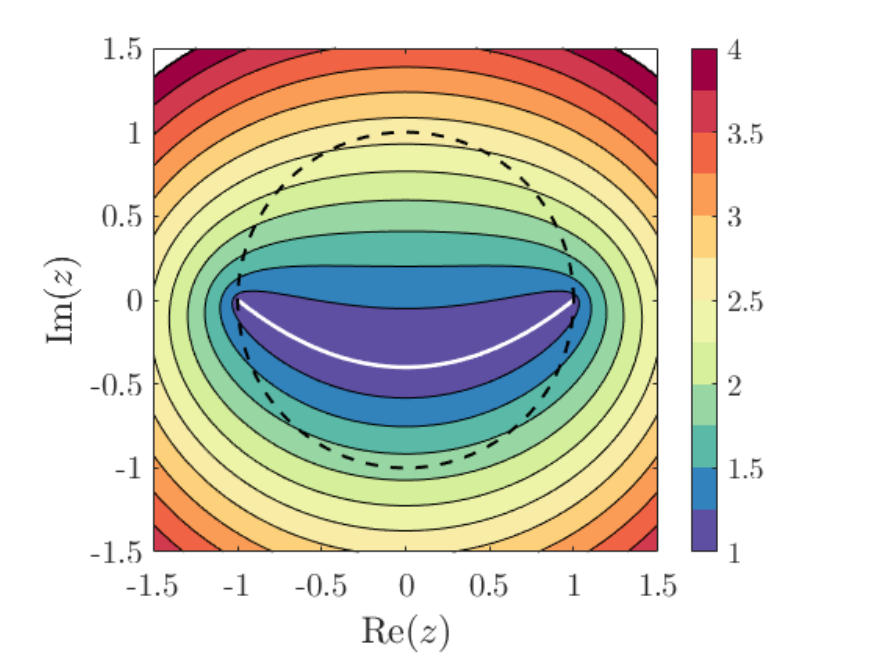}
      \caption{$\Omega=\left\{t + 0.4i(t^2-1):t\in[-1,1]\right\}$}
    \label{fig:erho_para}
    \end{subfigure}
  \begin{subfigure}{0.49\textwidth}
      \centering
      \includegraphics[width=\textwidth]{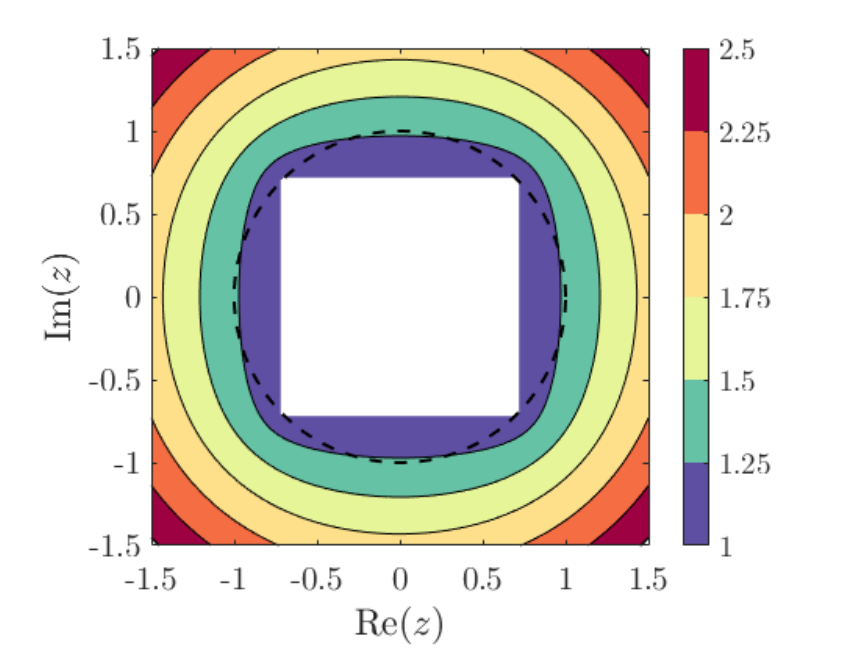}
      \caption{$\Omega=\left\{a+b i:a,b\in[-\sqrt{2}/2,\sqrt{2}/2]\right\}$}
    \label{fig:erho_box}
    \end{subfigure}
\caption{{\bf The level sets $E_\rho$ for several values of $\rho$,
corresponding to various sets $\Omega$}. 
The level sets $E_{\rho}$ for $\rho=1,1.25,1.5, \dots$ appear as the
boundaries between the colored regions in this figure. The color
bar indicates which boundary corresponds to which value of $\rho$. The set
$\Omega$ is shown in white.  The dashed circle is the boundary of the open
unit disk, used in the definition of $\rho_*$ in Definition~\ref{def:bern}.
The plots were made using the source code provided in \cite{barnett}.}
  \label{fig:erho}
\end{figure}

The parameter $\rho_*$ defined below appears in our bounds for both the
2-norm of the monomial coefficient vector of the interpolating polynomial,
and the growth rate of the 2-norm of the inverse of a Vandermonde matrix. It
denotes the parameter of the smallest region $E_\rho^o$ that contains the
open unit disk centered at the origin. 
\begin{definition}
  \label{def:bern}
Given a simply connected compact set $\Omega\subset\C$, we define 
$\rho_*:=\inf\{\rho>1:D_1\subset E_{\rho}^o\}$, where $D_1$ is the open
unit disk centered at the origin, and $E_\rho^o$ is the region corresponding
to $\Omega$ defined in Definition~\ref{def:genbern}.
\end{definition}

The following lemma provides an upper bound for the 2-norm of the monomial
coefficient vector of an arbitrary polynomial. We will use this lemma to
prove a much tighter bound in Theorem \ref{thm:coeffnorm}.
\begin{lemma}
  \label{lem:bnd}
Let $P_N:\C\to\C$ be a polynomial of degree $N$, where $P_N(z)=\sum_{k=0}^N
a_k z^k$ for some $a_0,a_1,\dots,a_N\in\C$.  Let $D_1$ denote the open unit
disk centered at the origin, and suppose that $\Omega\subset \C$ is simply
connected and compact.  The 2-norm of the coefficient vector
$a^{(N)}:=(a_0,a_1,\dots,a_N)^T$ satisfies
  \begin{align}
\norm{a^{(N)}}_2\leq\norm{P_N}_{L^\infty(\partial D_1)}\leq \rho_*^N
\norm{P_N}_{L^\infty(\Omega)},
  \end{align}
where  $\rho_*$ is given in Definition~\ref{def:bern}.
\end{lemma}
\begin{proof}
Observe that $P_N(e^{i\theta})=\sum_{k=0}^N a_k e^{ik\theta}$.
By Parseval's identity, we have that
  \begin{align}
\hspace*{-3.5em}
\norm{a^{(N)}}_2=\Bigl(\frac{1}{2\pi}\int_0^{2\pi} \abs{P_N(e^{i\theta})}^2 \d
\theta\Bigr)^{1/2}\leq \norm{P_N}_{L^\infty(\partial D_1)} \leq 
\norm{P_N}_{L^\infty(E_{\rho_*}^o)},
  \end{align}
where the last inequality comes from the fact that $D_1\subset E_{\rho_*}^o$
(see Definition \ref{def:genbern}). Finally, using one of Bernstein's
inequalities (see Section 4.6 in \cite{walsh}), we have that
  \begin{align}
\norm{P_N}_{L^\infty(E_{\rho_*}^o)}\leq \rho_*^N
\norm{P_N}_{L^\infty(\Omega)}.
  \end{align}
\end{proof}

The following theorem is the principal result of this paper. It demonstrates
that the upper bound for the 2-norm of the monomial coefficient vector of an
arbitrary polynomial, provided in Lemma \ref{lem:bnd}, can be significantly
improved if the given polynomial is viewed as an interpolating
polynomial of a function $F$. Instead of using only the size of
$\norm{P_N}_{\li(\Omega)}$, this theorem uses the properties of the function
$F$.
\begin{theorem}
  \label{thm:coeffnorm}
Let $\Omega$ be a simply connected compact set in the complex plane, and let
$F:\Omega\to\C$ be an arbitrary function. Suppose that there exists
a finite sequence of polynomials $\{Q_n\}_{n=0,1,\dots,N}$, where
$Q_n$ has degree $n$, which satisfies 
  \begin{align}
\norm{F-Q_n}_{L^\infty(\Omega)}\leq C_N\rho^{-n},\quad 0\leq n\leq
N,\label{for:exist}
  \end{align}
for some constants $\rho>1$ and $C_N\geq 0$.
Define $P_N(z)=\sum_{k=0}^N a_k z^k$ to be
the $N$th degree interpolating polynomial of $F$ for a given set of distinct
interpolation points $Z=\{z_j\}_{j=0,1,\dots,N}\subset \Omega$. The 2-norm of
the monomial coefficient vector $a^{(N)}:=(a_0,a_1,\dots,a_N)^T$ of $P_N$
satisfies
  \begin{align}
\norm{a^{(N)}}_2 \leq \norm{F}_{\li(\Omega)}+
C_N\Bigl(\Lambda_N\Bigl(\frac{\rho_*}{\rho}\Bigr)^N+ 2 \rho_*\sum_{j=0}^{N-1}
\Bigl(\frac{\rho_*}{\rho}\Bigr)^{j} + 1\Bigr),
  \end{align}
where $\rho_*$ is given in Definition \ref{def:bern}, and
$\Lambda_N$ denotes the Lebesgue constant for $Z$.
\end{theorem}
\begin{proof}
Given $n\geq 0$, let $M^{(n)}:\C^{n+1}\to \P_n$ be the bijective linear map
associating each vector $(u_0,u_1,\dots,u_n)^T\in\R^{n+1}$ with the $n$th
degree polynomial $\sum_{k=0}^n u_k z^k\in\P_n$.
It follows immediately from Lemma \ref{lem:bnd} that, given any 
polynomial $P\in\P_n$,
  \begin{align}
\bnorm{(M^{(n)})^{-1}[P]}_{2}\leq \rho_*^n \norm{P}_{\li(\Omega)}.
  \end{align}
Therefore, by the triangle inequality, the 2-norm of the monomial
coefficient vector of the polynomial $Q_N$ satisfies
  \begin{align}
  \hspace*{-2em}
\bnorm{(M^{(N)})^{-1}[Q_N]}_2\leq\,& \bnorm{(M^{(N)})^{-1}[Q_0]}_2 +
\sum_{j=0}^{N-1}\bnorm{(M^{(N)})^{-1}[Q_{j+1}-Q_{j}]}_2 \notag\\
=\,&  \norm{Q_0}_{\li(\Omega)} +
\sum_{j=0}^{N-1}\bnorm{(M^{(j+1)})^{-1}[Q_{j+1}-Q_{j}]}_2\notag\\
\leq \,& \bigl(\norm{F}_{\li(\Omega)}+C_N\bigr)+ \sum_{j=0}^{N-1}
\rho_*^{j+1}\snorm{Q_{j+1}-Q_{j}}_{\li(\Omega)} \notag\\ 
\leq \,&
\bigl(\norm{F}_{\li(\Omega)}+ C_N\bigr)+2C_N \rho_*\sum_{j=0}^{N-1}
\Bigl(\frac{\rho_*}{\rho}\Bigr)^{j},
  \label{for:boundp3}
  \end{align}
from which it follows that $\norm{a^{(N)}}_2$ satisfies
\begin{align}
 \norm{a^{(N)}}_2\leq\,& \bnorm{(M^{(N)})^{-1}[P_N-Q_N]}_2 +
 \bnorm{(M^{(N)})^{-1}[Q_N]}_2\notag\\
\leq \,& \rho_*^N\norm{P_N-Q_N}_{\li(\Omega)}+
\bnorm{(M^{(N)})^{-1}[Q_N]}_2\notag\\
\leq \,& \rho_*^N\Lambda_N\norm{F-Q_N}_{\li(\Omega)}+
\bnorm{(M^{(N)})^{-1}[Q_N]}_2\notag\\
\leq \,&\norm{F}_{\li(\Omega)}+
C_N\Bigl(\Lambda_N\Bigl(\frac{\rho_*}{\rho}\Bigr)^N+ 2
\rho_*\sum_{j=0}^{N-1} \Bigl(\frac{\rho_*}{\rho}\Bigr)^{j}+1\Bigr), 
\label{for:boundp4}
\end{align}
where the third inequality comes from the observation that $P_N-Q_N$ is the
interpolating polynomial of $F-Q_N$ for the set of interpolation points $Z$.
\end{proof}

\begin{remark}
The assumption \eqref{for:exist} made in Theorem \ref{thm:coeffnorm} can be
satisfied for any function $F$ by choosing $C_N$ to be sufficiently large.
When the function $F$ has, for some $\rho>\rho_*$, an analytic continuation
on the closure of the region $E_\rho^o$ corresponding to $\Omega$ (see
Definitions~\ref{def:genbern} and $\ref{def:bern}$), 
one can show that $\norm{a^{(N)}}_2\lesssim 2\norm{F}_{\li(\partial D_1)}$,
where $D_1$ is the open unit disk centered at the origin.  This result is
derived from Section 4.7 in \cite{walsh}, which establishes that
$\norm{F-P_N}_{\li(E_{\rho_*}^o)}=\O\bigl((\rho_*/\rho)^N\bigr)$, where the
asymptotic constant is proportional to $\norm{F}_{\li(E_{\rho_*}^o)}$.
Consequently, it follows that
  \begin{align}
  \hspace*{-6.0em}
\norm{a^{(N)}}_2\leq
\norm{P_{N}}_{L^\infty(\partial D_1)} \leq
\norm{F-P_{N}}_{\li(\partial D_1)}+\norm{F}_{\li(\partial D_1)}\lesssim
2\norm{F}_{\li(\partial D_1)}.
  \end{align}
\end{remark}

Theorem \ref{thm:coeffnorm} provides an a priori upper bound on the 2-norm
of the monomial coefficient vector. In order to apply Theorem
\ref{thm:mono_err}, we also need to determine when the condition
$\norm{(V^{(N)})^{-1}}_2\leq \frac{1}{2u\cdot \gamma_N}$ applies. 
The following theorem, which is an immediate corollary of Lemma
\ref{lem:bnd}, bounds the growth of $\norm{(V^{(N)})^{-1}}_2$. It includes
some previous results \cite{walter2,walter1} as special cases of our bound. 
\begin{theorem}
  \label{thm:cond2}
Suppose that $V^{(N)}\in\C^{(N+1)\times(N+1)}$ is a Vandermonde matrix with
$(N+1)$ distinct interpolation points $Z=\{z_j\}_{j=0,1,\dots,N}\subset \C$.
Suppose further that $\Omega\subset \C$ is a simply connected compact set such that
$Z\subset\Omega$. The 2-norm of $(V^{(N)})^{-1}$ is bounded by
  \begin{align}
\norm{(V^{(N)})^{-1}}_2\leq  \rho_*^N \Lambda_N,\label{for:cond2}
  \end{align}
where $\rho_*$ is given in Definition \ref{def:bern}, and
$\Lambda_N$ denotes the Lebesgue constant for the set
of interpolation points $Z$ over $\Omega$.
\end{theorem}
\begin{proof}
Let $f^{(N)}=(f_0,f_1,\dots, f_N)^T\in\C^{N+1}$ be an arbitrary vector.
Suppose that $P_N$ is an interpolating polynomial of degree $N$ for the set
$\{(z_j,f_j)\}_{j=0,1,\dots,N}$. By Lemma \ref{lem:bnd}, the
2-norm of the monomial coefficient vector $a^{(N)}$ of $P_N$ satisfies 
  \begin{align}
\norm{a^{(N)}}_2\leq \rho_*^N\norm{P_N}_{L^{\infty}(\Omega)}\leq \rho_*^N
\Lambda_N \norm{f^{(N)}}_\infty\leq \rho_*^N \Lambda_N
\norm{f^{(N)}}_2,\label{for:cond222}
  \end{align}
where the second inequality follows from the definition of the Lebesgue
constant. Therefore, the 2-norm of $(V^{(N)})^{-1}$ is bounded by
  \begin{align}
  \hspace*{-4.5em}
\norm{(V^{(N)})^{-1}}_2= \sup_{f^{(N)}\neq
0}\Bigl\{\frac{\norm{(V^{(N)})^{-1}f^{(N)}}_2}{\norm{f^{(N)}}_2}\Bigr\} =
\sup_{f^{(N)}\neq
0}\Bigl\{\frac{\norm{a^{(N)}}_2}{\norm{f^{(N)}}_2}\Bigr\}\leq\rho_*^N
\Lambda_N.\label{for:invv}
  \end{align}
\end{proof}

Note that the bound above applies to any simply connected compact domain
$\Omega\subset\C$ that contains the set of interpolation points $Z$.

\begin{observation}
  \label{obs:cond2}
In the case where the set of $(N+1)$ interpolation points $Z\subset\Omega$
are chosen such that the associated Lebesgue constant $\Lambda_N\lesssim 1$,
we observe in practice that the upper bound $\rho_*^N\Lambda_N$ is close to
the value of $\norm{(V^{(N)})^{-1}}_2$ (see Figures~\ref{fig:cond_veri}
and~\ref{fig:cond_general} for numerical evidence).
\end{observation}

\begin{remark}
  \label{rem:worst}
A worst case upper bound for $\norm{a^{(N)}}_2$ is provided by inequality
\eqref{for:cond222} in the proof of Theorem \ref{thm:cond2}, as the
right-hand side of this inequality depends only on the size of $F$. 
\end{remark}

\begin{remark}
Formula (4.5) in \cite{beck_thesis} establishes the bound
$\norm{(V^{(N)})^{-1}}_\infty\leq  \Delta_N(D_1,\Omega)\cdot \Lambda_N$, where
$\Delta_N(D_1,\Omega):=\sup_{p\in\P_N,p\neq 0}
\frac{\norm{p}_{L^\infty(D_1)}}{\norm{p}_{L^\infty(\Omega)}}$, and $\P_N$
denotes the set of all polynomials of degree less than or equal to $N$. By
combining the maximum principle with Lemma \ref{lem:bnd}, we obtain the
bound $\Delta_N(D_1,\Omega)\leq \rho_*^N$, from which it follows that
$\norm{(V^{(N)})^{-1}}_\infty\leq \rho_*^N\Lambda_N$.  This result can
alternatively be obtained more directly using the inequality
  \begin{align}
  \hspace*{-1.5em}
\norm{(V^{(N)})^{-1}}_\infty=
\sup_{f^{(N)}\neq
0}\Bigl\{\frac{\norm{a^{(N)}}_\infty}{\norm{f^{(N)}}_\infty}\Bigr\}\leq 
\sup_{f^{(N)}\neq
0}\Bigl\{\frac{\norm{a^{(N)}}_2}{\norm{f^{(N)}}_\infty}\Bigr\}
\leq\rho_*^N \Lambda_N,
  \end{align}
where the final inequality follows from \eqref{for:cond222}.
\end{remark}

\subsection{Under what conditions is interpolation in the monomial basis as
good as interpolation in a well-conditioned polynomial basis?}
\label{sec:good}

As before, we assume that the simply connected compact set $\Omega$ is inside the unit
disk centered at the origin (such that $\gamma_N\lesssim 1$). Furthermore,
we choose a set of $(N+1)$ interpolation points $Z\subset \Omega$ with a
Lebesgue constant $\Lambda_N\lesssim 1$, and let $V^{(N)}$ denote the
corresponding Vandermonde matrix.  Recall from Theorem \ref{thm:mono_err}
that, if 
\begin{align}
  \norm{(V^{(N)})^{-1}}_2\leq \frac{1}{2u\cdot \gamma_N},\label{for:precon}
\end{align}
then the monomial approximation error $\norm{F-\hat
P_N}_{L^\infty(\Omega)}$ is bounded a priori by
  \begin{align}
\hspace*{-0.0em}
\norm{F-\hat P_N}_{L^\infty(\Omega)} \lesssim
\norm{F-P_N}_{L^\infty(\Omega)}+\mach \cdot \norm{a^{(N)}}_2,
\label{for:apri}
  \end{align}
where $u$ denotes machine epsilon, $\hat P_N$ is the computed monomial
expansion, $P_N$ is the exact $N$th degree interpolating polynomial of $F$
for the set of interpolation points $Z$, and~$a^{(N)}$ is the monomial
coefficient vector of $P_N$.

By Theorem~\ref{thm:coeffnorm}, if we choose a constant $C_N\geq 0$ and a
finite sequence of polynomials $\{Q_n\}_{n=0,1,\dots,N}$ such that
$\norm{F-Q_n}_{L^\infty(\Omega)}\leq C_N\rho_*^{-n}$ for $0\leq n \leq N$,
where $Q_n$ has degree $n$ and $\rho_*$ is given in Definition
\ref{def:bern}, then the monomial coefficient vector $a^{(N)}$ of
$P_N$ satisfies
  \begin{align}
\norm{a^{(N)}}_2\lesssim C_N \Lambda_N N\approx C_N,
\label{for:bpri00}
  \end{align}
and inequality \eqref{for:apri} becomes
  \begin{align}
\hspace*{-0.0em}
\norm{F-\hat P_N}_{L^\infty(\Omega)} \lesssim &\,
\norm{F-P_N}_{L^\infty(\Omega)}+\mach\cdot C_N.\label{for:bpri}
  \end{align}
In practice, one can take $\{Q_n\}_{n=0,1,\dots,N}$ to be a finite sequence
of interpolating polynomials $\{P_n\}_{n=0,1,\dots,N}$ of $F$ for sets of
interpolation points with Lebesgue constants $\lesssim 1$.  When the Lebesgue
constant $\Lambda_N$ is small, it follows from Theorem~\ref{thm:cond2} that
the condition \eqref{for:precon} is satisfied when $\rho_*^N\lesssim
\frac{1}{u}$. Suppose, then, that $N$ is sufficiently
small so that $\rho_*^N\lesssim \frac{1}{u}$.  Without loss of generality,
we assume that the upper bound for $\norm{F-P_n}_{\li(\Omega)}$, i.e.,
$C_N\rho_*^{-n}$, is tight, in the sense that there exists some integer
$n\in[0,N]$ such that $\norm{F-P_n}_{\li(\Omega)}=C_N\rho_*^{-n}$.  Note
that the smallest uniform approximation error we can hope to obtain in
practice is $u\cdot \norm{F}_{L^\infty(\Omega)}$.

When $u\cdot C_N \lesssim \max(\norm{F-P_N}_{L^\infty(\Omega)},u\cdot
\norm{F}_{L^\infty(\Omega)})$, the use of a monomial basis for interpolation
introduces essentially no extra error.  Interestingly, this happens both if
the polynomial interpolation error decays quickly and if the polynomial
interpolation error decays slowly. Suppose that the polynomial interpolation
error decays quickly, so that the bound is tight for $n=0$, i.e.,
$\norm{F-P_0}_{L^\infty(\Omega)} = C_N$.  Since $C_N \lesssim
2\norm{F}_{L^\infty(\Omega)}$, we see that the extra error caused by the use
of a monomial basis is bounded by $u \cdot C_N \lesssim 2u\cdot
\norm{F}_{L^\infty(\Omega)} \lesssim u\cdot \norm{F}_{L^\infty(\Omega)}$.
Examples of this situation are illustrated in Figure~\ref{fig:rap_new1}.
Alternatively, suppose that the polynomial interpolation error decays slowly,
so that bound is tight for $n=N$, i.e., $\norm{F-P_N}_{L^\infty(\Omega)} =
C_N\rho_*^{-N}$.  Since we assumed that $\rho_*^N\lesssim \frac{1}{u}$, it
follows that $u \cdot C_N \lesssim \norm{F-P_N}_{\li(\Omega)}$. Examples of
this situation are illustrated in Figure~\ref{fig:rap_new2}.

When $u\cdot C_N \gtrsim \max(\norm{F-P_N}_{L^\infty(\Omega)},u\cdot
\norm{F}_{L^\infty(\Omega)})$, stagnation of
convergence can occur. In practice, we observe that the extra
error caused by the use of the monomial basis, i.e.,
$u\cdot\norm{a^{(N)}}_2$, is close to $u\cdot C_N$, so the monomial
approximation error $\norm{F-\hat P_N}_{L^\infty(\Omega)}$ generally
stagnates at an error level around $u\cdot C_N$. 
Note that a slow decay in $\norm{F-P_n}_{\li(\Omega)}$ results in a larger
value of $C_N$, while a fast decay in $\norm{F-P_n}_{\li(\Omega)}$ favors a
smaller final interpolation error $\norm{F-P_N}_{L^\infty(\Omega)}$.  This
means that, for stagnation of convergence to occur, the polynomial
interpolation error has to exhibit some combination of slow decay followed by
fast decay. 
Furthermore, note that an upper bound for $C_N$ is given by $C_N \lesssim
\rho_*^N$ (see Remark~\ref{rem:worst}).  This means that, the smaller the
value of $N$, the smaller the maximum possible value of $C_N$, and the more
rapid the rate of decay in $\norm{F-P_n}_{\li(\Omega)}$ required for
stagnation of convergence to occur. We present examples of this
situation in Figure \ref{fig:rap_new3}.

\begin{figure}[h]
    \centering
  \begin{subfigure}{0.49\textwidth}
      \centering
      \includegraphics[width=\textwidth]{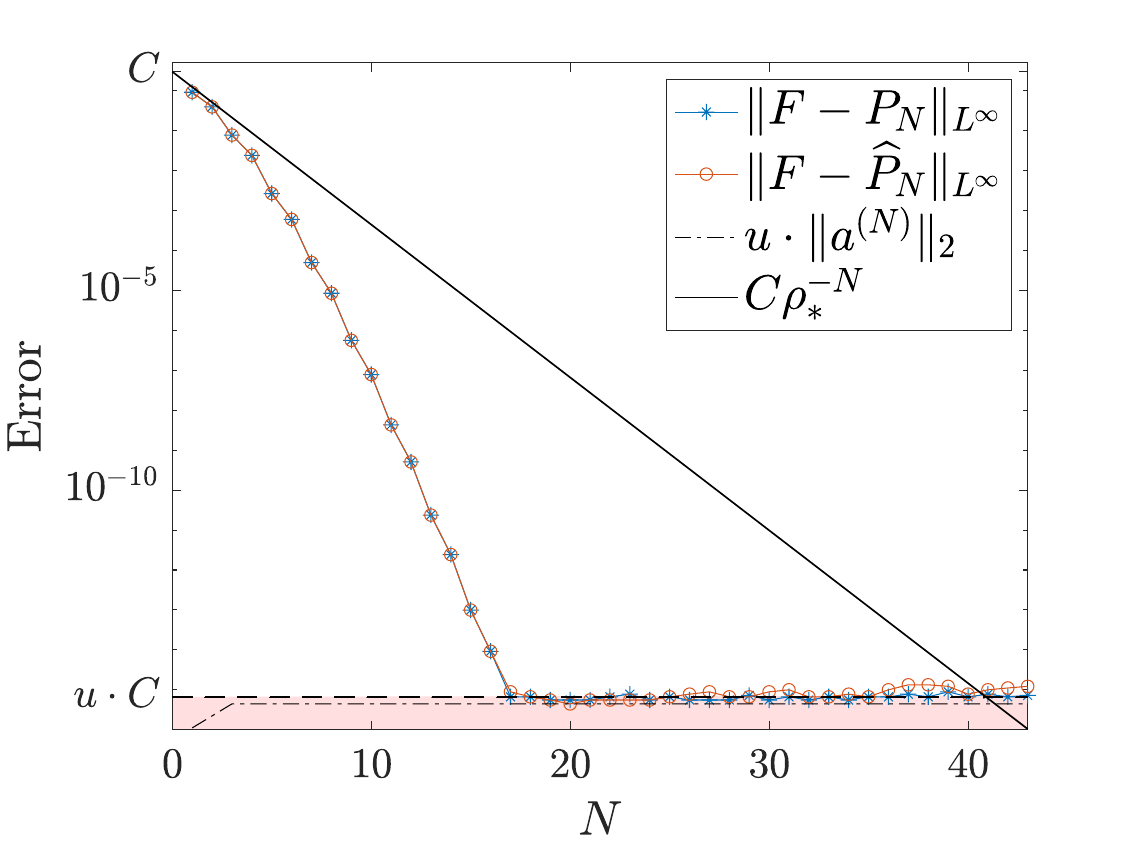}
      \caption{$F(x)=\cos(2x+1)$, $C=3\e{0}$}
    \end{subfigure}
  \begin{subfigure}{0.49\textwidth}
      \centering
      \includegraphics[width=\textwidth]{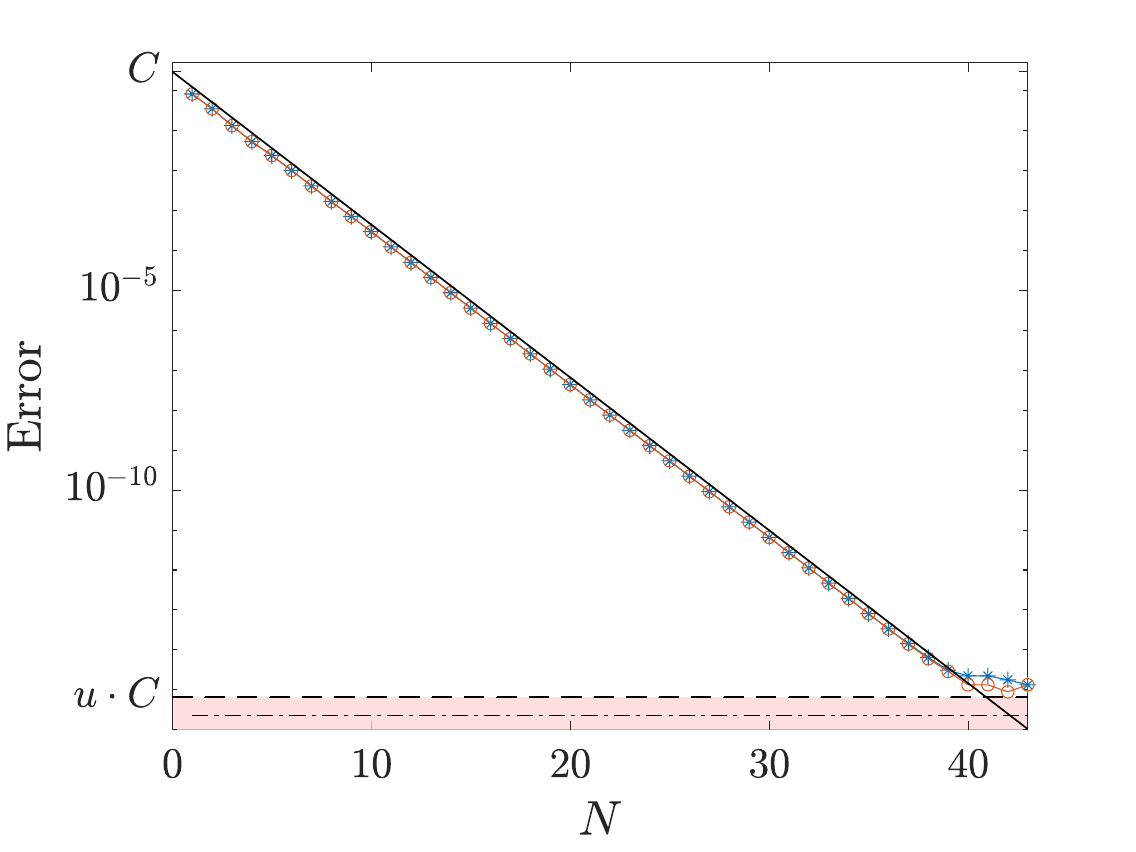}
      \caption{$F(x)=\frac{1}{x-\sqrt{2}}$, $C=3\e{0}$}
    \end{subfigure}
      \caption{{\bf Cases when the extra error caused by the use of a monomial basis is
      negligible}. The function $F$ is interpolated over $\Omega=[-1,1]$.
      The value of $C$ is chosen such that $\norm{F-P_N}_{\li([-1,1])}\leq
      C\rho_*^{-N}$ for $N=0,1,...,40$. We highlight the region bounded
      above by $u\cdot C$ in pink.} 
      \label{fig:rap_new1}
\end{figure}

\begin{figure}[h]
    \centering
  \begin{subfigure}{0.49\textwidth}
      \centering
      \includegraphics[width=\textwidth]{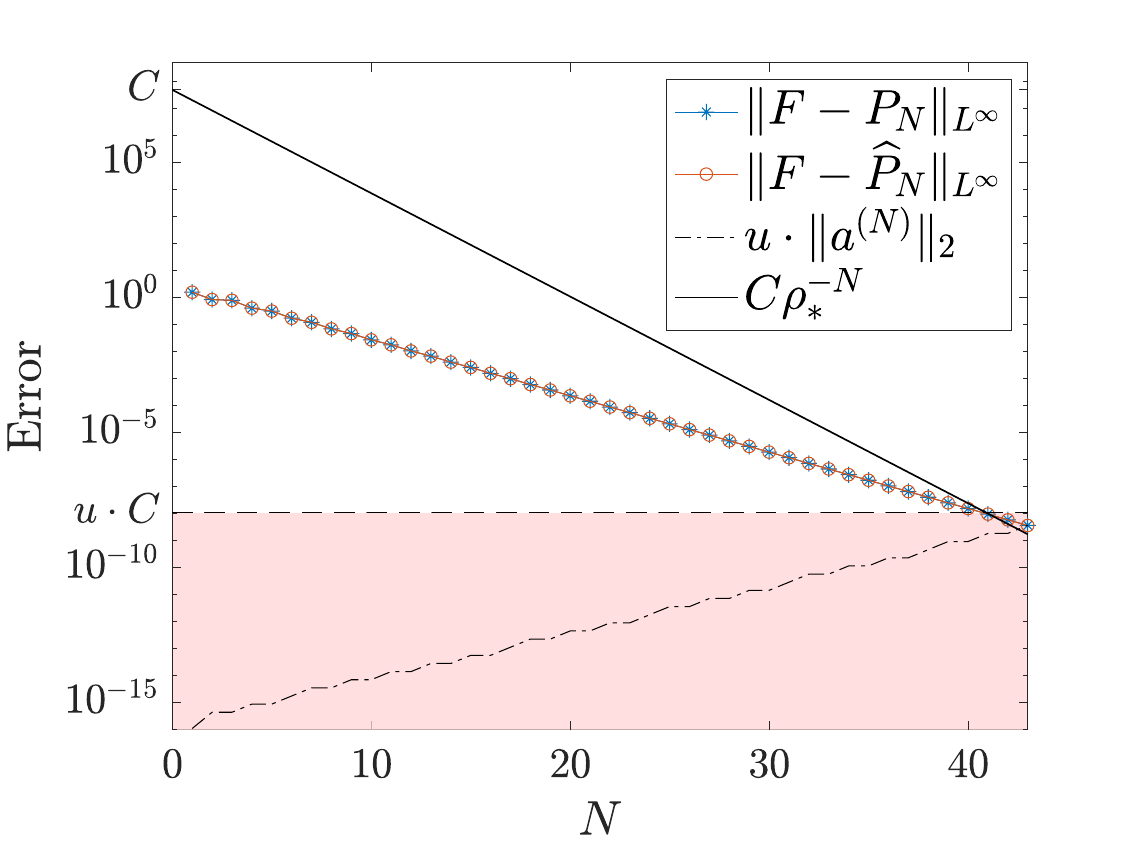}
      \caption{$F(x)=\frac{1}{x-0.5i}$, $C=5\e{7}$}
    \end{subfigure}
  \begin{subfigure}{0.49\textwidth}
      \centering
      \includegraphics[width=\textwidth]{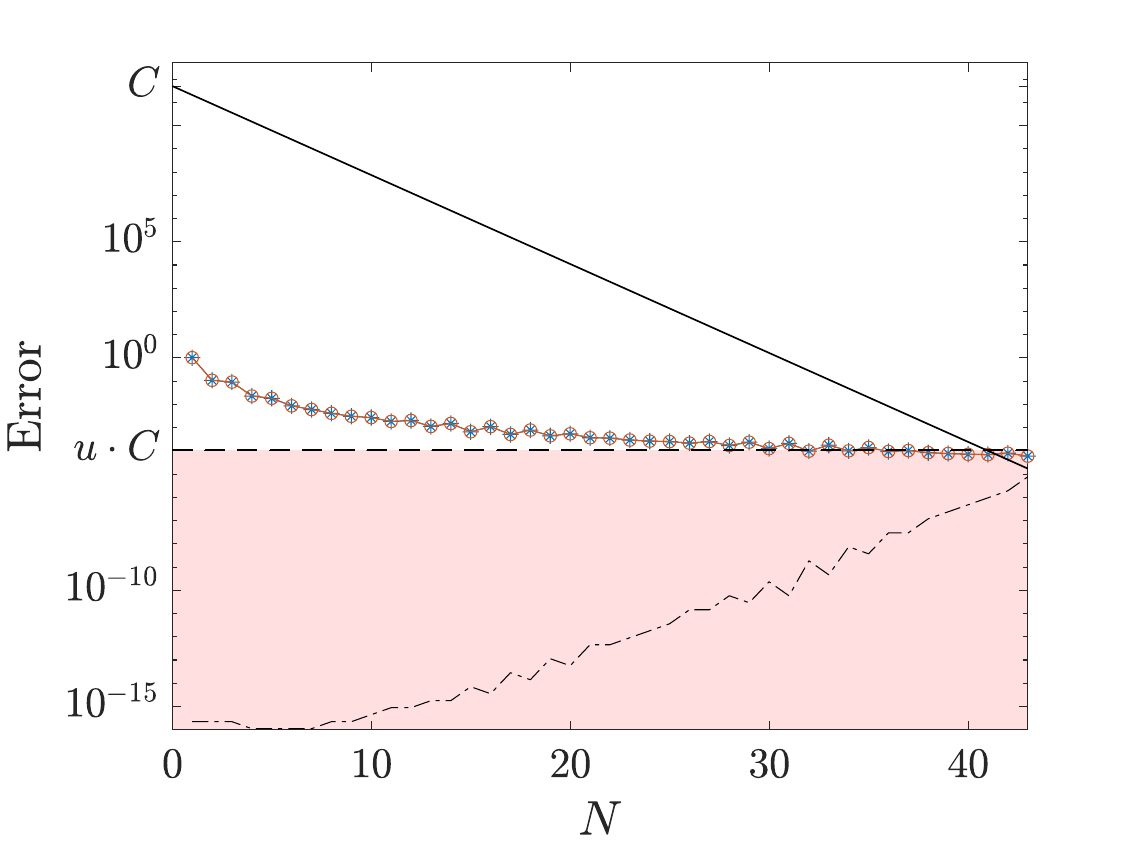}
      \caption{$F(x)=\abs{x+0.1}^{2.5}$, $C=5\e{11}$}
    \end{subfigure}      
      \caption{{\bf Cases when the extra error caused by the use of a
      monomial basis is no larger than $\norm{F-P_N}_{L^\infty(\Omega)}$}.
      The function $F$ is interpolated over $\Omega=[-1,1]$.
      The value of $C$ is chosen such that $\norm{F-P_N}_{\li([-1,1])}\leq
      C\rho_*^{-N}$ for $N=0,1,...,40$.  We highlight the region bounded
      above by $u\cdot C$ in pink.}
    \label{fig:rap_new2}
\end{figure}

\begin{figure}[h]
    \centering
  \begin{subfigure}{0.49\textwidth}
      \centering
      \includegraphics[width=\textwidth]{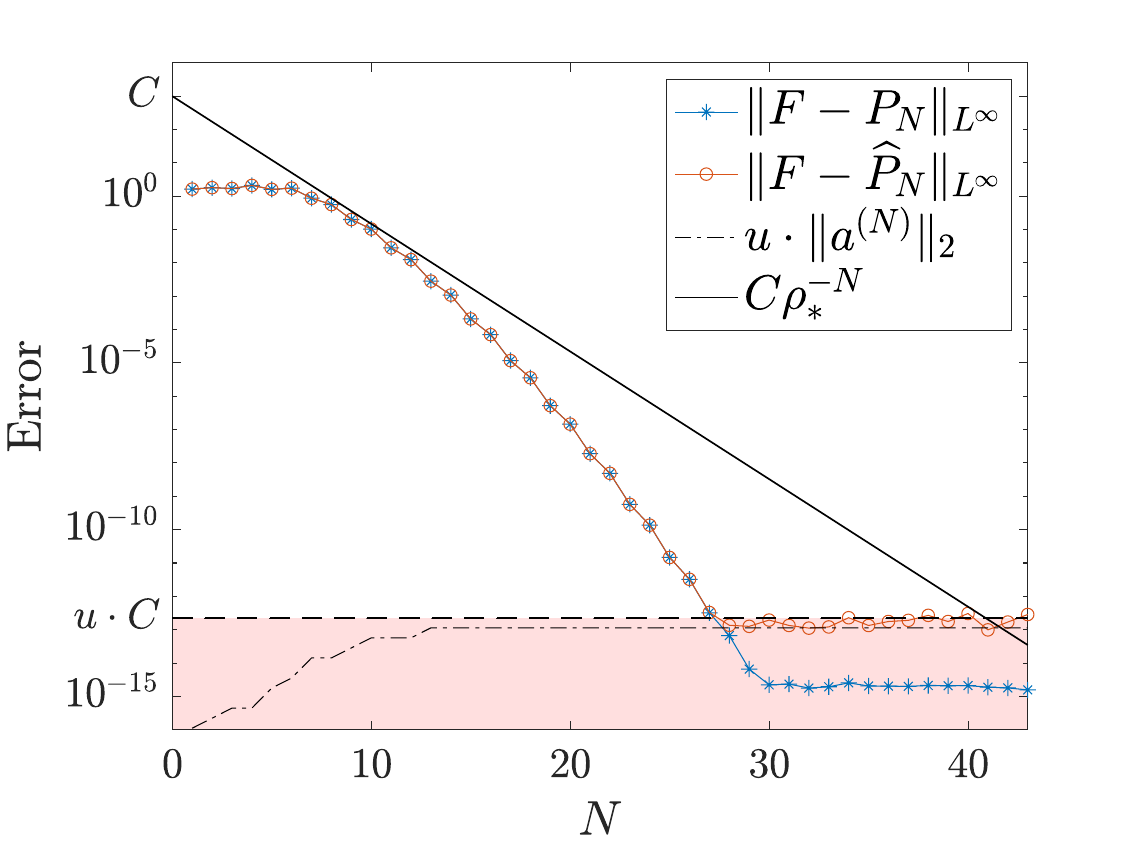}
      \caption{$F(x)=\cos(8x+1)$, $C=1\e{3}$}
    \end{subfigure}
    \begin{subfigure}{0.49\textwidth}
      \centering
      \includegraphics[width=\textwidth]{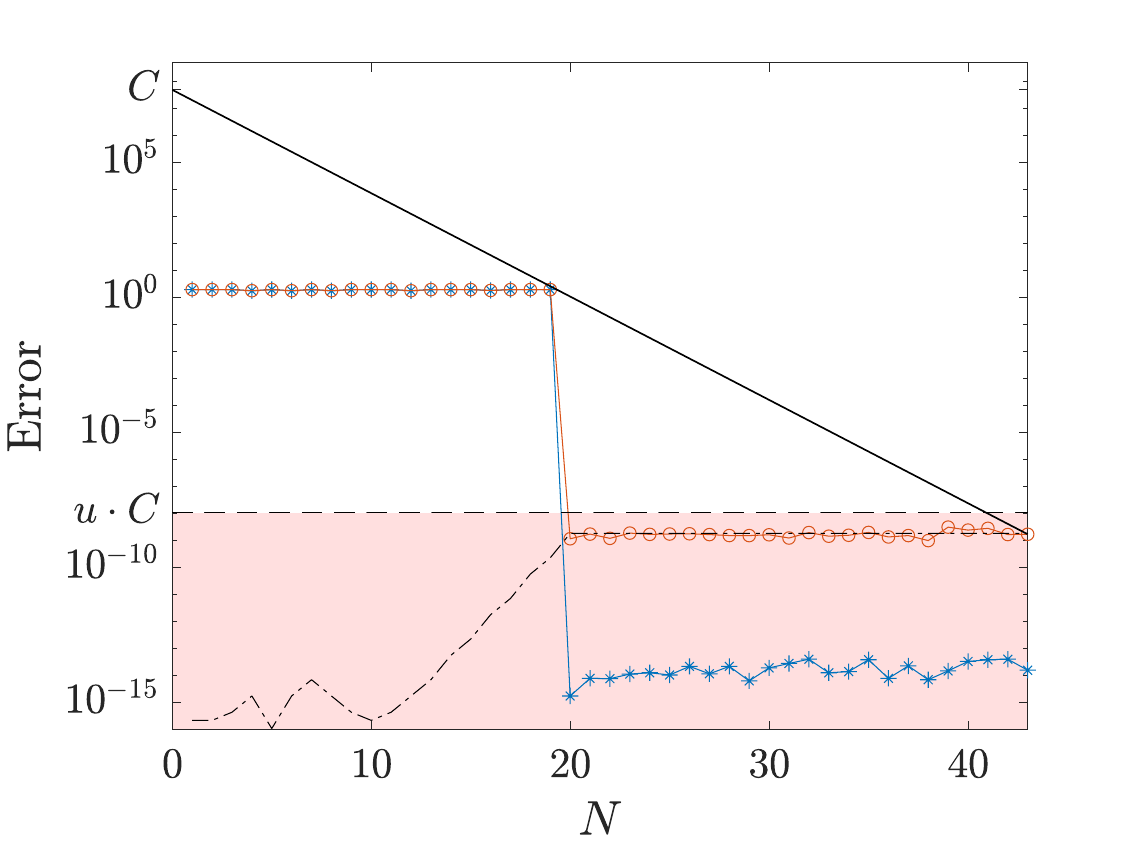}
      \caption{$F(x)=T_{20}(x)$, $C=5\e{7}$}
    \end{subfigure}
      \caption{{\bf Cases when convergence stagnates}. The function $F$
      is interpolated over $\Omega=[-1,1]$. 
      The value of $C$ is chosen such that $\norm{F-P_N}_{\li([-1,1])}\leq
      C\rho_*^{-N}$ for $N=0,1,...,40$. We highlight the region bounded
      above by $u\cdot C$ in pink.}
    \label{fig:rap_new3}
\end{figure}

\subsection{Practical use of a monomial basis for interpolation}
\label{sec:limit}

What are the restrictions on polynomial interpolation in the monomial basis?
Firstly, extremely high-order global interpolation is impossible in the
monomial basis, because the order $N$ must satisfy
$\norm{(V^{(N)})^{-1}}_2\lesssim \frac{1}{u}$ for our estimates to hold,
where $u$ denotes machine epsilon. In fact, even if this condition were not
required, there would still be no benefit in taking an order larger than
this threshold.  For most functions, if a polynomial of degree exceeding the
threshold is needed to resolve the function to machine precision, the error
caused by the use of a monomial basis typically dominates the true
polynomial interpolation error, which leads to stagnation of convergence.
This can be seen from the discussion in Section \ref{sec:good}.

On the other hand, piecewise polynomial interpolation in the monomial basis
over a partition of $\Omega$ can be carried out stably, provided that the
maximum order of approximation over each subelement of $\Omega$ is maintained below the
threshold $\argmax_N \norm{(V^{(N)})^{-1}}_2\lesssim \frac{1}{u}$, and that
the size of~$u\cdot \norm{a^{(N)}}_2\approx u\cdot \norm{\hat a^{(N)}}_2$ is
kept below the size of the polynomial interpolation error, where $a^{(N)}$
and $\hat a^{(N)}$ denote the exact and the computed monomial coefficient
vectors, respectively.  As demonstrated in Section \ref{sec:good}, the
latter requirement is often satisfied automatically, and when it is not,
adding an extra level of subdivision almost always resolves the issue.  In
addition, the extra error caused by the use of a monomial basis can always
be estimated promptly during computation, using the value of $u
\cdot\norm{\hat a^{(N)}}_2$.

Based on the discussion above, we summarize a proper method of using the
monomial basis for interpolation, as follows. For simplicity, we use the same
order of approximation over each subelement, and denote the order by $N$.
This value of $N$ needs to be smaller than the threshold $\argmax_N
\norm{(V^{(N)})^{-1}}_2\lesssim \frac{1}{u}$.  Given a
function $F:\Omega\to\C$ and an error tolerance $\varepsilon$, we subdivide
the domain $\Omega$ until $F$ can be approximated by a polynomial of degree
less than or equal to $N$ over each subelement to within an error of $\varepsilon$. Finally,
we subdivide the subelements further until the 2-norm of the monomial coefficient
vector $\hat a^{(N)}$ is less than $\varepsilon/u$ over each subelement.

Since the convergence rate of piecewise polynomial approximation is
$\O(h^{N+1})$, where $h$ and $N$ denote the maximum diameter and minimum
order of approximation over all subelements, respectively, and since the
aforementioned threshold is generally not small (e.g., the threshold is
approximately equal to $43$ when $\Omega=[-1,1]$), piecewise polynomial
interpolation in the monomial basis converges rapidly so long as we set the
value of $N$ to be large enough.  Therefore, there is no need to avoid the
use of a monomial basis when it offers an advantage over other bases.

\begin{remark}
It takes $\O(N^3)$ operations to solve a Vandermonde system of size $N\times
N$ by a standard backward stable solver.  Since the order of approximation
$N$ is almost always not large, the solution to the Vandermonde system can
be computed both rapidly, as well as accurately in the sense that
$\gamma_N$ is small, using highly optimized linear algebra libraries, e.g.,
LAPACK.  There also exist specialized algorithms that solve Vandermonde
systems in $\O(N^2)$ operations, e.g., the Bj\"orck-Pereyra algorithm
\cite{bjorck} and the Parker-Traub algorithm~\cite{traub}.  We note that the
Parker-Traub algorithm is known to be backward stable \cite{traub}, while
the stability of the Bj\"orck-Pereyra algorithm is known only under very
restrictive conditions \cite{higham}.
\end{remark}

\section{Numerical experiments}

  \label{sec:exp}

\subsection{Interpolation over an interval}
  \label{sec:mono_int}
In this section, we consider polynomial interpolation in the monomial basis
over an interval $\Omega=[a,b]\subset\R$. We compute the polynomial
interpolants using Chebyshev points on the interval $[a,b]$, since their
associated Lebesgue constant, $\Lambda_N$, is bounded by
$\frac{2}{\pi}\log(N+1)+1$, where $N$ is the order of approximation
\cite{zeller}.

In Figure~\ref{fig:cond_veri}, we report the values of
$\norm{(V^{(N)})^{-1}}_2$ and its upper bound $\rho_*^N\Lambda_N$ (see
Theorem \ref{thm:cond2}), for the domains $\Omega=[-1,1]$ and
$\Omega=[0,1]$. One can see that the upper bound is close to being tight.  Note
that when $\Omega=[-1,1]$, we have that $\rho_*=1+\sqrt{2}$ and
$\norm{(V^{(N)})^{-1}}_2\leq \frac{1}{u}$ for $N\leq 43$; when
$\Omega=[0,1]$, we have that $\rho_*=3+2\sqrt{2}$ and
$\norm{(V^{(N)})^{-1}}_2\leq \frac{1}{u}$ for $N\leq 22$.  

In Figure \ref{fig:888}, we interpolate several functions over
$\Omega=[-1,1]$. We denote the interpolant constructed using the monomial
basis by $\hat P_N$, and estimate the true polynomial interpolant $P_N$ using the Lagrange
basis evaluated by the Barycentric interpolation formula. 
In addition to the estimated values of
$\norm{F-P_N}_{\li([-1,1])}$ and $\norm{F-\hat P_N}_{\li([-1,1])}$, we plot
three additional curves in each figure: the estimated values of $u\cdot
\norm{a^{(N)}}_{2}$ based on inequality \eqref{for:mm999}, the upper 
bound $C\rho_*^{-N}$ for $\norm{F-P_N}_{\li([-1,1])}$, where
$C$ was chosen so that $\norm{F-P_N}_{\li([-1,1])}\leq
C\rho_*^{-N}$ for $N=0,1,...,40$, and the approximate upper bound
$u\cdot C$ for $u\cdot \norm{a^{(N)}}_{2}$. In
Figure \ref{fig:999}, we provide similar experiments for the case where
$\Omega=[0,1]$, this time choosing $C$ so that
$\norm{F-P_N}_{\li([0,1])}\leq C\rho_*^{-N}$ for $N=0,1,...,20$.  Based on
these experimental results, one can observe that the convergence stagnates
after the monomial approximation error $\norm{F-\hat P_N}_{\li(\Omega)}$
reaches $\mach \cdot \norm{a^{(N)}}_2$,
which implies that inequality \eqref{for:apri} is sharp. Moreover, the
values of $\mach \cdot \norm{a^{(N)}}_2$ are always less than the upper
bound $u\cdot C$, which agrees with our analysis in Section \ref{sec:good}. 

What happens when the order of approximation $N$ exceeds the threshold,
i.e., when $\norm{(V^{(N)})^{-1}}_2\geq \frac{1}{u}$?
In Figure \ref{fig:888_exceed}, we use two different solvers to
solve the Vandermonde system with Chebyshev interpolation points
over the interval $[-1,1]$: LU factorization with partial pivoting, and the
truncated singular value decomposition (TSVD) with a truncation cutoff of
$10^{-14}$. Since backward stability no longer guarantees that $\norm{\hat
a^{(N)}}_2\approx \norm{a^{(N)}}_2$ when $N$ is greater than the threshold,
the upper bound for the monomial approximation error \eqref{for:priori1}
becomes invalid, and our previous analysis is no longer applicable. As shown
in the experiments, the approximation error does not improve when $N>43$.
When the system is solved using LU decomposition, the error fluctuates. We
also note that the error could  grow exponentially after $N$ exceeds the
threshold, depending on the particular choice of the backward stable solver.
When the TSVD is used, the error is stable.  This can be explained by
Theorem 2.1 in \cite{mhz}, which provides an analogous bound to
\eqref{for:priori1} for the case of the TSVD.

\begin{figure}[h]
    \centering
  \begin{subfigure}{0.49\textwidth}
      \centering
      \includegraphics[width=\textwidth]{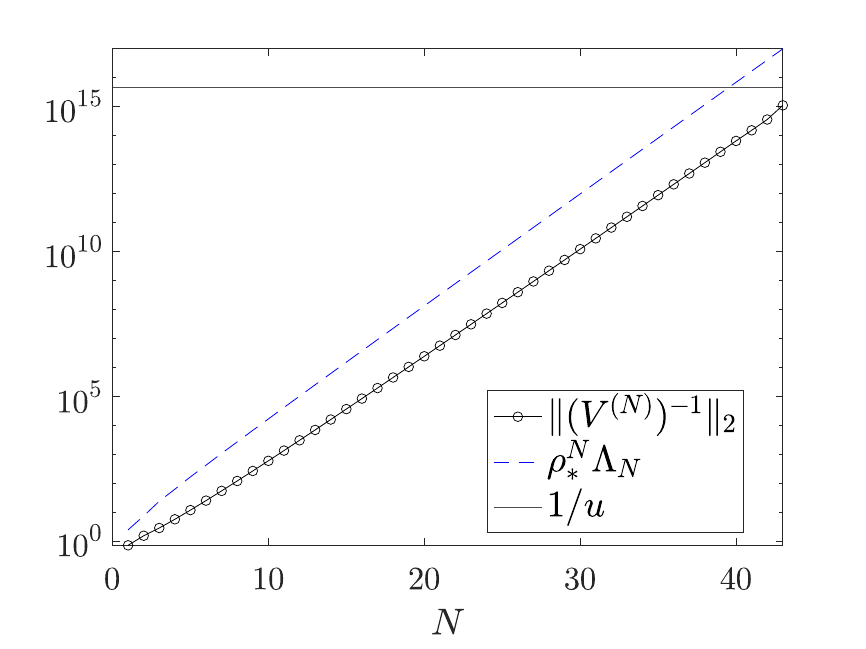}
      \caption{$\Omega=[-1,1]$}
      \label{fig:cond_veri:a}
    \end{subfigure}
  \begin{subfigure}{0.49\textwidth}
      \centering
      \includegraphics[width=\textwidth]{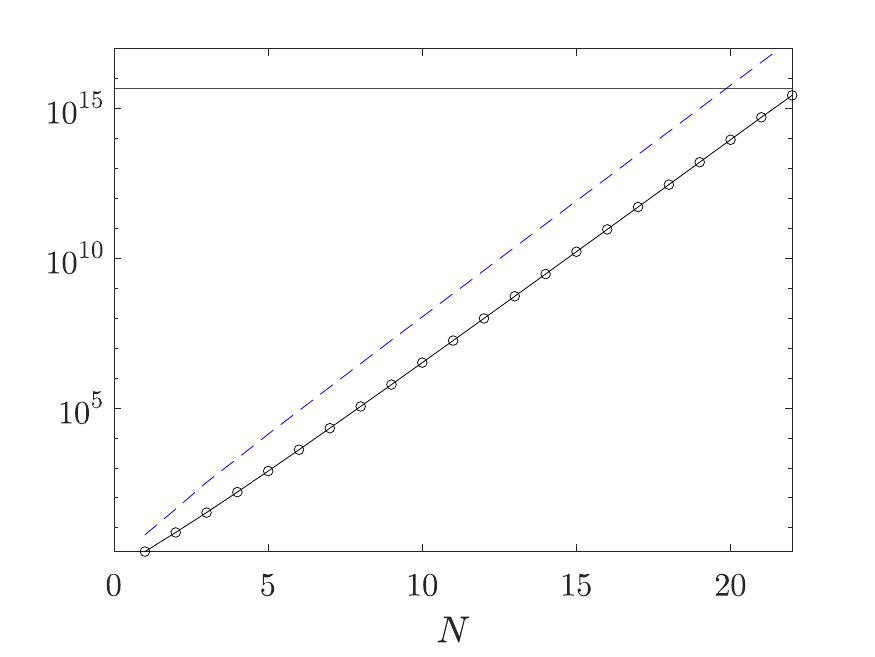}
      \caption{$\Omega=[0,1]$}
    \end{subfigure}
      \caption{{\bf The 2-norm of $(V^{(N)})^{-1}$ with Chebyshev
      interpolation points over an interval $\Omega$, and its upper bound,
      for different orders of approximation}. We note that
      $\rho_*=1+\sqrt{2}$ when $\Omega=[-1,1]$, and $\rho_*=3+2\sqrt{2}$
      when $\Omega=[0,1]$. }
      \label{fig:cond_veri}
\end{figure}

\begin{figure}
    \centering
    \begin{subfigure}{0.49\textwidth}
      \centering
      \includegraphics[width=\textwidth]{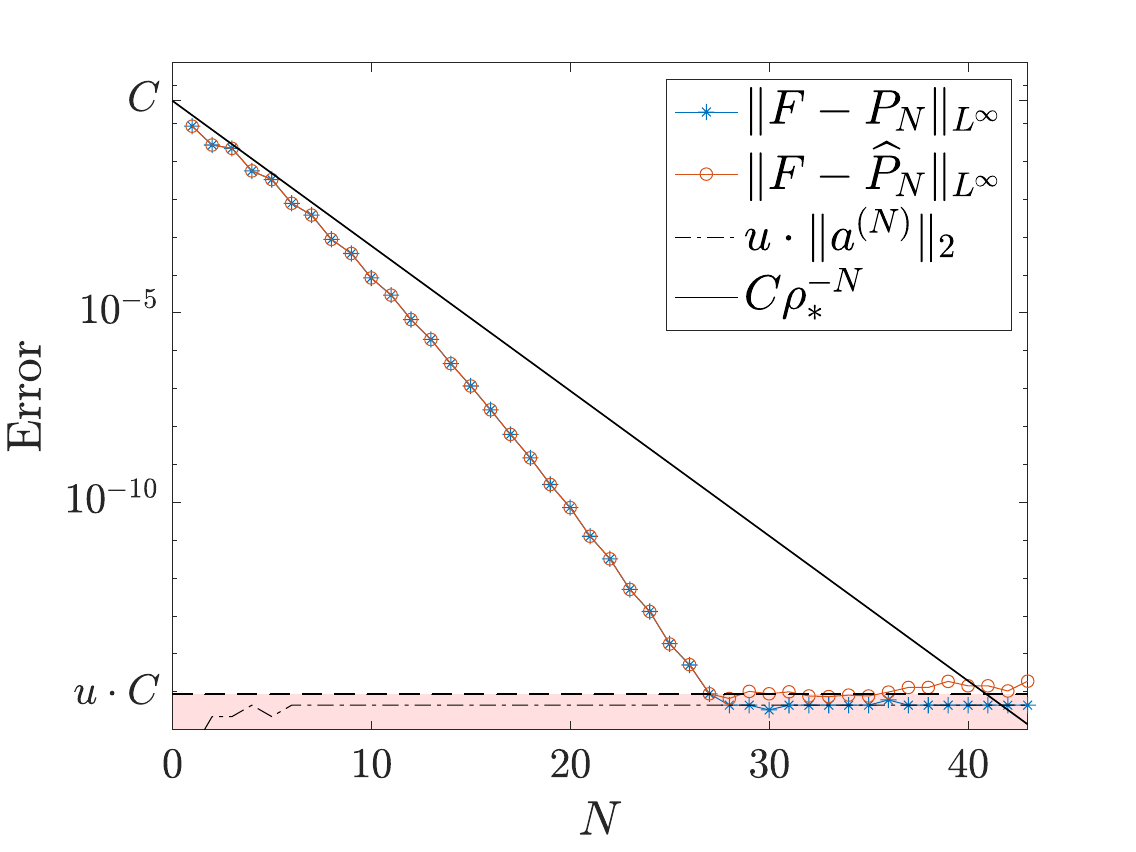}
      \caption{$F(x)=e^{-2(x+0.1)^2}$, $C=4\e{0}$}
    \end{subfigure}
    \begin{subfigure}{0.49\textwidth}
      \centering
      \includegraphics[width=\textwidth]{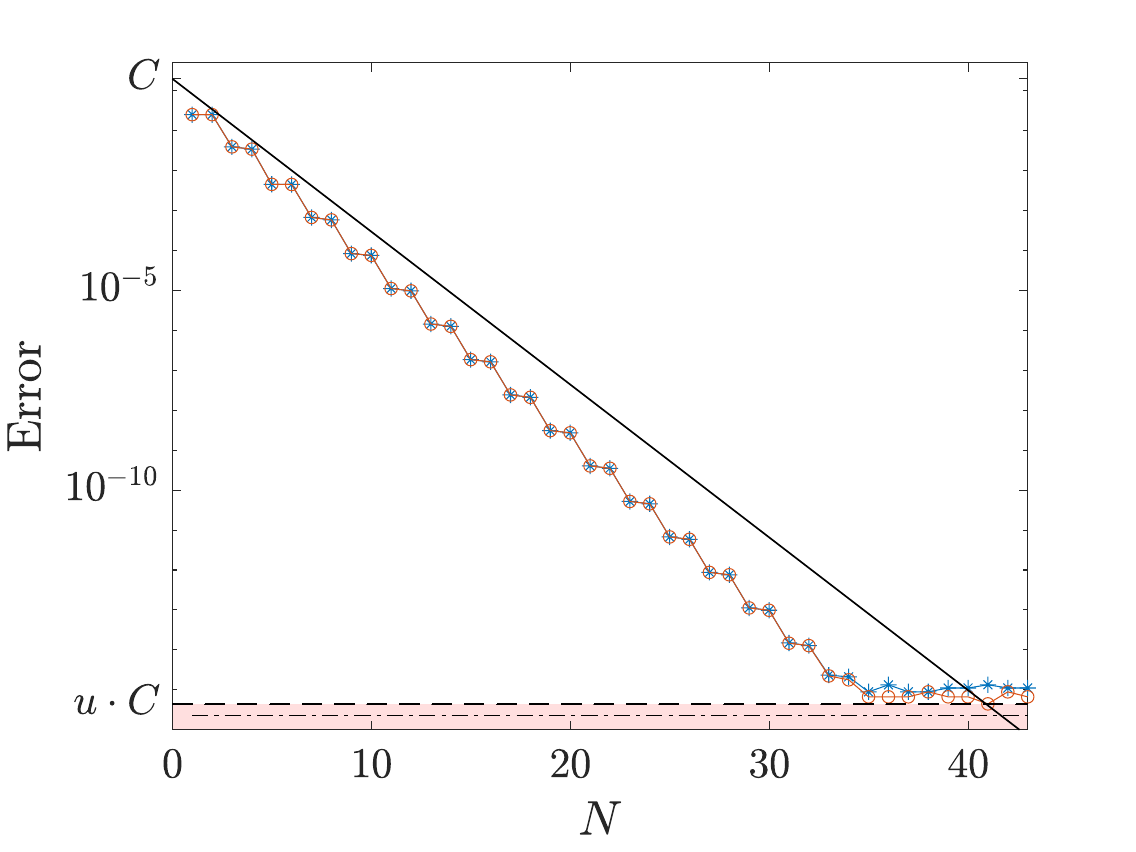}
      \caption{$F(x)=\tan(x)$, $C=2\e{0}$}
    \end{subfigure}
    \begin{subfigure}{0.49\textwidth}
      \centering
      \includegraphics[width=\textwidth]{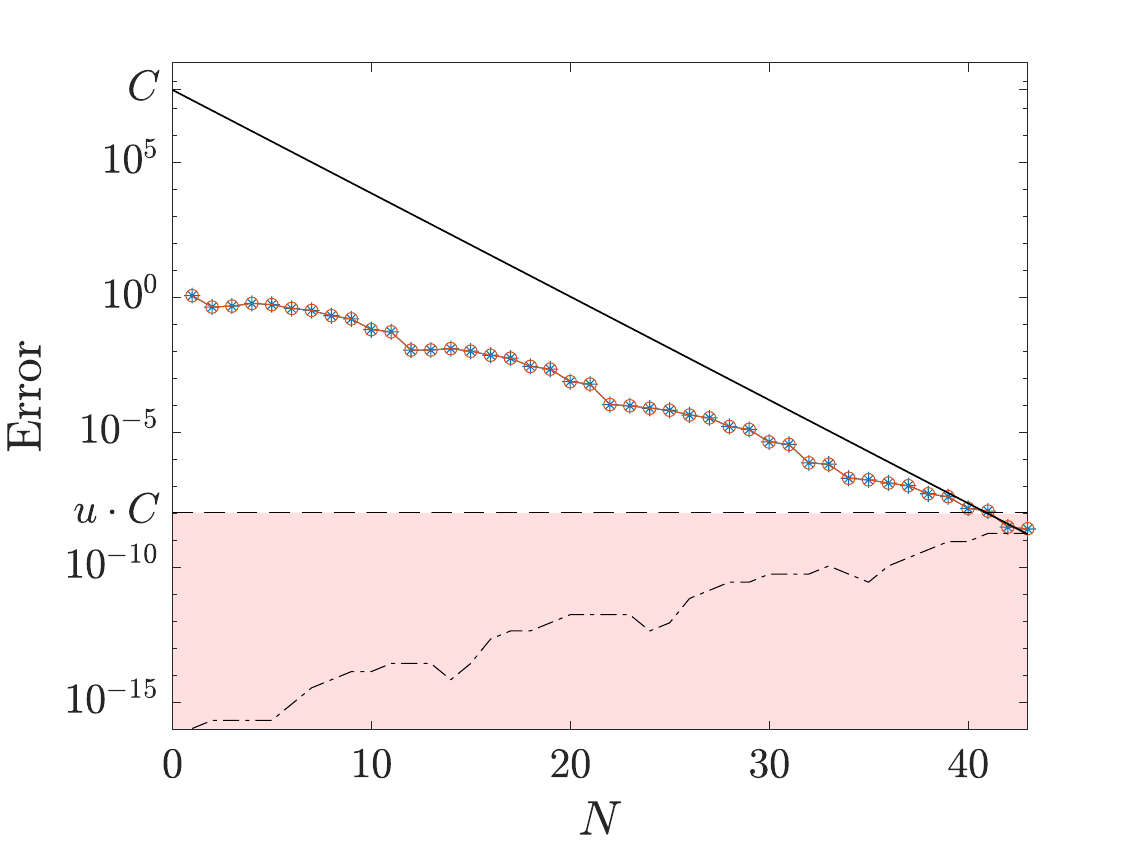}
      \caption{$F(x)=\cos(3x^8+1)$, $C=5\e{7}$}
    \end{subfigure}
    \begin{subfigure}{0.49\textwidth}
      \centering
      \includegraphics[width=\textwidth]{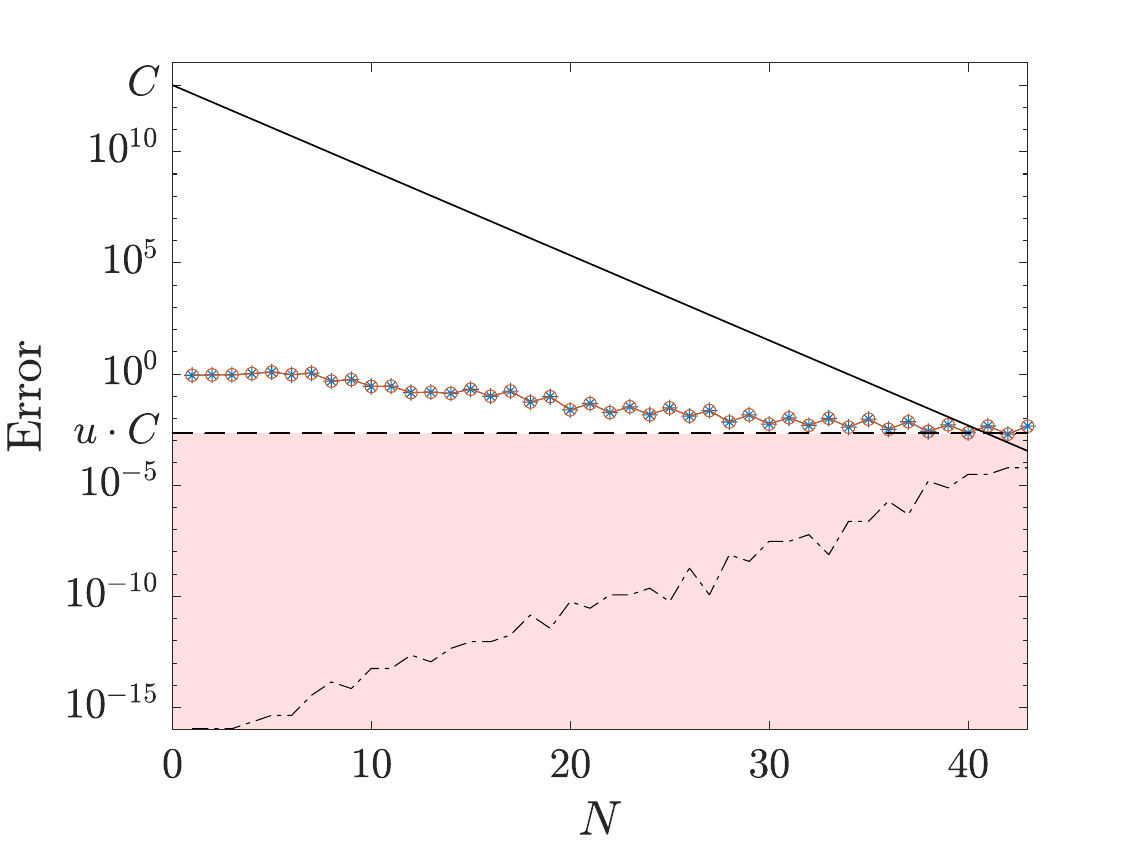}
      \caption{$F(x)=\abs{\sin(5x)}^3$, $C=1\e{13}$}
    \end{subfigure}
  \begin{subfigure}{0.49\textwidth}
      \centering
      \includegraphics[width=\textwidth]{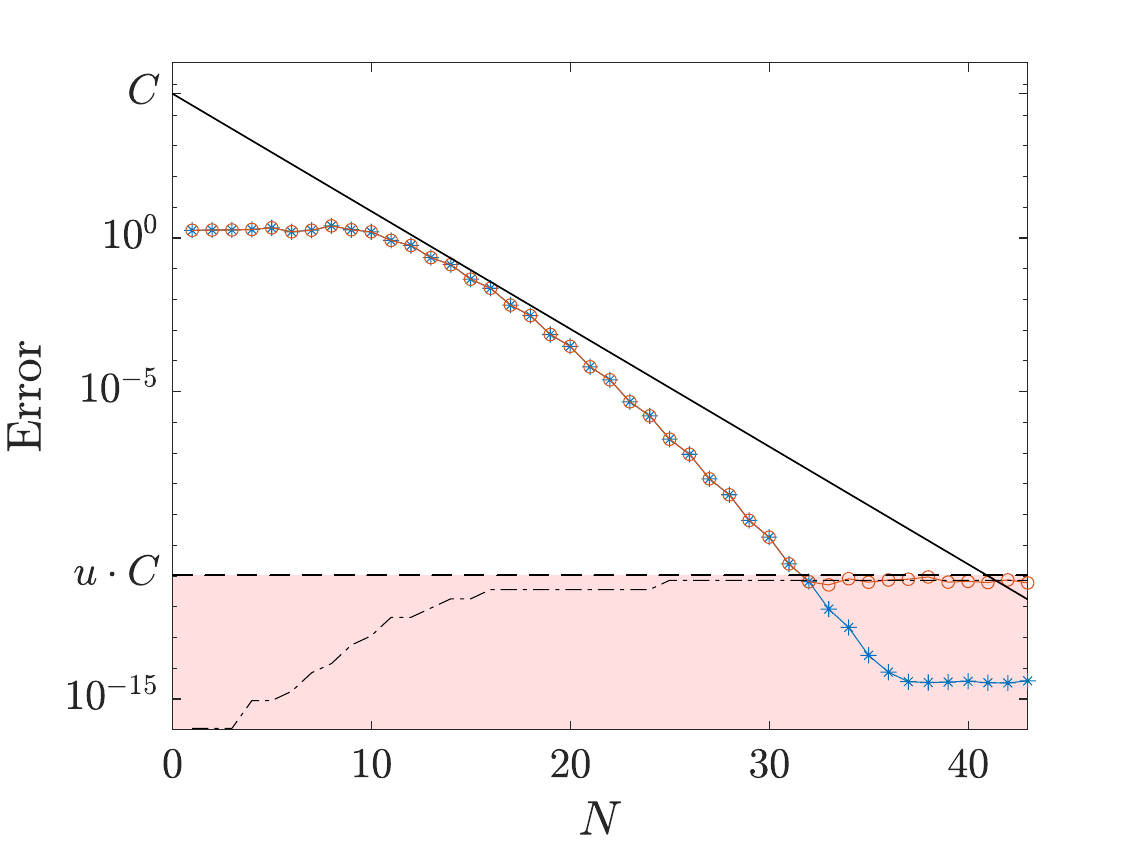}
      \caption{$F(x)=\cos(12x+1)$, $C=5\e{4}$}
    \end{subfigure}
    \begin{subfigure}{0.49\textwidth}
      \centering
      \includegraphics[width=\textwidth]{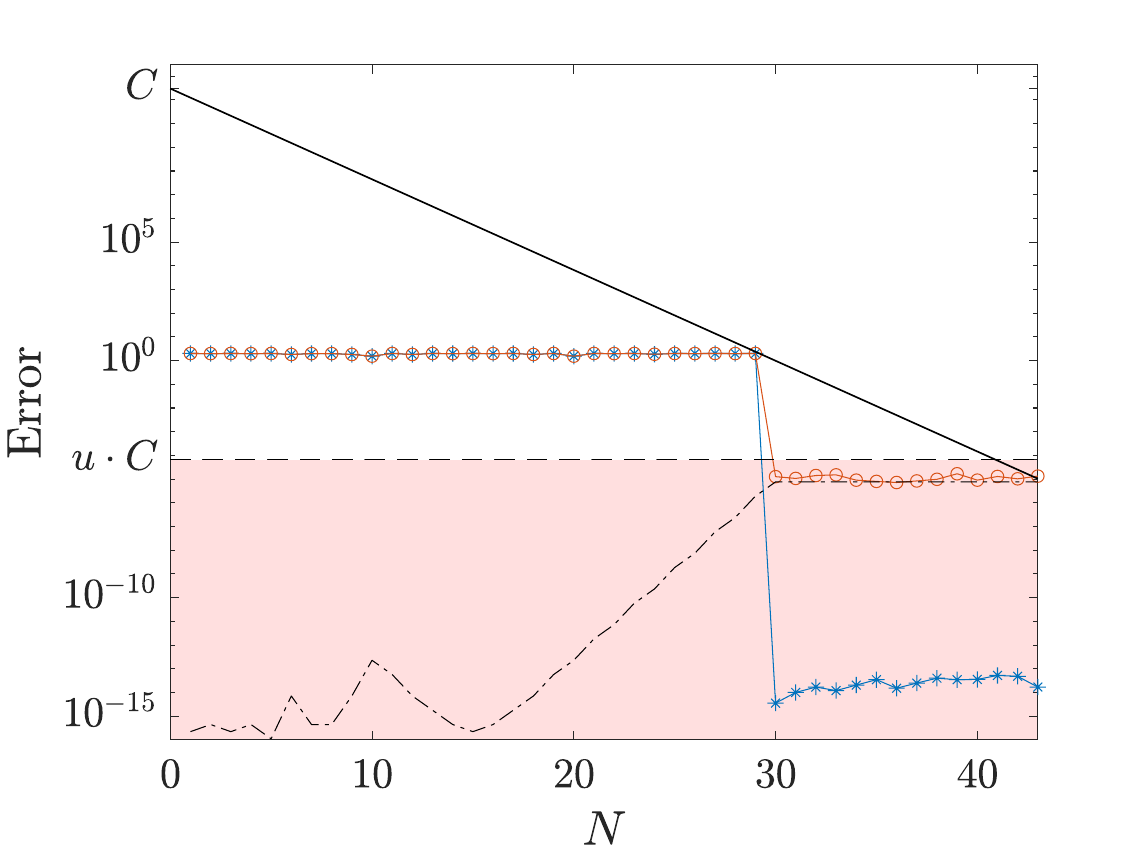}
      \caption{$F(x)=T_{30}(x)$, $C=3\e{11}$}
    \end{subfigure}
      \caption{{\bf Polynomial interpolation in the monomial basis over
      $\Omega=[-1,1]$}.  
      The constant $\rho_*$ equals $1+\sqrt{2}$. The
      value of $C$ is chosen such that $\norm{F-P_N}_{\li([-1,1])}\leq
      C\rho_*^{-N}$ for $N=0,1,...,40$. We highlight the region bounded
      above by $u\cdot C$ in pink.}
  \label{fig:888}
\end{figure}

\begin{figure}
    \centering
    \begin{subfigure}{0.49\textwidth}
      \centering
      \includegraphics[width=\textwidth]{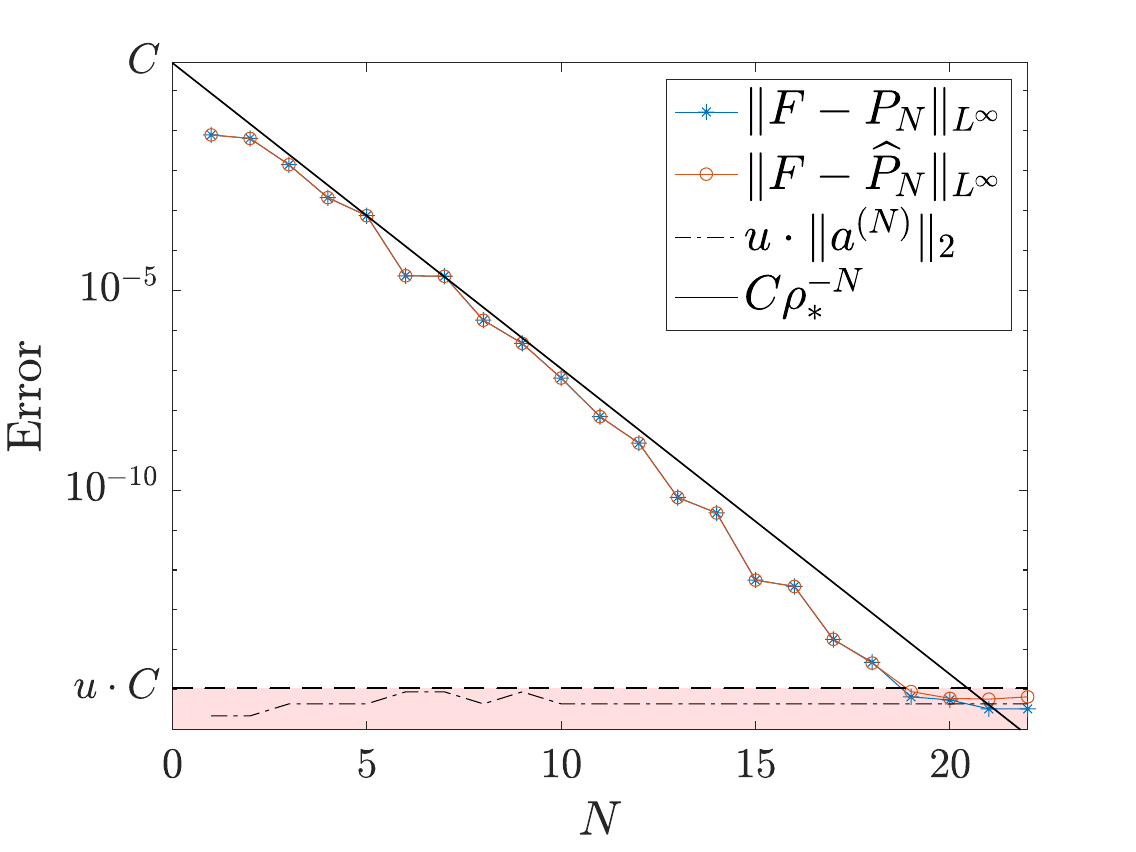}
      \caption{$F(x)=e^{-2(x+0.1)^2}$, $C=5\e{0}$}
    \end{subfigure}
    \begin{subfigure}{0.49\textwidth}
      \centering
      \includegraphics[width=\textwidth]{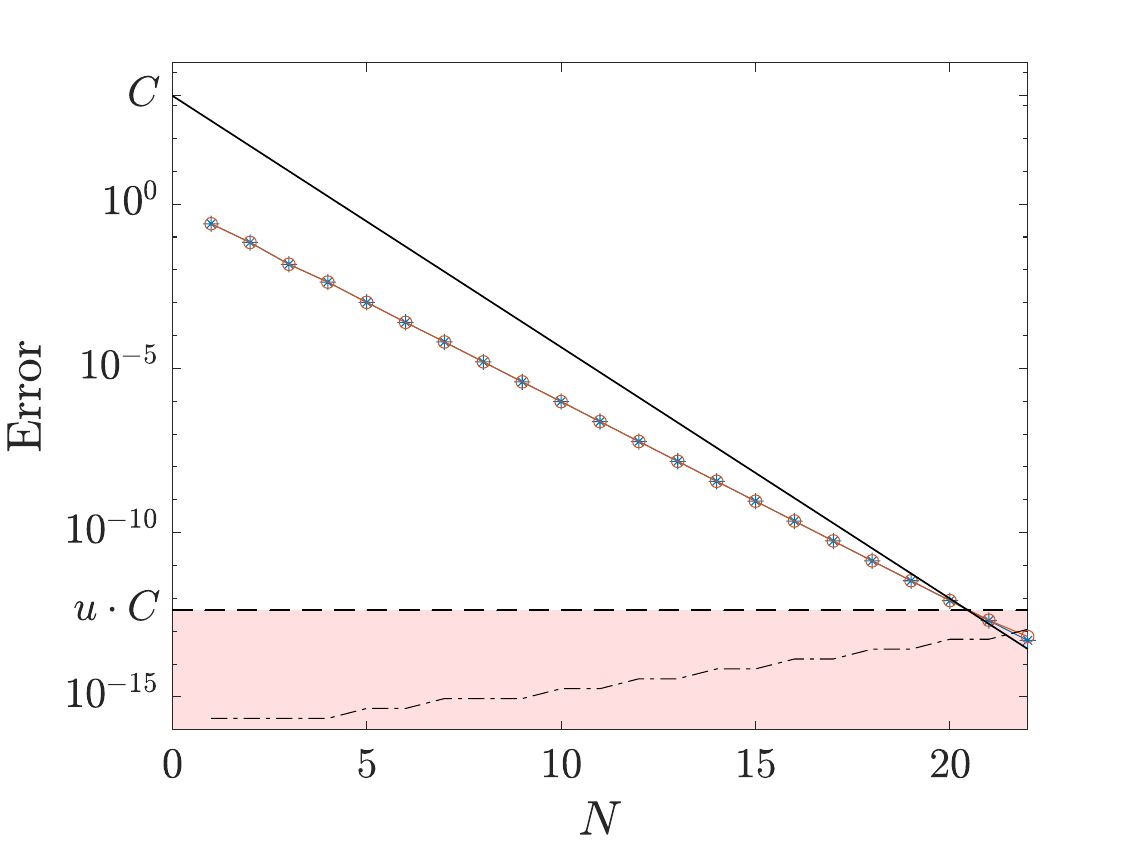}
      \caption{$F(x)=\tan(x)$, $C=2\e{3}$}
    \end{subfigure}
    \begin{subfigure}{0.49\textwidth}
      \centering
      \includegraphics[width=\textwidth]{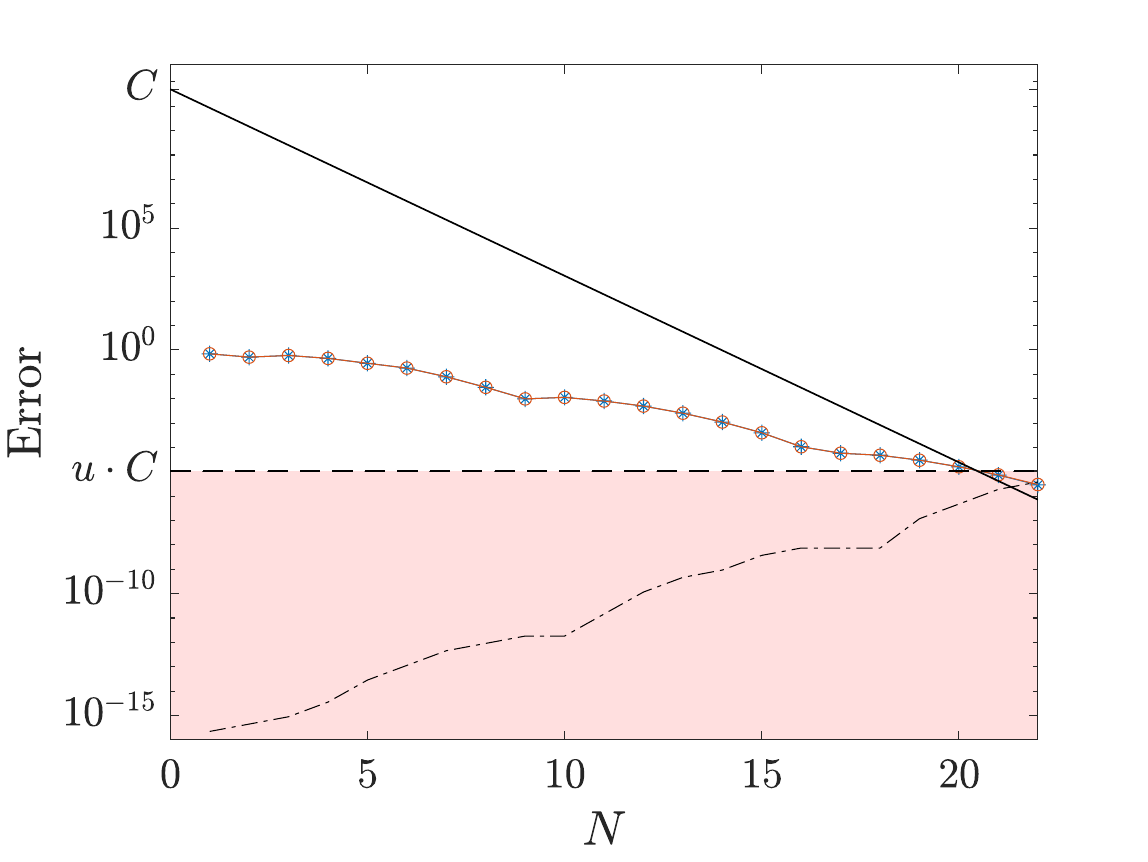}
      \caption{$F(x)=\cos(3x^8+1)$, $C=5\e{10}$}
    \end{subfigure}
    \begin{subfigure}{0.49\textwidth}
      \centering
      \includegraphics[width=\textwidth]{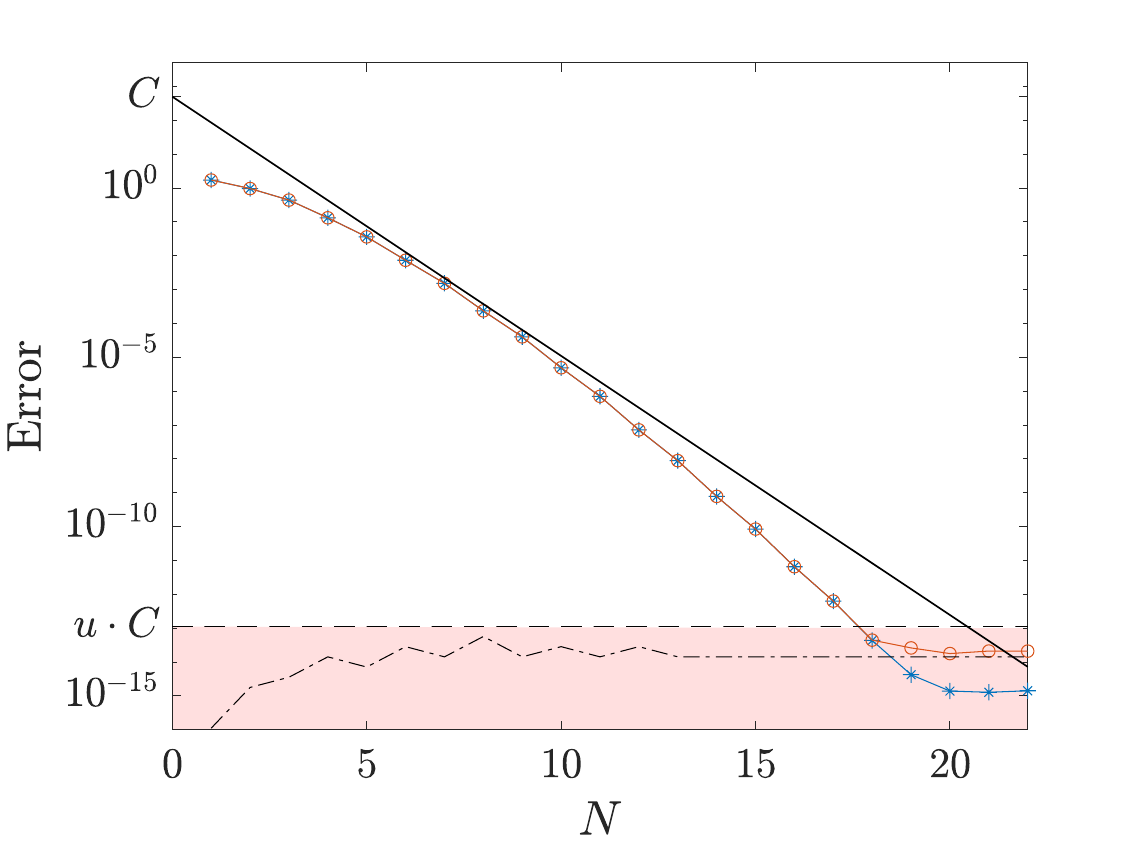}
      \caption{$F(x)=\sin(6x+1)$, $C=5\e{2}$}
    \end{subfigure}
      \caption{{\bf Polynomial interpolation in the monomial basis over
      $\Omega=[0,1]$}. 
      The constant $\rho_*$ equals $3+2\sqrt{2}$. The
      value of $C$ is chosen such that $\norm{F-P_N}_{\li([0,1])}\leq
      C\rho_*^{-N}$ for $N=0,1,...,20$. We highlight the region bounded
      above by $u\cdot C$ in pink.}
  \label{fig:999}
\end{figure}

\begin{figure}
    \centering
    \begin{subfigure}{0.49\textwidth}
      \centering
      \includegraphics[width=\textwidth]{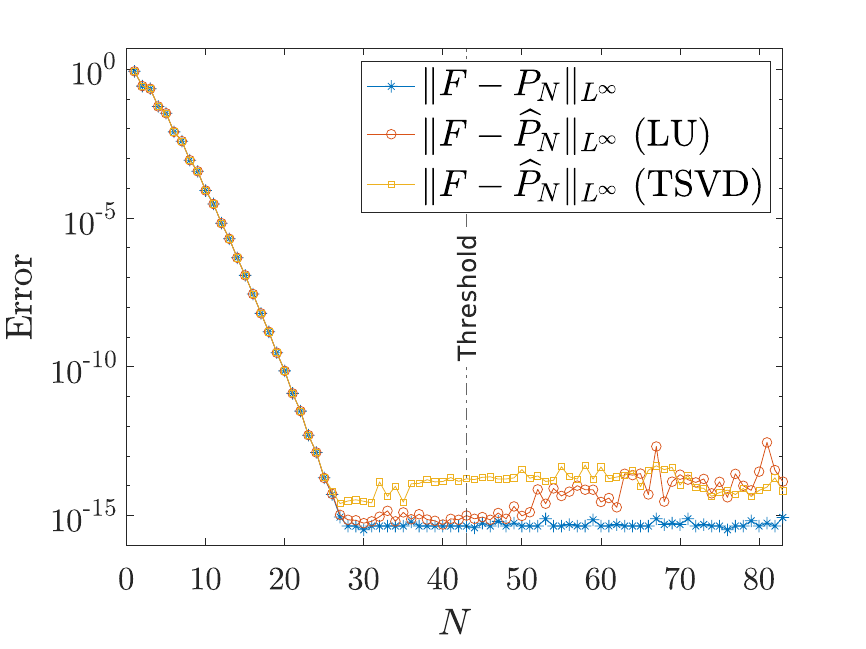}
      \caption{$F(x)=e^{-2(x+0.1)^2}$}
    \end{subfigure}
  \begin{subfigure}{0.49\textwidth}
      \centering
      \includegraphics[width=\textwidth]{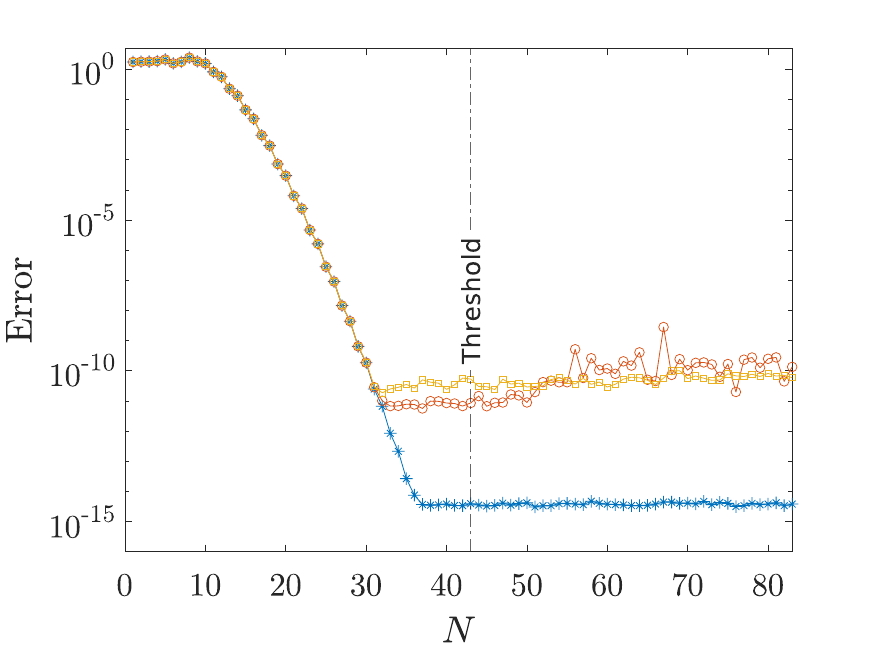}
      \caption{$F(x)=\cos(12x+1)$}
    \end{subfigure}
    \begin{subfigure}{0.49\textwidth}
      \centering
      \includegraphics[width=\textwidth]{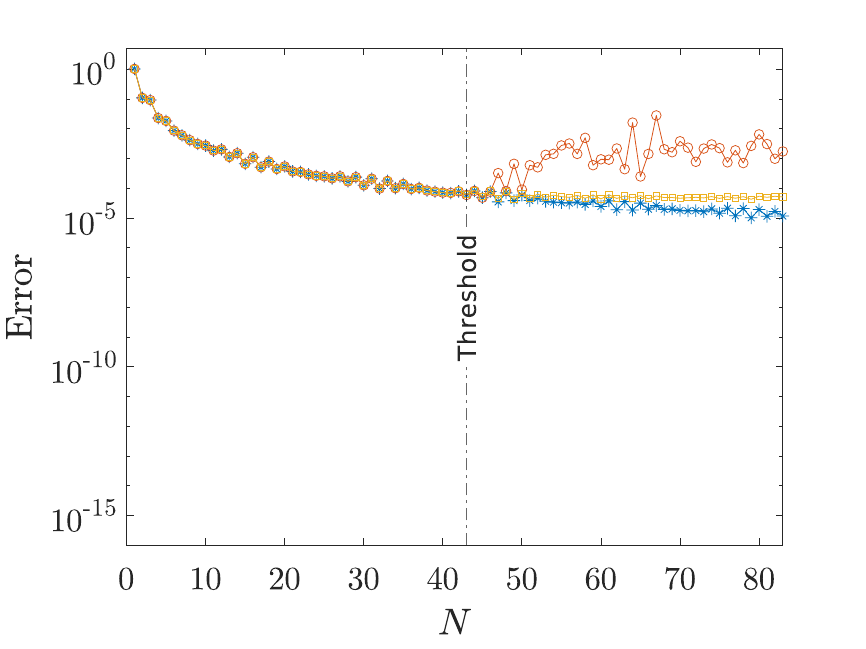}
      \caption{$F(x)=\abs{x+0.1}^{2.5}$}
    \end{subfigure}
    \begin{subfigure}{0.49\textwidth}
      \centering
      \includegraphics[width=\textwidth]{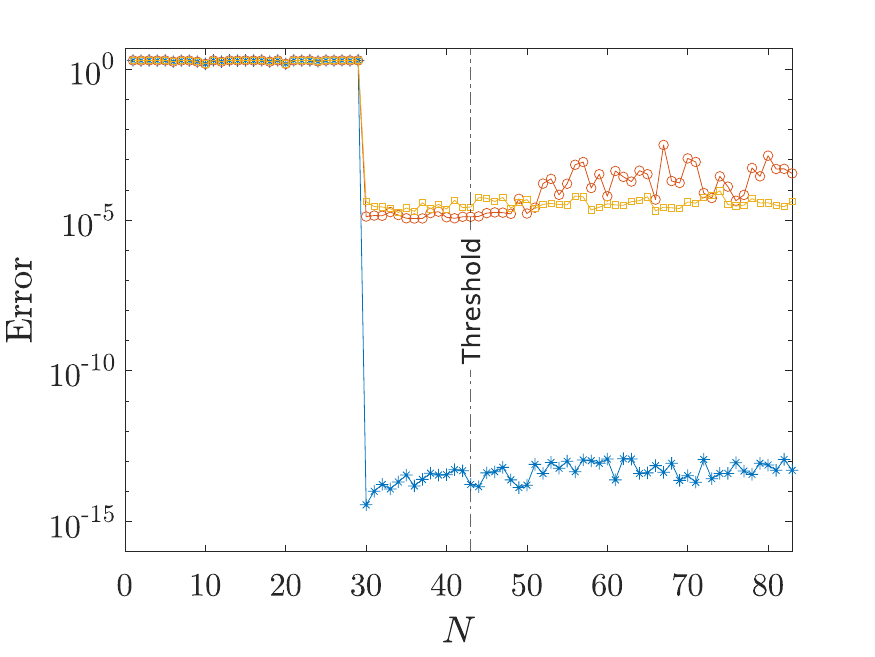}
      \caption{$F(x)=T_{30}(x)$}
    \end{subfigure}
      \caption{{\bf Polynomial interpolation in the monomial basis
      over $\Omega=[-1,1]$, when the order of approximation exceeds the
      threshold}. The vertical line with the label `Threshold' denotes the
      order of approximation $\min \bigl(\{N:\norm{(V^{(N)})^{-1}}_2\geq
      \frac{1}{u}\}\bigr)$, beyond which our theory is no longer applicable.}
      \label{fig:888_exceed}
\end{figure}

\subsection{Approximation over more general regions in the complex plane}
  \label{sec:mono_curv}

In this section, we consider polynomial interpolation and approximation in
the monomial basis over a more general simply connected compact set
$\Omega\subset\C$.  When $\partial \Omega$ is an analytic Jordan curve, it
is shown in \cite{fejar1} that the Lebesgue constant of the Fej\'er points
$\{\Phi^{-1}\bigl(\exp(i (\frac{2\pi j}{N+1}+a))\bigr)\}_{j=0,1,\dots, N}$
grows logarithmically, where $\Phi^{-1}$ is defined in Remark \ref{rem:conform}, and
where $a\in \R$ is arbitrary.  When $\Omega$ is a Jordan arc, it is shown in
\cite{afejar} that the Lebesgue constant of the Fej\'er points, with some
adjustment to ensure adequate spacing between interpolation points, also
exhibits logarithmic growth.

In the remainder of this section, we provide several numerical experiments
involving the following three domains: an ellipse with a major radius of
1 and a minor radius of $0.2$, centered at the origin (see Figure
\ref{fig:erho_elli}); a parabola parameterized by $g:[-1,1]\to\C$, where
$g(t):=t+0.4i(t^2-1)$ (see Figure \ref{fig:erho_para}); and a square with
side length $\sqrt{2}$, centered at the origin.  We note that the estimated
constants $\rho_*$ (see Definition \ref{def:bern}) for the three domains are
approximately $2$, $2.6$, and $1.3$, respectively (see Figures
\ref{fig:erho_elli}, \ref{fig:erho_para}, and \ref{fig:erho_box}).

In the first domain (the ellipse), the conformal mapping $\Phi^{-1}$ is
specified by the formula $\Phi^{-1}(z)=\frac{3}{5}\cdot
z+\frac{2}{5}\cdot\frac{1}{z}$, and we use the Fej\'er points for
interpolation.  Since the Fej\'er points for the second and the third
domains are not known in closed form, we employ different approaches tailored
to each case.
In the second domain (the parabola), we set the interpolation
points to be $\{g(t_j)\}_{j=0,1,\dots,N}$, where $\{t_j\}_{j=0,1,\dots,N}$
is the set of $(N+1)$ Chebyshev points on the interval $[-1,1]$. We find
that the Lebesgue constants associated with these points are all
less than 10 when $\kappa(V^{(N)})\lesssim \frac{1}{u}$. In the third domain
(the square), we instead solve a least-squares fitting problem with sampling
points chosen to be the $2(N+1)$ Chebyshev points of the first kind along
each side of the square, where $N$ denotes the order of approximation (see
Remark \ref{rem:leastsq_mono}).

In Figure \ref{fig:cond_general}, we report the values of
$\norm{(V^{(N)})^{-1}}_2$ and its upper bound $\rho_*^N\Lambda_N$ (see
Theorem \ref{thm:cond2}) for the ellipse domain and the parabola domain,
where the value of $\rho_*$ for each domain is estimated based on the plots
in Figure \ref{fig:erho}. Again, one can see that our upper bound is very
tight.

\begin{figure}
    \centering
  \begin{subfigure}{0.49\textwidth}
      \centering
      \includegraphics[width=\textwidth]{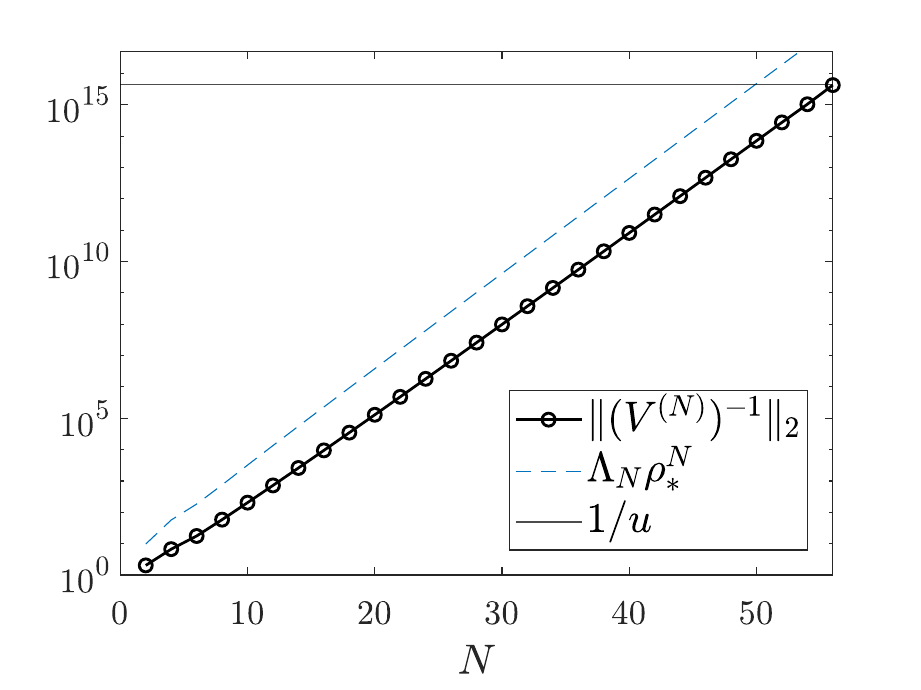}
      \caption{The ellipse, $\rho_*\approx 2$}
    \end{subfigure}
    \begin{subfigure}{0.49\textwidth}
      \centering
      \includegraphics[width=\textwidth]{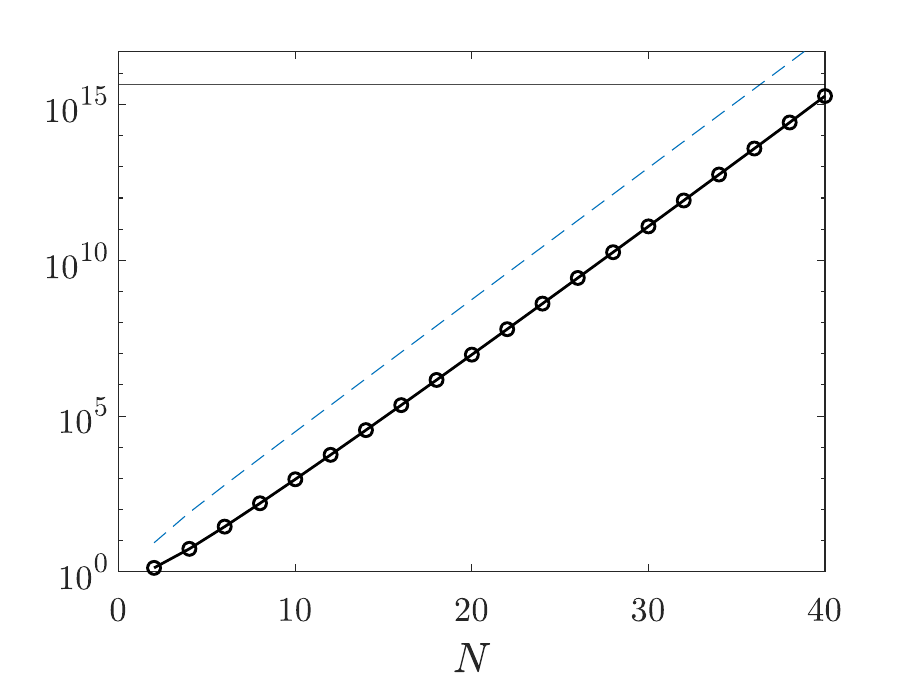}
      \caption{The parabola, $\rho_*\approx 2.6$}
    \end{subfigure}
\caption{{\bf The 2-norm of $(V^{(N)})^{-1}$ and its upper bound}.  }
  \label{fig:cond_general}
\end{figure}

In Figures \ref{fig:elli_exp}, \ref{fig:para_exp}, and \ref{fig:box_exp}, we
report the estimated values of $\norm{F-P_N}_{\li(\Omega)}$, $\norm{F-\hat
P_N}_{\li(\Omega)}$, and  $u\cdot \norm{a^{(N)}}_{2}$, over each of these
three domains.  We estimate the exact polynomial interpolant $P_N$  by
expressing $P_N$ in a discrete orthogonal polynomial basis (which is
well-conditioned), computed using the Vandermonde with Arnoldi
orthogonalization procedure \cite{arnoldi}.  In addition, we report the the
upper bound $C\rho_*^{-N}$ for $\norm{F-P_N}_{\li(\Omega)}$, where $C$ is
chosen so that $\norm{F-P_N}_{\li(\Omega)}\leq C\rho_*^{-N}$ for a range of
$N$ indicated in the figure caption, and the upper bound $u\cdot C$ for
$u\cdot \norm{a^{(N)}}_{2}$. These results demonstrate that the observations
made at the near end of Section \ref{sec:mono_int} are indeed applicable to more
general domains in the complex plane, and to polynomial approximation using
least-squares fitting with monomials.

\begin{figure}
    \centering
  \begin{subfigure}{0.49\textwidth}
      \centering
      \includegraphics[width=\textwidth]{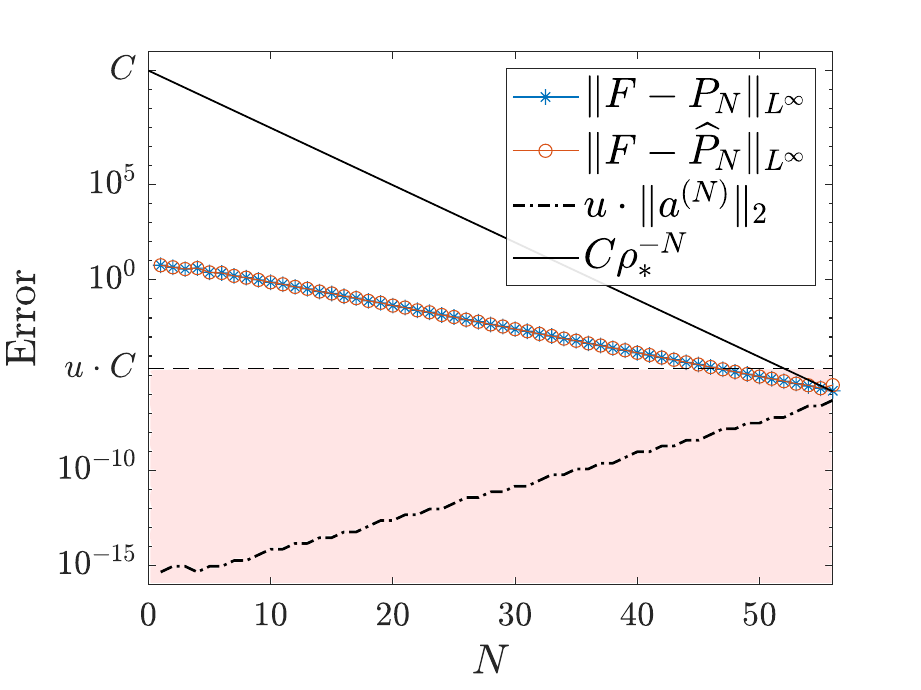}
      \caption{$F(z)=\frac{1}{z-1.1}$, $C=1\e{11}$}
    \end{subfigure}
    \begin{subfigure}{0.49\textwidth}
      \centering
      \includegraphics[width=\textwidth]{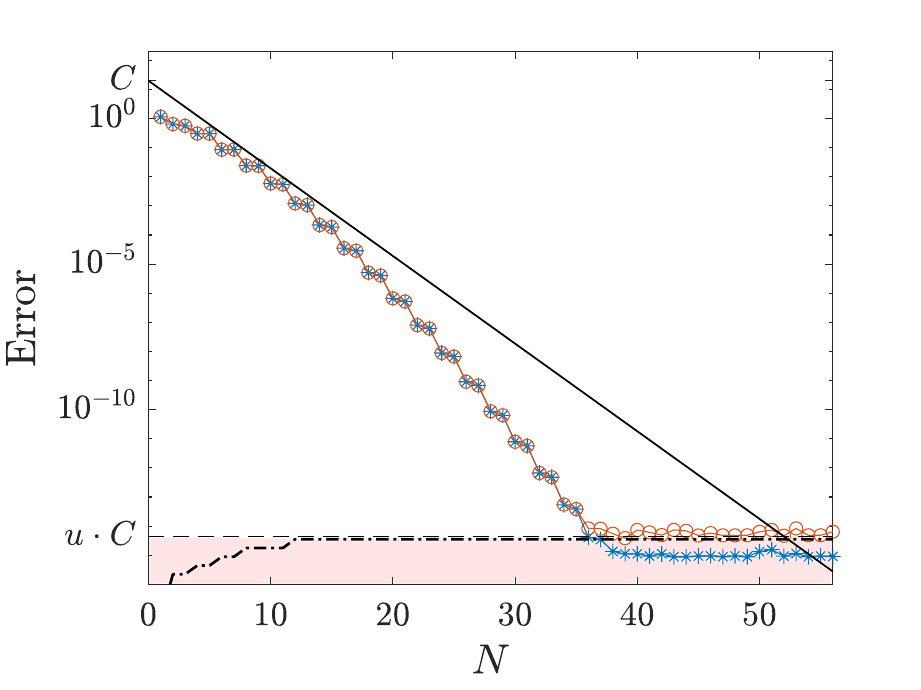}
      \caption{$F(z)=e^{-4z^2}$, $C=3\e{4}$}
    \end{subfigure}
    \begin{subfigure}{0.49\textwidth}
      \centering
      \includegraphics[width=\textwidth]{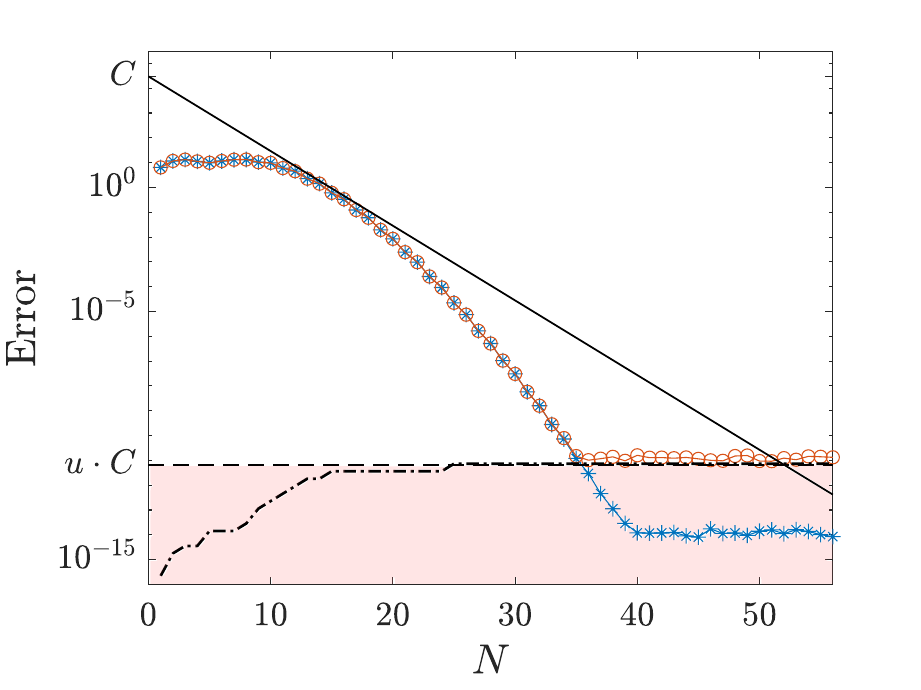}
      \caption{$F(z)=\cos(12z+1)$, $C=2\e{1}$}
    \end{subfigure}
    \begin{subfigure}{0.49\textwidth}
      \centering
      \includegraphics[width=\textwidth]{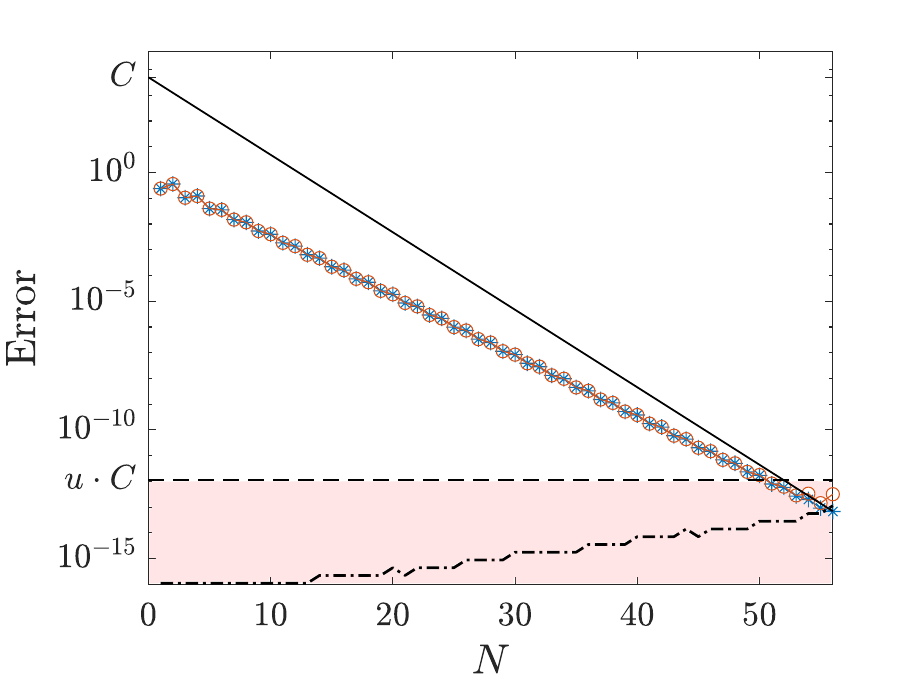}
      \caption{$F(z)=\tan(\tan(z)/2)$, $C=5\e{3}$}
    \end{subfigure}
\caption{{\bf Polynomial interpolation in the monomial basis over the
ellipse~$\Omega$}. We estimate $P_N$ by expressing it in a discrete
orthogonal polynomial basis, computed using the Vandermonde with Arnoldi
orthogonalization procedure \cite{arnoldi}.  The constant $\rho_*$ is
approximately equal to $2.0$.  The value of $C$ is chosen such that
$\norm{F-P_N}_{\li(\Omega)}\leq C\rho_*^{-N}$ for $N=0,1,...,55$. We
highlight the region bounded above by $u\cdot C$ in pink.
}
  \label{fig:elli_exp}
\end{figure}

\begin{figure}
    \centering
  \begin{subfigure}{0.49\textwidth}
      \centering
      \includegraphics[width=\textwidth]{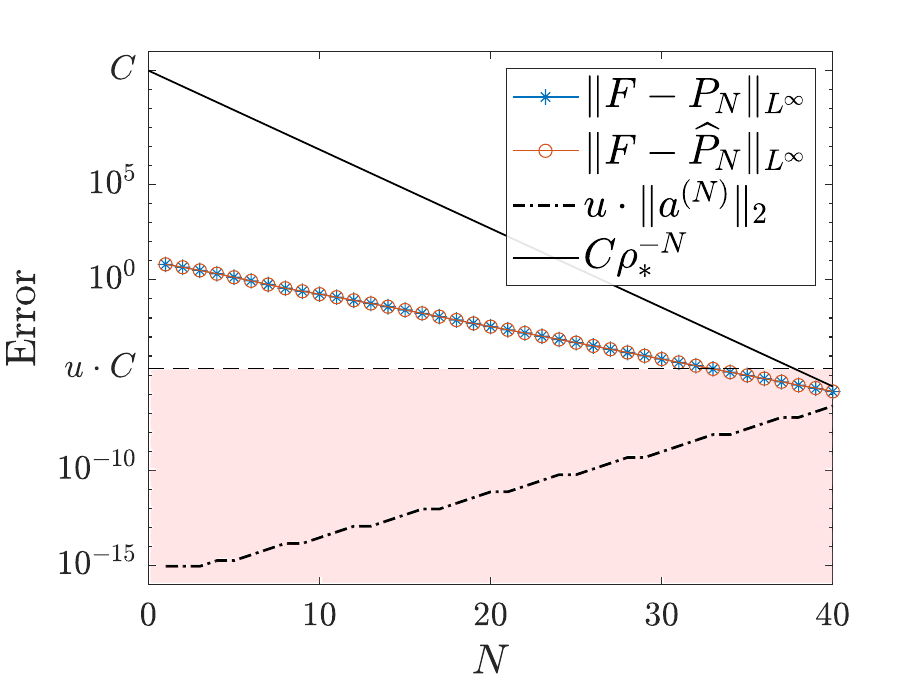}
      \caption{$F(z)=\frac{1}{z-1.1}$, $C=1\e{11}$}
    \end{subfigure}
    \begin{subfigure}{0.49\textwidth}
      \centering
      \includegraphics[width=\textwidth]{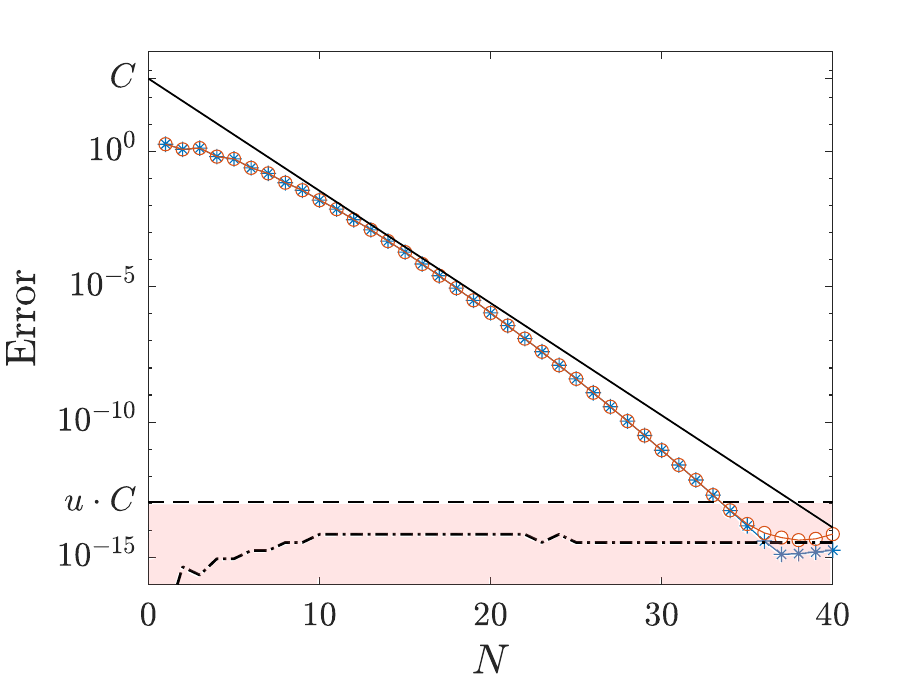}
      \caption{$F(z)=e^{-4z^2}$, $C=3\e{6}$}
    \end{subfigure}
    \begin{subfigure}{0.49\textwidth}
      \centering
      \includegraphics[width=\textwidth]{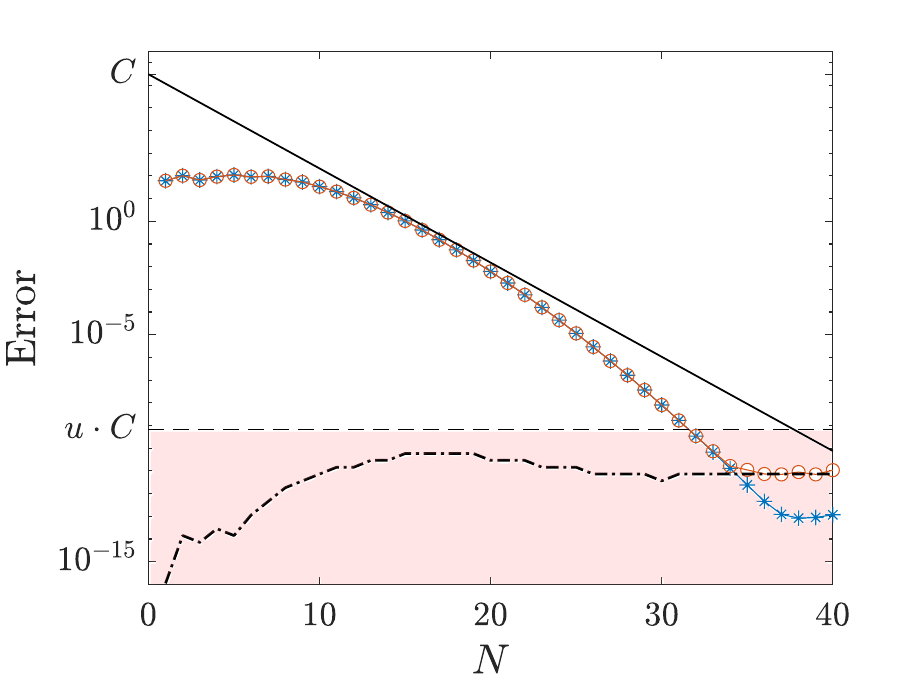}
      \caption{$F(z)=\cos(12z+1)$, $C=5\e{2}$}
    \end{subfigure}
    \begin{subfigure}{0.49\textwidth}
      \centering
      \includegraphics[width=\textwidth]{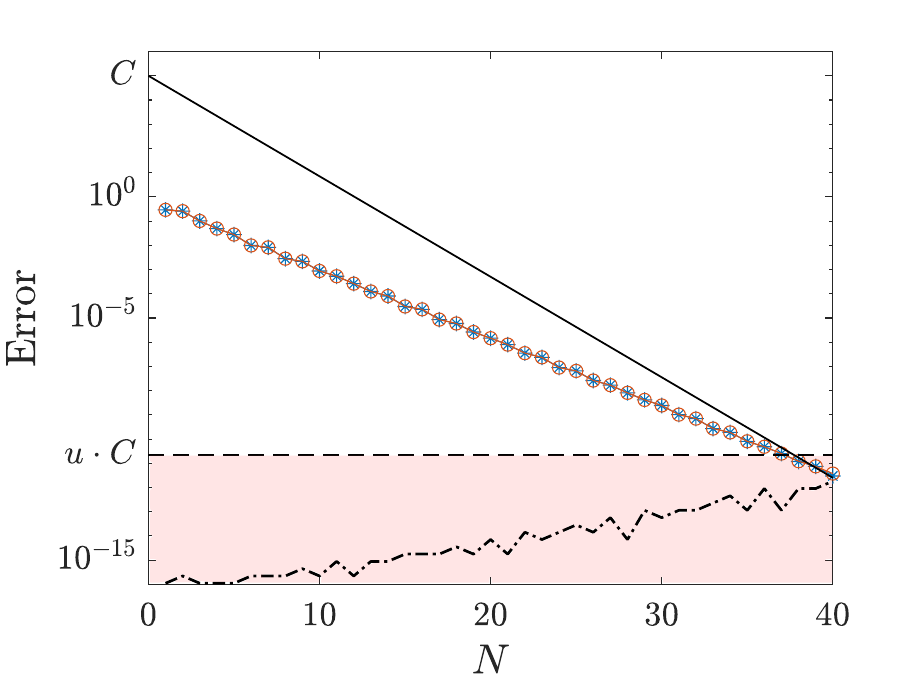}
      \caption{$F(z)=\tan(\tan(z)/2)$, $C=1\e{5}$}
    \end{subfigure}
\caption{{\bf Polynomial interpolation in the monomial basis over the
parabola~$\Omega$}.
We estimate $P_N$ by expressing it in a discrete
orthogonal polynomial basis, computed using the Vandermonde with Arnoldi
orthogonalization procedure \cite{arnoldi}.  
The constant $\rho_*$ is approximately equal to
$2.6$. The value of $C$ is chosen such that $\norm{F-P_N}_{\li(\Omega)}\leq
C\rho_*^{-N}$ for $N=0,1,...,40$. We highlight the region bounded above by
$u\cdot C$ in pink.}

  \label{fig:para_exp}
\end{figure}

\begin{figure}
    \centering
  \begin{subfigure}{0.49\textwidth}
      \centering
      \includegraphics[width=\textwidth]{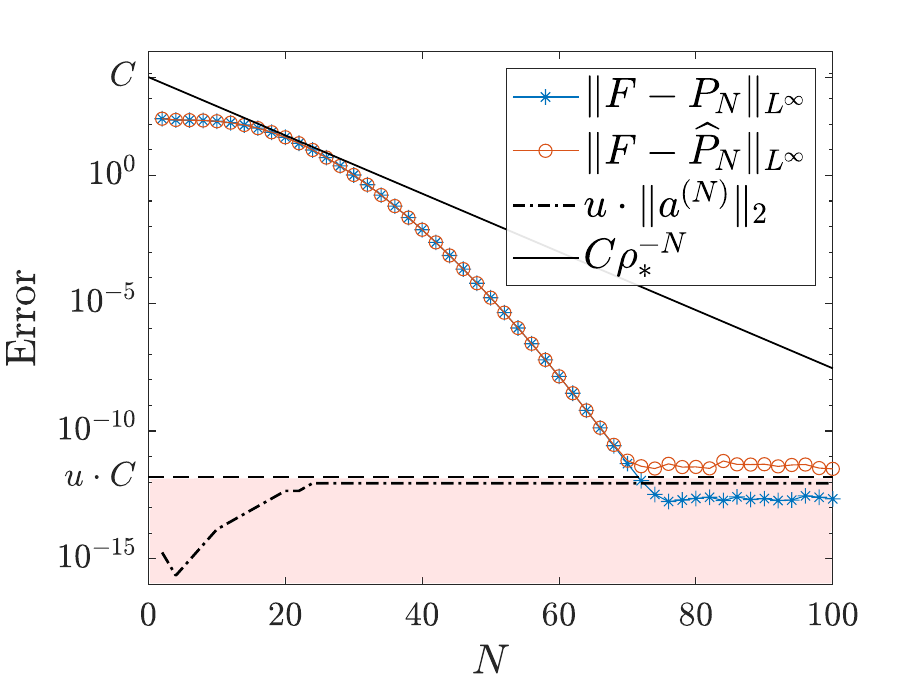}
      \caption{$F(z)=e^{-10z^2}$, $C=7\e{3}$}
    \end{subfigure}
    \begin{subfigure}{0.49\textwidth}
      \centering
      \includegraphics[width=\textwidth]{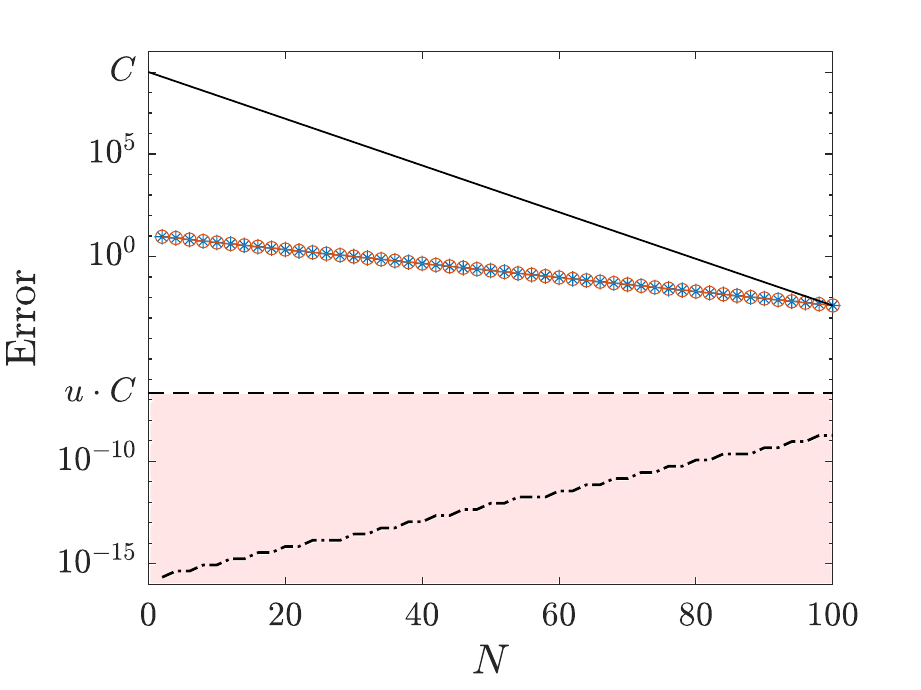}
      \caption{$F(z)=\frac{1}{z-0.8}$, $C=5\e{4}$}
    \end{subfigure}
    \begin{subfigure}{0.49\textwidth}
      \centering
      \includegraphics[width=\textwidth]{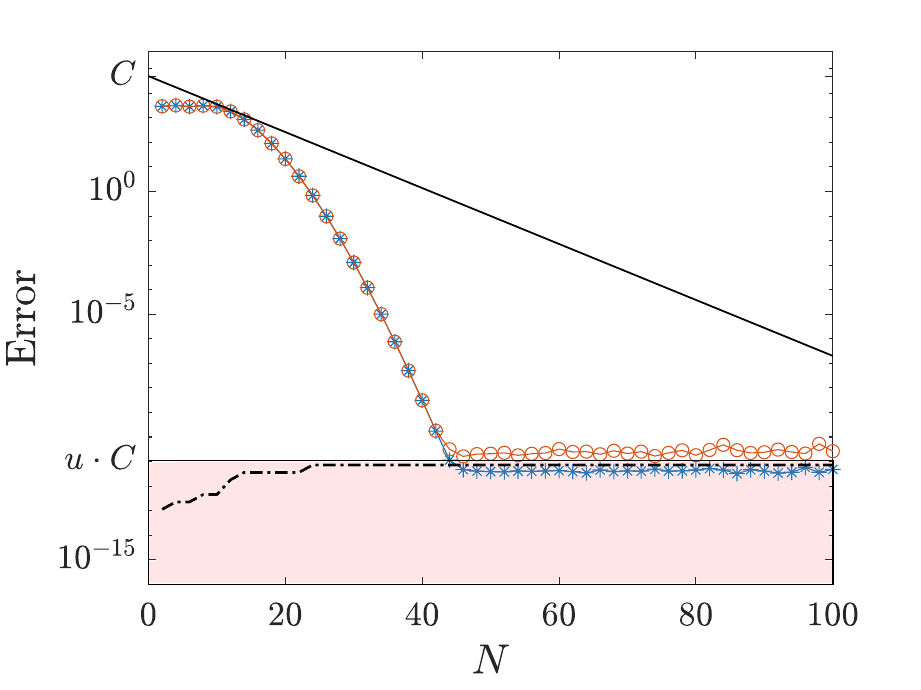}
      \caption{$F(z)=\cos(12z+1)$, $C=1\e{9}$}
    \end{subfigure}
    \begin{subfigure}{0.49\textwidth}
      \centering
      \includegraphics[width=\textwidth]{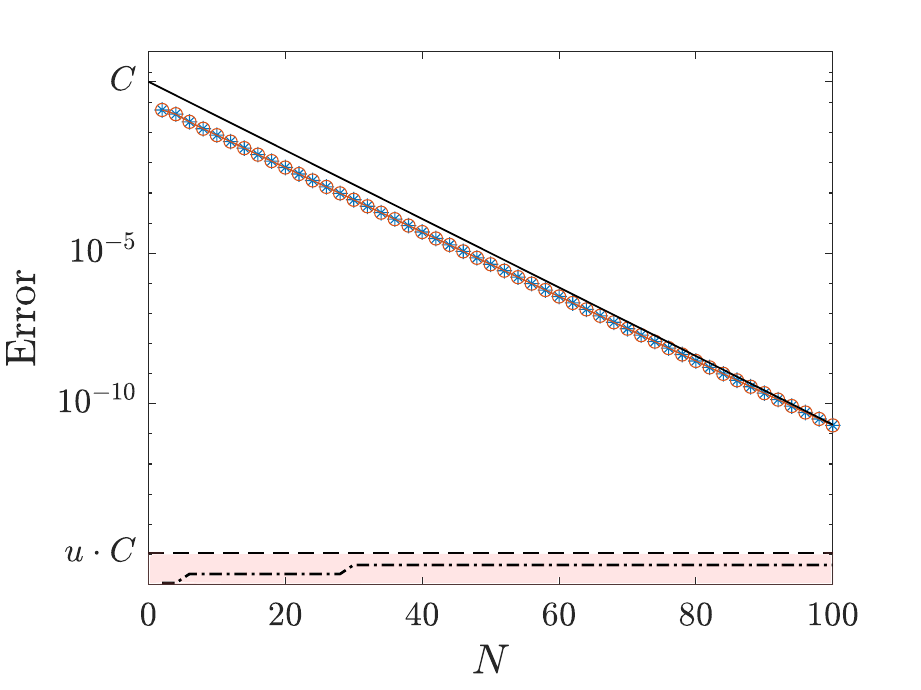}
      \caption{$F(z)=\tan(\tan(z))$, $C=5\e{0}$}
    \end{subfigure}
\caption{{\bf Polynomial approximation in the monomial basis over the
square~$\Omega$}. 
We estimate $P_N$ by expressing it in a discrete orthogonal polynomial
basis, computed using the Vandermonde with Arnoldi orthogonalization
procedure \cite{arnoldi}.  
The constant $\rho_*$ is approximately equal to
$1.3$ (see Figure \ref{fig:erho_box}). The value of $C$ is chosen such that
$\norm{F-P_N}_{\li(\Omega)}\leq C\rho_*^{-N}$ for $N=0,1,...,100$. We
highlight the region bounded above by $u\cdot C$ in pink. }
  \label{fig:box_exp}
\end{figure}

\begin{remark}
In certain applications \cite{helsing,barnett,bobbie}, the function
$F:\Gamma\to\C$ is defined by the formula $F(z):=\sigma(g^{-1}(z))$, where
$\Gamma$ is a simple smooth arc, $g:[-1,1]\to \C$ is an analytic function
that parameterizes $\Gamma$, and $\sigma:[-1,1]\to\C$ is analytic. In this
case, the analytic continuation of $F$ can have a singularity close to
$\Gamma$ even when $\sigma$ is entire, because the inverse of the
parameterization (i.e., $g^{-1}$) has so-called Schwarz singularities at
$z=g(t^*)$, where $g'(t^*)=0$. In \cite{barnett}, the authors show that, the
higher the curvature of the arc $\Gamma$, the closer the singularity induced
by $g^{-1}$ is to $\Gamma$. As a result, the approximation of such a
function $F$ by polynomials is efficient only when the curvature of $\Gamma$
is small.
\end{remark}

\section{Applications}
  \label{sec:app}
After justifying the use of a monomial basis for polynomial interpolation, a
natural question to ask is why one would want to use it in the first place.
For one, the monomial basis is the simplest polynomial basis to
manipulate.  For example, the evaluation of an $N$th degree polynomial
expressed in the monomial basis can be achieved using only $N$
multiplications through the application of Horner's rule. 
This evaluation can be further accelerated using Estrin's scheme
\cite{estrin}, which leverages instruction-level parallelism in modern CPUs.
This type of parallelism is typically unavailable when evaluating a
polynomial expressed in an orthogonal polynomial basis.

Besides this obvious advantage, we discuss some other applications below.

\subsection{Oscillatory integrals and singular integrals}
Given an oscillatory or singular function $\Psi:\Gamma\to \C$ and a smooth
function $F:\Gamma\to \C$ over a smooth simple arc $\Gamma\subset\C$,
the calculation of 
$\int_\Gamma \Psi(z) F(z) \d z$
by standard quadrature rules can be extremely expensive or inaccurate due to
the oscillations or the singularity of $\Psi$. However, when
$F$ is a monomial, there exists a wide range of integrals in this form 
that can be efficiently computed to high accuracy by
either analytical formulas or by recurrence relations, often derived using
integration by parts.  Therefore, when the smooth function $F$ is accurately
approximated by a monomial expansion of order $N$, such integrals can be
efficiently evaluated by the formula
$\sum_{k=0}^N a_k \Bigl(\int_\Gamma\Psi(z)z^k\d z\Bigr)$,
where $\{a_k\}_{k=0,1,\dots,N}$ denotes the coefficients of the monomial
expansion. Integrals of this type include the Fourier integral
$\int_a^b e^{i c x} F(x)\d x$ and various
layer potentials and Stieltjes transforms, e.g., $\int_{\Gamma}\log(z-\xi)
F(z) \d z$ and $\int_{\Gamma} \frac{F(z)}{z-\xi}\d z$, where $\xi\in\C$ is
given.  We refer the readers to \cite{iserles,iserles2} for more detailed
discussion on the Fourier integral, and to \cite{helsing,barnett,bobbie} for
more detailed discussion on the application of polynomial interpolation in
the monomial basis to the evaluation of layer potentials. Some interesting
applications can also be found in \cite{jiang,navier,tornberg}.

\subsection{Root finding}
Given a simply connected compact set $\Omega\subset\C$ and a function
$F:\Omega\to\C$, one method for computing the roots of $F$ over $\Omega$ is
to first approximate $F$ by a polynomial $P_N(z)=\sum_{j=0}^Na_j z^j$ to high
accuracy, and then to compute the roots of $P_N$ by calculating the
eigenvalues of the corresponding companion matrix, which we denote
by~$C(P_N)$.  Recently, a backward stable algorithm that computes the
eigenvalues of $C(P_N)$ in $\O(N^2)$ operations with $\O(N)$ storage has
been proposed in \cite{root}. This algorithm is backward stable in the sense
that the computed roots are the exact roots of a perturbed polynomial $\hat
P_N(z)=\sum_{j=0}^N (a_j+\delta a_j) z^j$, so that the backward error satisfies
$\norm{\delta a^{(N)}}_2 \lesssim u \norm{a^{(N)}}_2$, where $u$
denotes machine epsilon, $\delta a^{(N)}:=(\delta a_0,\delta a_1,\dots,
\delta a_N)^T$ and $a^{(N)}:=(a_0,a_1,\dots, a_N)^T$. It follows that 
$\norm{P_N-\hat P_N}_{\li(\Omega)} \leq \norm{\delta a^{(N)}}_1 \lesssim
u\sqrt{N+1} \norm{a^{(N)}}_2$.
When $\norm{a^{(N)}}_2 \approx \norm{P_N}_{\li(\Omega)}$, the computed roots
are backward stable in the polynomial $P_N$. This condition, however, does
not hold for all polynomials $P_N$. Furthermore, the calculation of the
coefficients $a^{(N)}$ from the function $F$, which involves the solution of
a Vandermonde system of equations, is highly ill-conditioned.
In this paper, we show that, when $F$ is sufficiently smooth (in the mild
sense discussed in Section \ref{sec:good}), it is possible to compute the
coefficients of an interpolating polynomial $P_N(z)=\sum_{j=0}^Na_j z^j$,
with $\norm{a^{(N)}}_2 \approx \norm{F}_{L^\infty(\Omega)}$, which
approximates $F$ uniformly to high accuracy, even when the condition number
of the Vandermonde matrix is close to the reciprocal of machine epsilon.
From this, we see that a backward stable root finder can be constructed by
combining the piecewise polynomial approximation procedure described in
Section \ref{sec:limit} with the algorithm presented in~\cite{root}.
Interestingly, by using the least-squares fitting procedure described in
Section \ref{sec:mono_curv} for the case when $\Omega$ is a square, an
adaptive root-finding algorithm for a complex analytic function in a square
domain, analogous to the one described in \cite{hanwen}, can be developed
utilizing solely the monomial basis.

\section{Discussion}
  \label{sec:dis}

Since the invention of digital computers, most research on the topic of
polynomial interpolation in the monomial basis has focused on showing that it
is a bad idea.  The condition number of Vandermonde matrices has been
studied extensively in recent decades (see~\cite{walter1} for a literature
review), and it is known that its growth rate is at least exponential,
unless the interpolation points are asymptotically distributed uniformly on
the complex unit circle centered at the origin~\cite{pan}.  As a result, the
computed monomial coefficients are generally highly inaccurate when the
dimensionality of the Vandermonde matrix is not small.  For this reason,
other better-conditioned bases are often used for polynomial
interpolation \cite{nick,arnoldi}. On the other hand, it has long been
observed that polynomial interpolation in the monomial basis produces highly
accurate approximations for many smooth functions (see, for example,
\cite{heath,helsing}).  This is because computing the monomial coefficients
inaccurately does not imply that the resulting interpolating polynomial is
necessarily a poor approximation, since it is
the backward error $\norm{V^{(N)}\hat a^{(N)}-f^{(N)}}_2$ of the numerical
solution $\hat a^{(N)}$ to the Vandermonde system~$V^{(N)}a^{(N)}=f^{(N)}$
that determines the approximation accuracy, and $\norm{V^{(N)}\hat
a^{(N)}-f^{(N)}}_2$ can be small even when the condition number
$\kappa(V^{(N)})$ is large.  It is easy to show that $\norm{V^{(N)}\hat
a^{(N)}-f^{(N)}}_2\lesssim \mach\cdot \norm{a^{(N)}}_2$, where $\mach$
denotes machine epsilon, from which one can show that the monomial
approximation error is bounded by the sum of the polynomial interpolation
error and the extra error term $\mach \cdot \norm{a^{(N)}}_2$.  In this
paper, we establish an upper bound for $\norm{a^{(N)}}_2$ that depends
on the polynomial interpolation errors
$\{\norm{F-P_n}_{\li}\}_{n=0,1,\dots,N}$.  Based on this bound, we find that
this extra error term is generally smaller than the polynomial interpolation
error, provided that the order of approximation is no larger than the
maximum order allowed by the constraint $\kappa(V^{(N)}) \lesssim
\frac{1}{u}$ (see Section \ref{sec:good}). This finding elucidates the
unexpectedly good performance of polynomial interpolation using the
monomial basis in practical applications.  In the rare situations
where the monomial basis results in a less accurate interpolant, we show
that such instances can be easily identified and corrected a posteriori,
leading to a robust algorithm for piecewise polynomial interpolation over
simply connected compact regions in the complex plane using the monomial
basis, with no extra error and with almost no extra cost (see Section
\ref{sec:limit}). As a by-product of our analysis, we derive a tight upper
bound for the condition number of any Vandermonde matrix (see Theorem
\ref{thm:cond2}), which includes several previously established results
\cite{walter2, walter1} as special cases. Finally, we present applications
where using the monomial basis for interpolation is particularly
advantageous (see Section \ref{sec:app}).

While not discussed in this paper, our results can be easily generalized to
higher dimensions and to polynomial approximants constructed using
least-square fitting. In \cite{zs2}, we study bivariate polynomial
interpolation in the monomial basis over a (possibly curved) triangle, and
demonstrate that similarly accurate approximations can be constructed for
orders up to 20, regardless of the triangle's aspect ratio.

\section{Acknowledgements}
We are deeply grateful to James Bremer, Daan Huybrechs, Andreas Kl\"oeckner,
Adam Morgan, Nick Trefethen, and the anonymous referees for their valuable
advice and insightful discussions.


\begin{thebibliography}{99}



\bibitem{navier}
af Klinteberg, L., Askham, T., Kropinski, M. C.: {A Fast Integral
Equation Method for the Two-Dimensional Navier-Stokes Equations}. J.
Comput. Phys., \textbf{409}, 109353 (2020)


\bibitem{barnett}
af Klinteberg, L.,  Barnett, A. H.:  {Accurate Quadrature of Nearly
Singular Line Integrals in Two and Three Dimensions by Singularity
Swapping}. BIT Numer. Math. \textbf{61}.1, 83--118  (2021)


\bibitem{frame}
Adcock, B., Huybrechs, D.: {Frames and Numerical Approximation}.
SIAM Rev. \textbf{61}(3), 443--473 (2019)

\bibitem{frame2}
Adcock, B., Huybrechs, D.: {Frames and Numerical Approximation II:
Generalized Sampling}. J. Fourier Anal. Appl. \textbf{26}(6),  1--34 (2020)


  

\bibitem{root}
Aurentz, J. L., Mach, T., Vandebril, R., Watkins, D. S.: {Fast and Backward
Stable Computation of Roots of Polynomials}. SIAM J. Matrix Anal.  Appl.
\textbf{36}(3), 942--973  (2015)


\bibitem{barnett2}
Barnett, A. H., Betcke, T.: {Stability and Convergence of the
Method of Fundamental Solutions for Helmholtz Problems on Analytic Domains}.
J. Comput. Phys. \textbf{227}(14), 7003--7026  (2008)

\bibitem{beck_thesis}
Beckermann, B.: {On the numerical condition of polynomial bases: estimates
for the condition number of Vandermonde, Krylov and Hankel matrices}. PhD
dissertation, Verlag nicht ermittelbar (1997)


\bibitem{beck}
Beckermann, B.: {The Condition Number of Real Vandermonde, Krylov and
Positive Definite Hankel Matrices}. Numer. Math. \textbf{85}(4), 553--577 (2000)


\bibitem{bary}
Berrut, J.-P., Trefethen, L. N.: {Barycentric Lagrange
Interpolation}. SIAM Rev. \textbf{46}(3), 501--517 (2004)

\bibitem{bjorck}
Bj\"orck, \AA., Pereyra, V.: {Solution of Vandermonde Systems of
Equations}. Math. Comput. \textbf{24}(112), 893--903 (1970)






\bibitem{arnoldi}
Brubeck, P. D., Nakatsukasa Y., Trefethen, L. N.: {Vandermonde
with Arnoldi}, SIAM Rev. \textbf{63}(2), 405--415 (2021)



\bibitem{zeller}
Ehlich, H., Zeller, K.: {Auswertung der Normen von Interpolationsoperatoren}.
Math. Ann. \textbf{164}, 105--112 (1966)


\bibitem{estrin} 
Estrin, G.: {Organization of Computer Systems: The Fixed
plus Variable Structure Computer}. Proc. Western Joint Computer Conference,
33--40 (1960)

\bibitem{walter2}
Gautschi, W.: {The Condition of Polynomials in Power Form}. Math.
Comput. \textbf{33}(145), 343--352 (1979)

\bibitem{walter1}
Gautschi, W.: {How (Un)stable are Vandermonde Systems?}. Asymptot.  Comput.
Anal. \textbf{124}, 193--210 (1990)


\bibitem{traub}
Gohberg, I., Olshevsky, V.: {The Fast Generalized Parker-Traub
Algorithm for Inversion of Vandermonde and Related Matrices}. J. Complex.
\textbf{13}(2), 208--234 (1997)

\bibitem{heath}
Heath, M. T.: {Scientific Computing: An Introductory Survey, Revised
Second Edition}. SIAM (2018)





\bibitem{helsing}
Helsing, J., Ojala, R.: {On the Evaluation of Layer Potentials Close to Their
Sources}. J. Comput. Phys. \textbf{227}(5),  2899--2921 (2008)




\bibitem{jiang}
Helsing, J., Jiang, S.: {On Integral Equation Methods for the
First Dirichlet Problem of the Biharmonic and Modified Biharmonic Equations
in NonSmooth Domains}. SIAM J. Sci. Comput. \textbf{40}(4), A2609--2630 (2018)


\bibitem{higham}
Higham, N. J.: {Error Analysis of the Bj\"orck-Pereyra Algorithms for
Solving Vandermonde Systems}, Numer. Math. 50.5, 613--632 (1987).

\bibitem{higham_text}
Higham, N. J.: Accuracy and Stability of Numerical Algorithms, SIAM (2002).

\bibitem{higham3}
Higham, N. J.: {The Numerical Stability of Barycentric Lagrange
Interpolation}. IMA J. Numer. Anal. \textbf{24}(4),  547--556  (2004)


\bibitem{iserles}
Iserles, A., N\o{}rsett, S. P.: {Efficient Quadrature of Highly
Oscillatory Integrals Using Derivatives}. Proc. Math. Phys. Eng. \textbf{461}(2057),
1383--1399 (2005)

\bibitem{iserles2}
Iserles, A., N\o{}rsett, S. P., Olver, S.: {Highly Oscillatory Quadrature: The
Story So Far}. In: Proceedings of ENUMATH, Santiago de Compostela (2006).
Springer, Berlin,  97--118 (2006)






\bibitem{tornberg}
Ojala, R., Tornberg, A.-K.:  {An Accurate Integral Equation Method for
Simulating Multi-Phase Stokes Flow}. J. Comp. Phys. \textbf{298}, 145--160
(2015)


\bibitem{pan}
Pan, V. Y.: {How Bad are Vandermonde Matrices?}. SIAM J. Matrix Anal.
Appl. \textbf{37}(2), 676--694 (2016)

\bibitem{bbcm}
Pommerenke, C.: {Boundary Behaviour of Conformal Maps}. Vol. 299. Springer
Science \& Business Media (2013).



\bibitem{fejar1}
Reichel, L.: On Polynomial Approximation in the Complex Plane with
Application to Conformal Mapping. Math. Comput., \textbf{44}(170), 425--433
(1985).



\bibitem{bobbie}
Wu, B., Zhu, H., Barnett, A., Veerapaneni S.: {Solution of Stokes
Flow in Complex Nonsmooth 2D Geometries via a Linear-Scaling High-Order
Adaptive Integral Equation Scheme}. J. Comput. Phys. \textbf{410}, pp. 109361 (2020)




\bibitem{zs2}
Shen, Z, Serkh, K.:  {Rapid Evaluation of Newtonian Potentials on Planar
Domains}. SIAM J. Sci. Comput. \textbf{46}(1), A609--628 (2024)


\bibitem{stein}
Stein, D. B., Barnett, A. H.: Quadrature by Fundamental Solutions:
Kernel-Independent Layer Potential Evaluation for Large Collections of
Simple Objects. Adv. Comput. Math. \textbf{48}(60) (2022)




\bibitem{nick}
Trefethen, L. N.: Approximation Theory and Approximation Practice.
SIAM (2019)

\bibitem{nick2}
Trefethen, L. N.: {Polynomial and Rational Convergence Rates for Laplace
Problems on Planar Domains}. Proc. Roy. Soc. A, \textbf{480}(2295)
(2024)



\bibitem{nick_linalg}
Trefethen, L. N., Bau, D.: Numerical Linear Algebra.  SIAM (1997)




\bibitem{walsh}
Walsh, J. L.: Interpolation and Approximation by Rational Functions
in the Complex Domain. vol. 20. AMS, Philadelphia (1935)

\bibitem{hanwen}
Zhang, H., Rokhlin V.: {Finding Roots of Complex Analytic Functions via
Generalized Colleague Matrices}. Adv. Comput. Math. \textbf{50}(4), 71
(2024)


\bibitem{mhz}
Zhao, M, Serkh, K.:  {On the Approximation of Singular Functions by Series
of Non-integer Powers}. arXiv:2308.10439v2 (2023)


\bibitem{afejar}
Zhong, L. F., Zhu, L. Y.: The Marcinkiewicz-Zygmund Inequality on a
Smooth Simple Arc. J. Approx. Theory \textbf{83}(1), 65--83 (1995)

\end{thebibliography}
\end{document}